\documentclass[a4paper,11pt]{article}

\usepackage{geometry}
\usepackage{latexsym}
\usepackage{mathrsfs}
\usepackage{amsfonts}
\usepackage{amssymb}
\usepackage{subfigure}    
\usepackage{graphicx}
\usepackage{epstopdf}
\usepackage{float}
\usepackage{multirow}
\usepackage{booktabs}
\usepackage{bm}
\usepackage{etoolbox}
\usepackage{amsmath}
\usepackage{amsthm}
\usepackage{color}
\usepackage[subnum]{cases}
\usepackage{rotating}
\usepackage{enumitem}
\usepackage{accents}
\usepackage{algorithm}
\usepackage{algorithmic}
\usepackage[title]{appendix}
\usepackage{mathtools}
\usepackage{caption}
\usepackage{prettyref}
\usepackage{tabularx}
\usepackage{stmaryrd}
\usepackage{url}
\usepackage{listings}
\usepackage{tcolorbox}
\usepackage{titlefoot}
\usepackage{makecell}
\usepackage{pifont}

\makeatletter
\@addtoreset{equation}{section}
\makeatother

\makeatletter
\def \wideubar{\underaccent{{\cc@style\underline{\mskip15mu}}}}
\def \widebar{\accentset{{\cc@style\underline{\mskip10mu}}}}
\makeatother

\definecolor{blue}{rgb}{0,0,0.9}
\definecolor{red}{rgb}{0.9,0,0}
\definecolor{green}{rgb}{0,0.9,0}
\definecolor{lightgreen}{rgb}{0.1,0.5,0.1}
\newcommand{\blue}[1]{\begin{color}{blue}#1\end{color}}

\usepackage[colorlinks=true,breaklinks=true,bookmarks=true,urlcolor=blue,
     citecolor=lightgreen,linkcolor=lightgreen,bookmarksopen=false,draft=false]{hyperref}

\graphicspath{{figures/}}

\begin{document}

\newtheorem{property}{Property}[section]
\newtheorem{proposition}{Proposition}[section]
\newtheorem{append}{Appendix}[section]
\newtheorem{definition}{Definition}[section]
\newtheorem{lemma}{Lemma}[section]
\newtheorem{corollary}{Corollary}[section]
\newtheorem{theorem}{Theorem}[section]
\newtheorem{remark}{Remark}[section]
\newtheorem{problem}{Problem}[section]
\newtheorem{example}{Example}[section]
\newtheorem{assumption}{Assumption}
\renewcommand*{\theassumption}{\Alph{assumption}}

\title{ripALM: A Relative-Type Inexact Proximal Augmented Lagrangian Method for Linearly Constrained Convex Optimization}\unmarkedfntext{\hspace{-1mm}The first two authors contributed equally.}

\author{
Jiayi Zhu\thanks{School of Computer Science and Engineering, Sun Yat-sen University ({\tt zhujy86@mail2.sysu.edu.cn}).}, ~~
Ling Liang\thanks{Department of Mathematics, University of Tennessee, Knoxville ({\tt liang.ling@u.nus.edu}).},~~
Lei Yang\thanks{(Corresponding author) School of Computer Science and Engineering, and Guangdong Province Key Laboratory of Computational Science, Sun Yat-sen University ({\tt yanglei39@mail.sysu.edu.cn}).},~~
Kim-Chuan Toh\thanks{Department of Mathematics, and Institute of Operations Research and Analytics, National University of Singapore ({\tt mattohkc@nus.edu.sg}).}
}

\date{\today}

\maketitle

\begin{abstract}
Inexact proximal augmented Lagrangian methods (ipALMs) have been widely used for solving linearly constrained convex optimization problems, owing to their strong theoretical guarantees and excellent numerical performance. In practice, however, existing ipALMs typically employ Rockafellar-type absolute error criteria for solving the subproblems, which require delicate problem-dependent tuning of error-tolerance sequences. In this paper, we propose ripALM, a relative-type ipALM whose subproblem error criterion has only a \textit{single} tolerance parameter in $[0,1)$. This makes the method simpler to implement and less sensitive to parameter tuning in practice. On the other hand, the use of such a relative-type error criterion renders the convergence of our ripALM beyond the scope of the convergence theory of existing ipALMs. To address this gap, we develop a new analysis framework under which ripALM is shown to admit desirable global convergence properties and it achieves an asymptotic (super)linear convergence rate under a standard error bound condition. While there exist other relative-type inexact pALMs, to ensure convergence, they require additional correct steps that generally impede the convergence speed. To the best of our knowledge, ripALM is the first relative-type inexact version of the vanilla pALM that avoids both summable tolerance parameter sequences and correction steps, while retaining rigorous convergence guarantees. Numerical experiments on quadratically regularized optimal transport and basis pursuit denoising problems demonstrate the effectiveness and robustness of our proposed method.

\vspace{5mm}

\end{abstract}


\section{Introduction}\label{sec:introduction}

Linearly constrained convex optimization is a central topic in optimization and arises in numerous applications in machine learning, data science, engineering, and operations research. In this paper, we consider the following classical model:
\begin{equation}\label{eq:mainprob}
\min_{\bm{x}\in\mathbb{R}^N} \quad f(\bm{x}), \quad \mathrm{s.t.} \quad A\bm{x} = \bm{b},
\end{equation}
where $f:\mathbb{R}^{N}\to\mathbb{R}\cup\{+\infty\}$ is a (possibly nonsmooth) proper closed convex function, $A\in \mathbb{R}^{M\times N}$ and $\bm{b}\in\mathbb{R}^{M}$ are given data. Notably, since the smoothness of the objective function $f$ is not required, problem \eqref{eq:mainprob} actually encompasses a broader class of application problems than it may appear at first glance. For instance, by incorporating indicator functions of suitable closed convex sets into the objective, problem \eqref{eq:mainprob} can accommodate more general convex constraints, provided that the resulting proximal mapping of the objective remains tractable; see two illustrative examples in Section \ref{sec-applications}.

Augmented Lagrangian methods (ALMs) and their variants are among the most widely used and effective methods for solving convex constrained optimization problems \cite{hestenes1969multiplier,powell1969method}. They can be applied either to the primal problem \eqref{eq:mainprob} or to its associated dual problem, which, up to a sign change, can be written as
\begin{equation}\label{eq:maindual}
\min\limits_{\bm{y}\in\mathbb{R}^M} \quad f^*({A}^{\top}\bm{y}) - \bm{b}^{\top}\bm{y},
\end{equation}
where $f^*:\mathbb{R}^N\to\mathbb{R}\cup\{+\infty\}$ denotes the convex conjugate of $f$, and $A^\top$ is the transpose of $A$. The choice between primal-based and dual-based ALMs depends on the structure of the problem. In particular, when $f$ is nonsmooth but its proximal mapping is easy to compute, a dual-based approach is often preferable, because the corresponding ALM subproblem amounts to minimizing a continuously differentiable objective function and is therefore typically more tractable numerically. For this reason, we focus on the dual-based framework in this paper. Nevertheless, the analysis developed below applies equally to the primal-based setting.

More specifically, for a given penalty parameter $\sigma>0$, the augmented Lagrangian function associated with the dual problem \eqref{eq:maindual} is given by
\begin{equation*}
\mathcal{L}_\sigma^{\rm dual}(\bm{y}, \bm{x})
=
-\bm{b}^{\top}\bm{y}
+\frac{1}{2\sigma}\big\|\bm{x}+\sigma A^{\top}\bm{y}\big\|^2
-\frac{1}{2\sigma}\|\bm{x}\|^2
-\mathtt{M}_{\sigma f}\big(\bm{x}+\sigma A^{\top}\bm{y}\big),
\quad
(\bm{x},\bm{y})\in\mathbb{R}^N\times\mathbb{R}^M,
\end{equation*}
where $\mathtt{M}_{\sigma f}(\bm{x}) := \min_{\bm{z}} \left\{ f(\bm{z})+\frac{1}{2\sigma}\|\bm{z}-\bm{x}\|^2 \right\}$ denotes the Moreau envelope of $\sigma f$ at $\bm{x}$; see Section~\ref{sec:ripALM} for the detailed derivation. Then, given a sequence of penalty parameters $\{\sigma_k\}\subseteq\mathbb{R}_{++}$ and an initial point $\bm{x}^0\in\mathbb{R}^N$, the dual-based ALM reads as follows:
\begin{equation}\label{dualALMscheme}
\begin{cases}
\bm{y}^{k+1}\in \arg\min\limits_{\bm{y}\in\mathbb{R}^M}
\left\{
\mathcal{L}_{\sigma_k}^{\rm dual}(\bm{y},\bm{x}^k)
\right\}, \\[5pt]
\bm{x}^{k+1}
=
\bm{x}^k+\sigma_k\nabla_{\bm{x}}\mathcal{L}_{\sigma_k}^{\rm dual}(\bm{y}^{k+1},\bm{x}^k)
=
\mathtt{prox}_{\sigma_k f}\big(\bm{x}^k+\sigma_k A^\top \bm{y}^{k+1}\big).
\end{cases}
\end{equation}

The applicability of the augmented Lagrangian method (ALM) has expanded substantially since Rockafellar's seminal works \cite{r1976augmented,r1976monotone}, which established its deep connection with the proximal point algorithm (PPA). This connection has led to a rich convergence theory for the ALM, including asymptotic (super)linear convergence under suitable error bound conditions; see, for example, \cite{cui2019r,luque1984asymptotic} and the references therein. Beyond its theoretical appeal, the ALM and its variants have demonstrated remarkable practical success in a wide range of applications, including conic programming \cite{li2018qsdpnal,lst2020efficient,liang2021inexact,zst2010newton-cg}, statistical optimization \cite{lst2018highly,lin2019efficient,zhang2020efficient}, machine learning \cite{sun2021convex,yan2020efficient}, and image and signal processing \cite{li2013efficient,llm2019nonergodic}. These developments also show that the effectiveness of the ALM depends critically not only on the outer augmented Lagrangian framework itself, but also on how its subproblems are solved and how the associated errors are controlled.

To solve ALM subproblems efficiently, second-order methods, such as Newton-type methods \cite{qi1993nonsmooth}, are particularly attractive because of their fast local convergence. In favorable regimes, only a few Newton-type iterations are often sufficient to obtain highly accurate subproblem solutions. However, their reliable performance typically depends on regularity conditions that ensure the nonsingularity of the (generalized) Hessian of the ALM subproblem objective at a solution. Such conditions are generally difficult to verify in advance and may even fail in practice, thereby limiting the applicability of second-order methods. A standard remedy is to add a proximal term to the ALM subproblem, which gives rise to the proximal augmented Lagrangian method (pALM)\footnote{The pALM is also known as the proximal method of multipliers; see \cite{r1976augmented}.}. Specifically, given $\{\tau_k\}\subseteq\mathbb{R}_{++}$, $\{\sigma_k\}\subseteq\mathbb{R}_{++}$, and $(\bm{x}^0,\bm{y}^0)\in\mathbb{R}^N\times\mathbb{R}^M$, the dual-based pALM iteratively performs
\begin{equation}\label{pALMscheme}
\begin{cases}
\bm{y}^{k+1} = \arg\min\limits_{\bm{y}\in\mathbb{R}^M}
\Big\{\mathcal{L}_{\sigma_k}^{\rm dual}(\bm{y}, \bm{x}^k)
+ \cfrac{\tau_k}{2\sigma_k}\|\bm{y}-\bm{y}^k\|^2\Big\}, \\[8pt]
\bm{x}^{k+1} = \bm{x}^k + \sigma_k \nabla_{\bm{x}}\mathcal{L}_{\sigma_k}^{\rm dual}(\bm{y}^{k+1}, \bm{x}^k)
= \mathtt{prox}_{\sigma_k f}\big(\bm{x}^k+\sigma_k A^\top \bm{y}^{k+1}\big).
\end{cases}
\end{equation}
By incorporating this proximal term, the pALM improves the conditioning of the subproblem and facilitates the application of second-order methods as subsolvers. Consequently, it often exhibits more robust numerical performance in practice compared to the vanilla ALM. We refer the reader to \cite{hermans2019qpALM,lst2020asymptotically,liang2022qppal,lin2019efficient,lst2022augmented,pgk2024efficient,ylct2024corrected} for successful applications of the pALM.

To make the pALM truly implementable and practical, one must allow its subproblems to be solved approximately under an error criterion that is both computationally verifiable and able to preserve desirable convergence properties of the outer loop. To this end, an \textit{absolute-type} error criterion was introduced by Rockafellar for the pALM in \cite{r1976augmented} and has since been widely adopted in the literature; see, e.g., \cite{lst2020asymptotically,liang2022qppal,lin2019efficient,lst2022augmented,pgk2024efficient}. The term ``absolute" means that the accuracy required for solving each pALM subproblem is controlled by a prescribed sequence of tolerance parameters. Such a sequence typically needs to be summable and must be chosen carefully to avoid being either overly conservative or excessively aggressive. Consequently, achieving satisfactory numerical performance often requires delicate parameter tuning, which can be laborious and complicate practical implementation.

An appealing alternative is to employ a \textit{relative-type} error criterion. The term ``relative" means that the errors in solving each subproblem are controlled by certain computable quantities related to the current progress of the algorithm, rather than by a prescribed sequence of tolerance parameters. Such a criterion often involves only a \textit{single} tolerance parameter, and can therefore simplify parameter tuning and improve practical implementability. The use of relative-type error criteria can be traced back to the seminal works of Solodov and Svaiter \cite{ss1999hybridap,ss1999hybridpr,ss2000inexact} on relative-type inexact variants of PPA. Since then, this idea has influenced the development of relative-type inexact variants of numerous algorithms, including the variable metric PPA \cite{pls2008class}, ALM \cite{es2013practical}, ADMM \cite{ey2018relative}, FISTA \cite{bgk2023fista}, and the Bregman-type methods \cite{yht2024inexact,yt2023inexact}.

In contrast, relative-type inexact variants of the pALM remain much less explored, with only a few existing works; see \cite{es2010proximal,hss2004some,ylct2024corrected}. More importantly, these existing methods are not genuine inexact versions of the vanilla pALM. They are all derived by applying a (variable metric) hybrid proximal extragradient method \cite{pls2008class,ss1999hybridap,ss1999hybridpr} to a suitable primal-dual solution mapping of the underlying convex constrained problem, and therefore require an additional \textit{correction step} to guarantee convergence. As a result, the methods developed in \cite{es2010proximal,hss2004some,ylct2024corrected} depart from the vanilla pALM \eqref{pALMscheme}. This distinction is not merely formal. Indeed, our numerical observations indicate that such a correction step tends to slow down practical convergence and increase the overall computational cost. It is therefore natural and important to ask:
\begin{tcolorbox}[colback=black!0!white]
\textit{Can one develop a relative-type inexact version of the vanilla pALM \eqref{pALMscheme} that avoids both summable tolerance parameter sequences and correction steps, while still admitting rigorous convergence guarantees?}
\end{tcolorbox}

The primary objective of this paper is to answer the above question affirmatively. To this end, we propose a \underline{r}elative-type \underline{i}nexact \underline{p}roximal \underline{a}ugmented \underline{L}agrangian \underline{m}ethod (ripALM) for solving the dual problem \eqref{eq:maindual}. The proposed method is a genuine inexact version of the vanilla pALM \eqref{pALMscheme}, whose subproblem error criterion has only a single tolerance parameter in $[0,1)$. This makes ripALM simpler to implement and more robust than existing absolute-type inexact pALMs, as observed from our numerical experiments in Section \ref{sec:numexp}. However, the use of such a relative-type error criterion renders ripALM beyond the scope of the existing convergence theory for inexact pALMs. To overcome this difficulty, we develop a new analysis framework for ripALM. The main contributions of this paper are summarized as follows:
\begin{itemize}
\item We introduce ripALM, which, to the best of our knowledge, is the first relative-type inexact version of the vanilla pALM \eqref{pALMscheme} that avoids both summable tolerance parameter sequences and correction steps.

\item We develop a new analysis framework for ripALM, establish boundedness and global subsequential convergence properties of the generated iterates $\{\bm{y}^k\}$ and $\{\bm{x}^k\}$, and prove the convergence of the sequence $\{\bm{x}^k\}$. Moreover, under a standard error bound condition, we further prove the convergence of the sequence $\{\bm{y}^k\}$ and show that the sequence $\{(\bm{y}^k,\bm{x}^k)\}$ converges asymptotically at a (super)linear rate.

\item We demonstrate the practical effectiveness of ripALM on quadratically regularized optimal transport and basis pursuit denoising problems, where it is highly competitive with existing representative methods.
\end{itemize}

The remaining parts of this paper are organized as follows. Section \ref{sec:ripALM} describes the main algorithmic framework of ripALM, whose convergence analysis is conducted in Section \ref{sec:conver}. Section \ref{sec-applications} showcases how to apply the proposed ripALM for solving QROT and BPDN problems. Numerical experiments are conducted in Section \ref{sec:numexp}. Finally, some concluding remarks are summarized in Section \ref{sec:conclusions}.

\vspace{2mm}
\textbf{Notation.}
We use $\mathbb{R}^n$, $\mathbb{R}_{+}^n$, $\mathbb{R}^{m \times n}$ and $\mathbb{R}_{+}^{m \times n}$ to denote the sets of $n$-dimensional real vectors, $n$-dimensional non-negative vectors, $m \times n$ real matrices and $m \times n$ real non-negative matrices, respectively. For a vector $\bm{x}\in\mathbb{R}^n$, $x_i$ denotes its $i$-th entry, $\|\bm{x}\|$ denotes its Euclidean norm, and $\|\bm{x}\|_H:=\sqrt{\langle\bm{x},\, H\bm{x}\rangle}$ denotes its weighted norm associated with a symmetric positive definite matrix $H\in\mathbb{R}^{n\times n}$. For a matrix $X\in\mathbb{R}^{m\times n}$, ${X_{ij}}$ denotes its $(i,j)$-th entry, $\|X\|_F$ denotes its Frobenius norm, and $\texttt{vec}(X)$ denotes the vectorization of $X$, where $[\texttt{vec}(X)]_{i+(j-1)m}=X_{ij}$ for any $1\leq i\leq m$ and $1\leq j \leq n$. For simplicity, given an integer $n>0$, we use $\bm{1}_{n}\in\mathbb{R}^n$ to denote the $n$-dimensional vector of all ones, and use $I_n$ to denote the $n \times n$ identity matrix.

For an extended-real-valued function $f: \mathbb{R}^{n} \rightarrow [-\infty,\infty]$, we say that it is \textit{proper} if $f(\bm{x}) > -\infty$ for all $\bm{x}\in\mathbb{R}^{n}$ and its effective domain ${\rm dom}\,f:=\{\bm{x} \in \mathbb{R}^{n} : f(\bm{x})<\infty\}$ is nonempty. A proper function $f$ is said to be closed if it is lower semicontinuous. Assume that $f: \mathbb{R}^{n} \rightarrow (-\infty,\infty]$ is a proper and closed convex function, the subdifferential of $f$ at $\bm{x}\in{\rm dom}\,f$ is defined by $\partial f(\bm{x}):=\big\{\bm{d}\in\mathbb{R}^n: f(\bm{y}) \geq f(\bm{x}) + \langle \bm{d}, \,\bm{y}-\bm{x}\rangle, ~\forall\,\bm{y}\in\mathbb{R}^n\big\}$ and its conjugate function $f^*: \mathbb{R}^{n} \rightarrow (-\infty,\infty]$ is defined by $f^*(\bm{y}):=\sup\big\{\langle \bm{y},\,\bm{x}\rangle-f(\bm{x}) : \bm{x}\in\mathbb{R}^n\big\}$. For any $\nu>0$, the Moreau envelope of $\nu f$ at $\bm{x}$ is defined by $\mathtt{M}_{\nu f}(\bm{x}) := \min_{\bm{y}} \big\{f(\bm{y}) + \frac{1}{2\nu}\|\bm{y} - \bm{x}\|^2\big\}$, and the proximal mapping of $\nu f$ at $\bm{x}$ is defined by $\mathtt{prox}_{\nu f}(\bm{x}) := \arg\min_{\bm{y}} \big\{f(\bm{y}) + \frac{1}{2\nu}\|\bm{y} - \bm{x}\|^2\big\}$.

Let $\mathcal{S}$ be a closed convex subset of $\mathbb{R}^n$. Its indicator function $\delta_{\mathcal{S}}$ is defined by $\delta_{\mathcal{S}}(\bm{x})=0$ if $\bm{x}\in\mathcal{S}$ and $\delta_{\mathcal{S}}(\bm{x})=+\infty$ otherwise. Moreover, we denote the weighted distance of $\bm{x}\in\mathbb{R}^n$ to $\mathcal{S}$ by $\mathrm{dist}_{H}(\bm{x},\mathcal{S}):=\inf_{\bm{y}\in\mathcal{S}}\|\bm{x}-\bm{y}\|_H$ associated with a symmetric positive definite matrix $H$. When $H$ is the identity matrix, we omit $H$ in the notation and simply use $\mathrm{dist}(\bm{x},\mathcal{S})$ to denote the Euclidean distance of $\bm{x}\in\mathbb{R}^n$ to $\mathcal{S}$. Moreover, we use $\Pi_{\mathcal{S}}(\bm{x})$ to denote the projection of $\bm{x}$ onto $\mathcal{S}$.

\section{A relative-type inexact pALM}\label{sec:ripALM}

In this section, we focus on developing a relative-type inexact proximal augmented Lagrangian method (ripALM) to solve the dual problem \eqref{eq:maindual}. The algorithmic framework is developed based on the parametric convex duality framework (see, for example, \cite{r1970convex,r1974conjugate} and \cite[Chapter 11]{rw1998variational}).

We first identify problem \eqref{eq:maindual} with the following problem
\begin{equation}\label{eq:para-proorg}
\min\limits_{\bm{y}\in\mathbb{R}^M} \quad G(\bm{y}, \bm{0}),
\end{equation}
where $G:\mathbb{R}^{M}\times\mathbb{R}^{N}\to(-\infty,+\infty]$ is defined by
\begin{equation}\label{eq:para-objpri}
G(\bm{y},\bm{\xi}) := f^*({A}^{\top}\bm{y}+\bm{\xi}) - \bm{b}^{\top}\bm{y}.
\end{equation}
Note that $G$ is proper closed convex since $f^{*}$ is proper closed convex. Then, the (ordinary) Lagrangian function of problem \eqref{eq:maindual} can be defined by taking the concave conjugate of $G$ with respect to its second argument (see \cite[Definition 11.45]{rw1998variational}), that is,
\begin{equation}\label{eq:ripALM_lag}
\ell(\bm{y},\bm{x}) := \inf\limits_{\bm{\xi}\in\mathbb{R}^N}
\left\{G(\bm{y},\bm{\xi}) - \langle\bm{x}, \,\bm{\xi}\rangle\right\}
= - \bm{b}^{\top}\bm{y} + \langle\bm{x},\,{A}^{\top}\bm{y}\rangle - f(\bm{x}).
\end{equation}
Clearly, $\ell$ is convex in its first argument and concave in the second argument.  For a given penalty parameter $\sigma>0$, the augmented Lagrangian function of problem \eqref{eq:maindual} is defined by (see \cite[Example 11.57]{rw1998variational})
\begin{equation*}
\begin{aligned}
\mathcal{L}_{\sigma}(\bm{y},\,\bm{x})
:=& \sup\limits_{\bm{s}\in\mathbb{R}^N} \left\{\ell(\bm{y},\bm{s})-\frac{1}{2\sigma}\|\bm{s}-\bm{x}\|^2\right\} \\
=& -\bm{b}^{\top}\bm{y} + \frac{1}{2\sigma}\big\|\bm{x}+\sigma{A}^{\top}\bm{y}\big\|^2 - \frac{1}{2\sigma}\|\bm{x}\|^2 - \mathtt{M}_{\sigma f}\big(\bm{x}+\sigma{A}^{\top}\bm{y}\big).
\end{aligned}
\end{equation*}
Here, we assume for the remainder of this paper that $\mathcal{L}_{\sigma}$ denotes the augmented Lagrangian function of the \textit{dual} problem \eqref{eq:maindual}, omitting the superscript ``dual" as in \eqref{dualALMscheme}.

From the property of the Moreau envelope (see \cite[Proposition 12.30]{bc2017convex}), we know that $\mathcal{L}_{\sigma}$ is continuously differentiable with respect to its first argument and
\begin{equation*}
\nabla_{y}\mathcal{L}_{\sigma}(\bm{y},\bm{x}) = A\mathtt{prox}_{\sigma f}\big(\bm{x}+\sigma A^{\top}\bm{y}\big) - \bm{b}.
\end{equation*}
With the above preparations, we are now ready to present our ripALM for solving problem \eqref{eq:maindual} in Algorithm \ref{algo:ripALM}.

\begin{algorithm}[ht]
\caption{A relative-type inexact proximal augmented Lagrangian method (ripALM) for solving problem \eqref{eq:maindual}}\label{algo:ripALM}
\textbf{Input:} $\rho\in[0,1)$, $\{\sigma_{k}\}_{k=0}^{\infty}$ such that $\inf_{k\geq0}\{\sigma_{k}\}>0$, and $\{\tau_{k}\}_{k=0}^{\infty}$ such that $0<\inf_{k\geq0}\{\tau_{k}\}\leq\sup_{k\geq0}\{\tau_{k}\}<\infty$. Choose $\bm{y}^0,\,\bm{w}^0\in\mathbb{R}^M$ and $\bm{x}^0\in\mathbb{R}^N$ arbitrarily. Set $k=0$.  \vspace{1mm} \\
\textbf{while} the termination criterion is not met, \textbf{do} \vspace{-2mm}
\begin{itemize}[leftmargin=2.2cm]
\item[\textbf{Step 1}.] Approximately solve the subproblem:
		\begin{equation}\label{ripALM-subpro}
		\min\limits_{\bm{y}\in\mathbb{R}^{M}}~~
        \mathcal{L}_{\sigma_{k}}(\bm{y},\,\bm{x}^k)
        + \frac{\tau_{k}}{2\sigma_{k}}\big\|\bm{y}-\bm{y}^k\big\|^2
		\end{equation}
		to find $\bm{y}^{k+1}\in\mathbb{R}^M$ such that
		\begin{equation}\label{ripALM-inexcond}
		\begin{aligned}
		&\; 2\big|\langle\bm{w}^k-\bm{y}^{k+1},
        \,\sigma_{k}\Delta^{k+1}\rangle\big|
        + \big\|\sigma_{k}\Delta^{k+1}\big\|^2 \\[3pt]
	    \leq&\; \rho\left(
        \big\|\mathtt{prox}_{\sigma_{k} f}\big(\bm{x}^k+\sigma_{k}{A}^{\top}\bm{y}^{k+1}\big)-\bm{x}^{k}\big\|^2
        + \tau_{k}\big\|\bm{y}^{k+1}-\bm{y}^k\big\|^2\right),
		\end{aligned}
		\end{equation}
        where $\Delta^{k+1}
        := \nabla_{y} \mathcal{L}_{\sigma_{k}}(\bm{y}^{k+1},\bm{x}^k) + \tau_{k}\sigma_{k}^{-1}(\bm{y}^{k+1}-\bm{y}^k)$.

\vspace{1mm}	
\item[\textbf{Step 2}.] Update
		\begin{equation}\label{ripALM_xwupdate}
    \bm{x}^{k+1} =
        \mathtt{prox}_{\sigma_{k} f}\big(\bm{x}^k+\sigma_{k}{A}^{\top}\bm{y}^{k+1}\big), \quad
		\bm{w}^{k+1} = \bm{w}^k - \sigma_{k}\Delta^{k+1}.
        \end{equation}
		
\item[\textbf{Step 3}.]Set $k = k+1$ and go to \textbf{Step 1}. \vspace{-2mm}
		
\end{itemize}
\textbf{end while}  \\
\textbf{Output}: $(\bm{y}^k,\bm{x}^k)$. \vspace{0.5mm}
\end{algorithm}

In line with pALM-type methods, at each iteration, our ripALM in Algorithm \ref{algo:ripALM} \textit{approximately} minimizes the sum of the augmented Lagrangian function $\mathcal{L}_{\sigma_{k}}(\cdot,\,\bm{x}^k)$ and a proximal term $\frac{\tau_{k}}{2\sigma_{k}}\big\|\cdot-\bm{y}^k\big\|^2$ under the error criterion \eqref{ripALM-inexcond}, followed by the updates of the multiplier $\bm{x}^k$ and the auxiliary error variable $\bm{w}^k$. Next, we show that the error criterion \eqref{ripALM-inexcond} is achievable at each iteration, provided that a convergent subsolver is employed. Specifically, the subproblem \eqref{ripALM-subpro} admits a unique solution $\bm{y}^{k,*}$ due to the strong convexity, which satisfies
\begin{equation*}
\nabla_{y}\mathcal{L}_{\sigma_{k}}(\bm{y}^{k,*},\bm{x}^k) + \tau_{k}\sigma_{k}^{-1}(\bm{y}^{k,*}-\bm{y}^k)=\bm{0}.
\end{equation*}
Then, a suitable subsolver (e.g., a gradient-type or Newton-type method) for solving the subproblem \eqref{ripALM-subpro} can generate a sequence $\{\bm{y}^{k,t}\}$ converging to $\bm{y}^{k,*}$ as $t\to\infty$. Since the objective function in \eqref{ripALM-subpro} is continuously differentiable (by the property of the Moreau envelope \cite{rw1998variational}), it follows that
\begin{equation*}
\Delta^{k,t}:=\nabla_{y}\mathcal{L}_{\sigma_{k}}(\bm{y}^{k,t},\bm{x}^k) + \tau_{k}\sigma_{k}^{-1}(\bm{y}^{k,t}-\bm{y}^k)\to\bm{0}
\quad\mathrm{as}\quad t\to\infty,
\end{equation*}
which further implies that
\begin{equation*}
2\big|\langle\bm{w}^k-\bm{y}^{k,t}, \,\sigma_{k}\Delta^{k,t}\rangle\big|+ \big\|\sigma_{k}\Delta^{k,t}\big\|^2 \to 0
\quad\mathrm{as}\quad t\to\infty.
\end{equation*}
Therefore, the left-hand side of \eqref{ripALM-inexcond} converges to $0$ as $t\to\infty$. On the other hand, if $\bm{y}^k$ is not the unique solution of the subproblem \eqref{ripALM-subpro} with $\tau_{k}\neq0$ (which is the case of interest, as otherwise we can directly proceed to the next iteration), then the term $\tau_k\|\bm{y}^{k,t}-\bm{y}^k\|^2$ and thus the right-hand side of \eqref{ripALM-inexcond} would not approach $0$ as $t\to\infty$. Consequently, in this case, \eqref{ripALM-inexcond} must be satisfied at some iterate $\bm{y}^{k,t}$ after finitely many $t$ inner iterations. Once \eqref{ripALM-inexcond} is satisfied at $\bm{y}^{k,t}$, we set $\bm{y}^{k+1}:=\bm{y}^{k,t}$ and proceed to the next outer iteration.

Compared to the recent semismooth Newton based inexact proximal augmented Lagrangian ({\sc Snipal}) method developed in \cite[Section 3]{lst2020asymptotically}, our ripALM in Algorithm \ref{algo:ripALM} employs a significantly different error criterion \eqref{ripALM-inexcond} for solving the subproblem \eqref{ripALM-subpro}. Specifically, in our context, {\sc Snipal} requires the error term $\Delta^{k+1}$ to satisfy
\begin{equation}\label{eq:inexact_SNIPAL}
\begin{aligned}
&(A) ~~ \big\|\Delta^{k+1}\big\|
\leq \frac{\min\left(\sqrt{\tau_{k}}, \,1\right)}{\sigma_k}\varepsilon_k,
~~ \varepsilon_k\geq0,
~~ \sum_{k=0}^{\infty} \varepsilon_k<\infty,  \\
&(B) ~~ \big\|\Delta^{k+1}\big\|
\leq \frac{\delta_k\min\left(\sqrt{\tau_{k}}, \,1\right)}{\sigma_k}\sqrt{
\big\|\Delta_x^{k+1}\big\|^2 + \tau_{k}\big\|\Delta_y^{k+1}\big\|^2},
~~0 \leq \delta_k<1,
~~\sum_{k=0}^{\infty} \delta_k<\infty,
\end{aligned}
\end{equation}
with $\Delta_x^{k+1}:=\bm{x}^{k+1}-\bm{x}^{k}$ and $\Delta_y^{k+1}:=\bm{y}^{k+1}-\bm{y}^k$, to guarantee an asymptotic (super)linear convergence rate.\footnote{Note that the error criterion (A) in \eqref{eq:inexact_SNIPAL} alone is sufficient for establishing the global convergence of the {\sc Snipal}; see \cite[Section 3]{lst2020asymptotically}.} Both error criteria (A) and (B) in \eqref{eq:inexact_SNIPAL} are of the absolute type, meaning that they require the pre-specification of two summable tolerance parameter sequences, $\{\varepsilon_k\}\subseteq[0,\infty)$ and $\{\delta_k\}\subseteq[0,1)$, to control the error incurred in the inexact minimization of the subproblem. Since there is generally no direct guidance on optimally selecting these tolerance parameters to achieve good convergence efficiency, this absolute-type criterion typically requires careful tuning for the tolerance parameters, which may make  {\sc Snipal} less user-friendly in practical implementations. In contrast, the error criterion \eqref{ripALM-inexcond} used in ripALM is of the relative type, meaning that the error $2\big|\langle\bm{w}^k-\bm{y}^{k+1},\,\sigma_{k}\Delta^{k+1}\rangle\big| + \big\|\sigma_{k}\Delta^{k+1}\big\|^2$ is regulated by a tentative successive difference related to the progress of the algorithm. Moreover, this relative-type criterion only involves a \textit{single} tolerance parameter $\rho\in[0,1)$, thus simplifying the process of parameter tuning in both computation and implementation, as we shall see in Section \ref{sec:numexp}.

Thanks to the advantage of eliminating the need to select an infinite sequence of tolerance parameters, relative-type error criteria have been widely adopted in numerous well-known algorithms (e.g., PPA, ALM, ADMM, FISTA, etc.) to approximately solve subproblems over the past two decades. This trend began with the seminal works of Solodov and Svaiter \cite{ss1999hybridap,ss1999hybridpr,ss2000inexact}, which has since inspired the development of numerous relative-type inexact algorithms; see, for example, \cite{bgk2023fista,es2013practical,ey2018relative,pls2008class,yht2024inexact,ylct2024corrected,yt2023inexact}. Following this line of research, Yang et al. \cite{ylct2024corrected} recently developed a corrected inexact proximal augmented Lagrangian method (ciPALM) for solving a class of group-quadratic regularized optimal transport problems. The error criterion used there can be described as follows: choose $\rho\in[0,1)$, at each iteration, approximately solve the subproblem \eqref{ripALM-subpro} to find a point $\widetilde{\bm{y}}^{k+1}$ such that
\begin{equation}\label{eq:inexact_ciPALM}
\begin{aligned}
\widetilde{\Delta}^{k+1}
:=& ~\nabla_{y} \mathcal{L}_{\sigma_{k}}(\widetilde{\bm{y}}^{k+1},\bm{x}^k)+\tau_{k}\sigma_{k}^{-1}(\widetilde{\bm{y}}^{k+1}-\bm{y}^k),
\\[3pt]
\hspace*{-0.25cm}\big\|{\sigma_k}\widetilde{\Delta}^{k+1}\big\|^2
\leq &~\rho^{2}\min\left({\tau_{k}}, \,1\right)
\left(\left\|\mathtt{prox}_{\sigma_{k} f}\big(\bm{x}^k+\sigma_{k}{A}^{\top}\widetilde{\bm{y}}^{k+1}\big)-\bm{x}^{k}\right\|^2 + \tau_{k}\left\|\widetilde{\bm{y}}^{k+1} - \bm{y}^{k}\right\|^2\right).
\end{aligned}
\end{equation}
When \eqref{eq:inexact_ciPALM} is satisfied, the multiplier is updated as usual: $\bm{x}^{k+1}=\mathtt{prox}_{\sigma_{k} f}\big(\bm{x}^k+\sigma_{k}{A}^{\top}\widetilde{\bm{y}}^{k+1}\big)$. This is then followed by an extra \textit{correction} step: $\bm{y}^{k+1}=\bm{y}^{k}- \tau_{k}^{-1}\sigma_{k}\left(A\bm{x}^{k+1}-\bm{b}\right)$. After these updates, the algorithm proceeds to the next iteration. One can see that the error criterion \eqref{eq:inexact_ciPALM} used in ciPALM is also of the relative type and looks even simpler than \eqref{ripALM-inexcond} used in our ripALM. However, we should point out that ciPALM requires an additional correction step to ensure its convergence, and hence it may only be viewed as a pALM-like algorithm. Moreover, our numerical results indicate that this correction step tends to degrade the practical performance of the pALM framework. This raises the natural question of whether one can design a relative-type error criterion directly for the vanilla pALM framework. Addressing this question serves as the main motivation of this work. To this end, we propose the relative error criterion \eqref{ripALM-inexcond}, where the error variable $\bm{w}^k$ is used solely in the construction of the error criterion \eqref{ripALM-inexcond} and does not directly influence either the objective function in \eqref{ripALM-subpro} or the updating rules of the primal and dual variables. Our numerical results show that the proposed ripALM may offer greater robustness and efficiency; see Section \ref{sec:numexp} for numerical comparisons.

Our relative error criterion \eqref{ripALM-inexcond} is inspired by Eckstein and Silva's practical relative error criterion \cite{es2013practical}, which was developed for the approximate minimization of the subproblems in the vanilla ALM (i.e., without the proximal term $\frac{\tau_{k}}{2\sigma_{k}}\big\|\bm{y}-\bm{y}^k\big\|^2$ in the subproblem \eqref{ripALM-subpro} in our context). Unlike their work, we suggest incorporating $\frac{\tau_{k}}{2\sigma_{k}}\big\|\bm{y}-\bm{y}^k\big\|^2$ in the subproblem \eqref{ripALM-subpro}. This proximal term guarantees the existence and uniqueness of the optimal solution of the strongly convex subproblem \eqref{ripALM-subpro}. Introducing the proximal term also ensures the positive definiteness of the
generalized Hessian/Jacobian of the objective function in \eqref{ripALM-subpro}, thereby facilitating the effective application of the semi-smooth Newton method to the subproblem \eqref{ripALM-subpro}, as shown in Section \ref{sec:app2qrot}. In addition, as we will see later in the next section, our ripALM not only shares the same convergence properties for the sequence $\{\bm{x}^k\}$ as Eckstein and Silva's relative-type inexact ALM (denoted by ES-iALM) developed in \cite{as2016note,es2013practical,zc2020linear}, but also offers additional theoretical advantages: the sequence $\{\bm{y}^k\}$ is guaranteed to be bounded; moreover, under a relatively weaker error bound assumption (than the local upper Lipschitz continuity of $(\partial\ell)^{-1}$ at the origin), the sequence $\{\bm{y}^{k}\}$ is further proved to be convergent, and both primal and dual sequences achieve asymptotic (super)linear convergence rates.

For ease of comparison, Table \ref{tab:comparison_palms} summarizes the main algorithmic and theoretical differences between ripALM and the related methods discussed above. To the best of our knowledge, ripALM is the first relative-type inexact version of the vanilla pALM \eqref{pALMscheme} that avoids both summable tolerance sequences and correction steps.

\begin{table}[ht]
\centering
\caption{A comparison of ES-iALM, {\sc Snipal}, ciPALM, and ripALM for solving problem \eqref{eq:maindual}. In the table, ``Seq. Conv." means sequential convergence; ``\texttt{Abs}." means absolute-type; ``\texttt{Rel}." means relative-type; ``\ding{52}$^E$" indicates convergence under an error bound condition.}
\label{tab:comparison_palms}
\renewcommand{\arraystretch}{1.25}
\centering
\tabcolsep 5.5pt
\scalebox{0.85}{
\begin{tabular}{lcccccccc}
\toprule
\multirow{2}{*}{Methods} &
\multirow{2}{*}{Type} &
\multirow{2}{*}{\makecell{Vanilla \\ (p)ALM}} &
\multirow{2}{*}{\makecell{Proximal \\ Term}} &
\multicolumn{2}{c}{Boundedness} &
\multicolumn{2}{c}{Seq. Conv.}  &
\multirow{2}{*}{\makecell{Asymptotic \\ (super)linear Rate \\ (under Error Bound)}} \\
\cmidrule(l){5-6} \cmidrule(l){7-8}
& & & & $\{\bm{y}^{k}\}$ & $\{\bm{x}^{k}\}$ &
$\{\bm{y}^{k}\}$ & $\{\bm{x}^{k}\}$ &  \\
\midrule
ES-iALM \cite{as2016note,es2013practical,zc2020linear} & \texttt{Rel}. & \ding{52} & \ding{56} & \ding{56} & \ding{52} & \ding{56} & \ding{52} & $\{\bm{x}^{k}\}$ \\
{\sc Snipal} \cite{lst2020asymptotically} & \texttt{Abs}. & \ding{52} & \ding{52} & \ding{52} & \ding{52} & \ding{52} & \ding{52} & $\{(\bm{y}^{k},\bm{x}^{k})\}$ \\
ciPALM \cite{ylct2024corrected} & \texttt{Rel}. & \makecell{\ding{56}} & \ding{52} & \ding{52} & \ding{52} & \ding{52} & \ding{52} & $\{(\bm{y}^{k},\bm{x}^{k})\}$ \\
\textbf{ripALM [this work]} & \texttt{Rel}. & \ding{52} & \ding{52} & \ding{52} & \ding{52} & \ding{52}$^E$ & \ding{52} & $\{(\bm{y}^{k},\bm{x}^{k})\}$ \\
\bottomrule
\end{tabular}}
\end{table}

\section{Convergence analysis}\label{sec:conver}

In this section, we study the convergence properties of the ripALM in Algorithm \ref{algo:ripALM}. Before proceeding, we recall the definition of $G$ from \eqref{eq:para-objpri} and further define $F:\mathbb{R}^{M}\times\mathbb{R}^{N}\rightrightarrows\mathbb{R}^{M}\times\mathbb{R}^{N}$ to be the concave conjugate of $G$:
\begin{equation}\label{eq:para-objdual}
F(\bm{\theta}, \bm{x}) := \inf\limits_{\bm{y}\in\mathbb{R}^{M},\,\bm{\xi}\in\mathbb{R}^{N}} \big\{G(\bm{y},\bm{\xi}) - \langle\bm{\theta},\bm{y}\rangle - \langle\bm{x},\bm{\xi}\rangle\big\},
\end{equation}
which is a closed (upper semicontinuous) concave function. Then, the dual problem of \eqref{eq:para-proorg} is given by
\begin{equation}\label{eq:para-proorgdual}
\max\limits_{\bm{x}\in\mathbb{R}^N} \quad F(\bm{0},\bm{x}),
\end{equation}
which can be rewritten as problem \eqref{eq:mainprob}. Next, let $\partial G:\mathbb{R}^{M}\times\mathbb{R}^{N}\rightrightarrows\mathbb{R}^{M}\times\mathbb{R}^{N}$ and $\partial F:\mathbb{R}^{M}\times\mathbb{R}^{N}\rightrightarrows\mathbb{R}^{M}\times\mathbb{R}^{N}$
denote the subgradient maps of $G$ and $F$, respectively, that is,
\begin{equation*}
\begin{aligned}
&(\bm{\theta},\bm{x})\in\partial G(\bm{y},\bm{\xi}) ~\Leftrightarrow~
G(\bm{y}',\bm{\xi}') \geq G(\bm{y},\bm{\xi}) + \langle\bm{\theta}, \bm{y}'-\bm{y}\rangle + \langle\bm{x},\bm{\xi}'-\bm{\xi}\rangle, ~~\forall (\bm{y}',\bm{\xi}'), \\ 
&(\bm{y},\bm{\xi})\in\partial F(\bm{\theta},\bm{x}) ~\Leftrightarrow~
F(\bm{\theta}',\bm{x}') \leq F(\bm{\theta},\bm{x}) - \langle\bm{y}, \bm{\theta}'-\bm{\theta}\rangle - \langle\bm{\xi},\bm{x}'-\bm{x}\rangle, ~~\forall (\bm{\theta}',\bm{x}'). 
\end{aligned}
\end{equation*}
We also recall the definition of $\ell$ from \eqref{eq:ripALM_lag} and let $\partial\ell:\mathbb{R}^{M}\times\mathbb{R}^{N}\rightrightarrows\mathbb{R}^{M}\times\mathbb{R}^{N}$ denote the subgradient map of $\ell$, that is,
\begin{equation*}
(\bm{\theta},\bm{\xi}) \in \partial \ell(\bm{y},\bm{x}) ~\Leftrightarrow~
\left\{\begin{aligned}
&\ell(\bm{y}',\bm{x}) \geq \ell(\bm{y},\bm{x}) + \langle\bm{\theta}, \bm{y}'-\bm{y}\rangle, ~~\forall\bm{y}'\in\mathbb{R}^M, \\
&\ell(\bm{y},\bm{x}') \leq \ell(\bm{y},\bm{x}) - \langle\bm{\xi}, \bm{x}'-\bm{x}\rangle, ~~\forall\bm{x}'\in\mathbb{R}^N.
\end{aligned}\right.
\end{equation*}
By this definition, one can verify that
\begin{equation*}
\partial\ell(\bm{y},\bm{x}) = \left\{A\bm{x}-\bm{b}\right\}
\times\left\{\bm{v}-A^{\top}\bm{y} \mid \bm{v}\in \partial f(\bm{x})\right\}.
\end{equation*}
Moreover, all the subgradient maps $\partial G$, $\partial F$, and $\partial \ell$ are maximal monotone operators, and satisfy that (see also \cite[equation (23)]{es2013practical}):
\begin{equation}\label{subdiffrelation}
(\bm{\theta},\bm{x}) \in \partial G(\bm{y},\bm{\xi}) ~\Leftrightarrow~
(\bm{\theta},\bm{\xi}) \in \partial \ell(\bm{y},\bm{x}) ~\Leftrightarrow~
(\bm{y},\bm{\xi})\in\partial F(\bm{\theta},\bm{x}).
\end{equation}
If $(\bm{y}^*,\bm{x}^*)\in\mathbb{R}^M\times\mathbb{R}^N$ satisfies $(\bm{0},\bm{0})\in\partial \ell(\bm{y}^*,\bm{x}^*)$, i.e, $(\bm{y}^*,\bm{x}^*)$ solves the following KKT system:
\begin{equation*}
\bm{0} \in \partial f(\bm{x}) - A^{\top}\bm{y},
\quad
A\bm{x} - \bm{b} = \bm{0},
\end{equation*}
then $\bm{y}^*$ solves problem \eqref{eq:para-proorg} (i.e., problem \eqref{eq:maindual}) and $\bm{x}^*$ solves problem \eqref{eq:para-proorgdual} (i.e., problem \eqref{eq:mainprob}). In this case, we call $(\bm{y}^*,\bm{x}^*)$ a \textit{saddle point} of the Lagrangian function $\ell(\bm{y},\bm{x})$. If such a saddle point exists, then the strong duality holds, that is, the optimal values of problems \eqref{eq:para-proorg} and \eqref{eq:para-proorgdual} are finite and equal, i.e., $G(\bm{y}^*,\bm{0}) = F(\bm{0},\bm{x}^*)$. Moreover, the set of saddle points can be written as $\mathcal{Y}^* \times \mathcal{X}^*\subseteq\mathbb{R}^M\times\mathbb{R}^N$, where $\mathcal{Y}^*$ is the solution set of problem \eqref{eq:para-proorg} (i.e., problem \eqref{eq:maindual}) and $\mathcal{X}^*$ is the solution set of problem \eqref{eq:para-proorgdual} (i.e., problem \eqref{eq:mainprob}).

With the above preparations, we are now ready to establish the convergence results of the proposed ripALM in Algorithm \ref{algo:ripALM}. Note that, while our analysis is inspired by that of ES-iALM in \cite{as2016note,es2013practical,zc2020linear}, the inclusion of the proximal term $\frac{\tau_{k}}{2\sigma_{k}}\|\bm{y}-\bm{y}^k\|^2$ indeed introduces new technical challenges. We need to carefully derive new recursive inequalities that simultaneously involve both primal and dual sequences, $\{\bm{x}^k\}$ and $\{\bm{y}^k\}$, which cannot be obtained by applying existing results directly. In particular, to establish the asymptotic (super)linear convergence rate in Theorem \ref{thm:Q-linear-rate}, we derive a novel recursion \eqref{eq:dist-xybark}, which critically relies on the presence of a strictly positive proximal parameter $\tau_{k}$, as $\tau_{k}$ appears in the denominator of the coefficient on the right-hand side. Consequently, this novel recursion is tailored to our ripALM framework.

\begin{theorem}\label{thm:convergence}
Let the functions $G$, $F$ and $\ell$ be defined as in \eqref{eq:para-objpri}, \eqref{eq:para-objdual} and \eqref{eq:ripALM_lag}, respectively. Let $\rho\in[0,1)$, $\{\sigma_{k}\}$ be a positive sequence satisfying that $\sigma_k\geq\sigma_{\min}>0$ for all $k\geq0$, and $\{\tau_{k}\}$ be a positive sequence satisfying that
\begin{equation*}
\tau_{k}\geq\tau_{\min}>0, \quad \tau_{k+1}\leq(1+\nu_{k})\tau_{k} \quad
\mbox{with} \quad \nu_{k}\geq0 ~~\mbox{and}~~ {\textstyle\sum_{k=0}^{\infty}}\nu_{k} < +\infty.
\end{equation*}
Let $\{\bm{y}^{k}\}$, $\{\Delta^k\}$, $\{\bm{w}^k\}\subset\mathbb{R}^{M}$ and $\{\bm{x}^{k}\}\subset\mathbb{R}^{N}$ be sequences generated by Algorithm \ref{algo:ripALM}. If $\ell$ admits a saddle point (i.e., $(\partial\ell)^{-1}(\bm{0},\bm{0})\neq\emptyset$), then the following statements hold.
\begin{enumerate}[label=(\roman*)]
\item The sequences $\{\bm{y}^k\}$, $\{\bm{w}^k\}$ and $\{\bm{x}^k\}$ are bounded.
		
\item $\lim\limits_{{k}\to\infty}\Delta^{k+1}=\bm{0}$, $\lim\limits_{{k}\to\infty}\bm{\theta}^{k+1}=\bm{0}$ and $\lim\limits_{{k}\to\infty}\bm{\xi}^{k+1}=\bm{0}$, where $\bm{\theta}^{k+1}$ and $\bm{\xi}^{k+1}$ are defined by
	\begin{equation*}
	\bm{\theta}^{k+1}:=\Delta^{k+1} - \tau_{k}\sigma_{k}^{-1}(\bm{y}^{k+1}
    - \bm{y}^{k})\quad\text{and}\quad \bm{\xi}^{k+1}:=\sigma_{k}^{-1}(\bm{x}^{k}-\bm{x}^{k+1}), \quad \forall\,k\geq0.
	\end{equation*}
		
\item Both sequences $\left\{G(\bm{y}^{k+1},\bm{\xi}^{k+1})\right\}$ and $\left\{F(\bm{\theta}^{k+1},\,\bm{x}^{k+1})\right\}$ converge to the common optimal value of problems \eqref{eq:para-proorg} and \eqref{eq:para-proorgdual}.
	
\item Any accumulation point of $\{\bm{y}^k\}$ is an optimal solution of problem \eqref{eq:para-proorg} (i.e., problem \eqref{eq:maindual}), and any accumulation point of $\{\bm{x}^k\}$ is an optimal solution of problem \eqref{eq:para-proorgdual} (i.e., problem \eqref{eq:mainprob}).
    	
\item The sequence $\{\bm{x}^k\}$ converges to an optimal solution of problem \eqref{eq:para-proorgdual}.
\end{enumerate}
\end{theorem}
\begin{proof}
See Appendix \ref{sec:proofconver}.
\end{proof}

Next, we study the asymptotic (super)linear convergence rate of our ripALM under an error bound condition presented in Assumption \ref{asp:error-bound-Li}. As noted in \cite[Lemma 2.4]{lst2020asymptotically}, this error bound condition is weaker than the local upper Lipschitz continuity of $(\partial\ell)^{-1}$ at the origin. The latter condition was used in \cite{zc2020linear} to establish the asymptotic (super)linear convergence rate for Eckstein and Silva's relative-type inexact ALM, while the former, weaker condition has also been employed in \cite{lst2020asymptotically} and \cite{ylct2024corrected} to establish the asymptotic (super)linear convergence rate for {\sc Snipal} and ciPALM, respectively.

\begin{assumption}\label{asp:error-bound-Li}
For any $r>0$, there exists a constant $\kappa>0$ such that, for any $(\bm{y},\bm{x}) \in \left\{(\bm{y},\bm{x})\in\mathbb{R}^M\times\mathbb{R}^N \mid \mathrm{dist}\left((\bm{y},\bm{x}), \,(\partial \ell)^{-1}(\bm{0}, \bm{0})\right) \leq r\right\}$,
\begin{equation}\label{eq:error-bound-Li}
\mathrm{dist}\left((\bm{y},\bm{x}), \,(\partial\ell)^{-1}(\bm{0}, \bm{0})\right)
\leq \kappa\,\mathrm{dist}\left((\bm{0}, \bm{0}), \,\partial\ell(\bm{y},\bm{x})\right).
\end{equation}
\end{assumption}

\begin{theorem}\label{thm:Q-linear-rate}
Let $\ell$ be defined as in \eqref{eq:ripALM_lag}, and let $\rho\in[0,1)$, $\{\sigma_{k}\}$ be a positive sequence satisfying that $\sigma_k\geq\sigma_{\min}>0$ for all $k\geq0$, and $\{\tau_{k}\}$ be a positive sequence satisfying that
\begin{equation*}
\tau_{k}\geq\tau_{\min}>0, \quad \tau_{k+1}\leq(1+\nu_{k})\tau_{k} \quad
\mbox{with} \quad \nu_{k}\geq0 ~~\mbox{and}~~ {\textstyle\sum_{k=0}^{\infty}}\nu_{k} < +\infty.
\end{equation*}
Suppose additionally that $\ell$ admits a saddle point (i.e., $(\partial\ell)^{-1}(\bm{0},\bm{0})\neq\emptyset$), Assumption \ref{asp:error-bound-Li} holds, and the sequences of parameters $\rho$, $\{\sigma_k\}$ and $\{\tau_k\}$ satisfy that
\begin{equation}\label{para-conds}
\sqrt{\tau_{\min}} - 2\sqrt{\rho} > 0
\quad \text{and} \quad
\liminf\limits_{k\to\infty} ~ \sigma_{k} > c\cdot\frac{2\kappa\sqrt{\tau_{\max}}\left(\rho+\sqrt{\rho\,\overline{\tau}_{\max}}\right)}{\sqrt{\tau_{\min}} - 2\sqrt{\rho}},
\end{equation}
where $c>1$ is an arbitrarily given positive constant, $\tau_{\max}:=\tau_{0}\prod_{k=0}^{\infty}(1+\nu_{k})$, and $\overline{\tau}_{\max}:=\max\left\{1,\tau_{\max}\right\}$. Let $\Lambda^{k} := \operatorname{Diag}(\tau_{k}I_{M},I_{N})$, $\overline{\tau}_{k} := \max\left\{1, \tau_{k}\right\}$ and
\begin{equation*}
\gamma_{k} := \left(1 - \frac{2\kappa\sqrt{\tau_{k}}\left(\rho+\sqrt{\rho\overline{\tau}_{k}}\right)
+2\sigma_{k}\sqrt{\rho}}{\sigma_{k}\sqrt{\tau_{k}}}\right)
\frac{\sigma_{k}^{2}}{\kappa^2\left(\sqrt{\rho}+\sqrt{\overline{\tau}_{k}}\right)^2\overline{\tau}_{k}}.
\end{equation*}
Then, the following statements hold.
\begin{enumerate}[label=(\roman*)]
\item For all sufficiently large $k$, we have that
    \begin{equation*}
    \gamma_{k}\geq
    \left(\frac{c-1}{c}\right)
    \cdot\frac{\sqrt{\tau_{\min}} - 2\sqrt{\rho}}{\sqrt{\tau_{\min}}}
    \cdot\frac{\sigma_k^2}{\kappa^2\left(\sqrt{\rho}+\sqrt{\overline{\tau}_{\max}}\right)^2\overline{\tau}_{\max}} > 0,
    \end{equation*}
    and
    \begin{equation*}
    \mathrm{dist}_{\Lambda^{k+1}}\left((\bm{y}^{k+1},\bm{x}^{k+1}), \,(\partial\ell)^{-1}(\bm{0}, \bm{0})\right) \leq \mu_{k}\, \mathrm{dist}_{\Lambda^{k}}\left((\bm{y}^{k},\bm{x}^{k}), \,(\partial\ell)^{-1}(\bm{0}, \bm{0})\right),
    \end{equation*}
    where
    \begin{equation*}
    \mu_{k}:=\sqrt{\frac{1+\nu_{k}}{1+\gamma_{k}}}
    \quad \mbox{satisfies} \quad
    \limsup\limits_{k\to\infty}\,\left\{\mu_{k}\right\}<1
    ~~\mbox{as}~~\nu_{k}\to0.
    \end{equation*}

\item The sequence $\{\bm{y}^k\}$ is convergent.
\end{enumerate}

\end{theorem}
\begin{proof}
See Appendix \ref{sec:proof_linear_rate}.
\end{proof}

\begin{remark}[Comments on $\gamma_k$ and $\mu_{k}$]\label{rmkmu}
One can see from the expression of $\gamma_k$ in Theorem \ref{thm:Q-linear-rate} that, after a finite number of iterations, $\gamma_k$ becomes proportional to the squared penalty parameter $\sigma_k^2$, provided that $\rho$, $\{\sigma_k\}$ and $\{\tau_k\}$ satisfy the conditions in \eqref{para-conds}. Therefore, if $\sqrt{\tau_{\min}}-2\sqrt{\rho} > 0$, the convergence factor $\mu_{k}$ can be less than 1 when $\sigma_k$ is sufficiently large, and it approaches 0 as $\sigma_k$ tends to $+\infty$. This readily demonstrates that the sequence $\left\{(\bm{y}^{k},\bm{x}^{k})\right\}$ converges to the set of saddle points at a Q-(super)linear rate if $\sigma_k$ is sufficiently large and $\sqrt{\tau_{\min}} - 2\sqrt{\rho} > 0$. In practical implementations, one could simply choose an increasing sequence of $\{\sigma_{k}\}$ with $\sigma_{k} \uparrow +\infty$, and set $\tau_{\min}\geq4$ to ensure that $\sqrt{\tau_{\min}} - 2\sqrt{\rho} > 0$ for any $\rho \in [0,1)$. In contrast, as discussed in
\cite[Remark 2.1]{ylct2024corrected}, ciPALM (which is also of relative-type but uses the error criterion \eqref{eq:inexact_ciPALM}) should require $\rho<\frac{1}{3}$ to guarantee an asymptotic (super)linear convergence rate under the same error bound condition in Assumption \ref{asp:error-bound-Li}. This highlights another advantage of our ripALM, as it offers greater flexibility in choosing the tolerance parameter $\rho$.
\end{remark}

In summary, based on the discussions in Section \ref{sec:ripALM} and the convergence results established in this section, we see that ripALM enjoys stronger theoretical guarantees than ES-iALM and, unlike ciPALM, does not require correction steps.

\section{Applications of ripALM}\label{sec-applications}

In this section, we present applications of the ripALM in Algorithm \ref{algo:ripALM} to two important problems: the quadratically regularized optimal transport problem and the basis pursuit denoising problem.

\subsection{Quadratically regularized optimal transport}\label{sec:app2qrot}

As an important variant of the classical optimal transport problem, the quadratically regularized optimal transport (QROT) problem introduces a quadratic regularization (i.e., ridge regularization) term into the objective function. This renders the problem strongly convex while preserving the sparsity of the optimal transport plan. These favorable properties make QROT a compelling alternative to the popular entropically regularized optimal transport problem, especially in applications where the sparse structure is desirable; see, e.g., \cite{blondel2018smooth,essid2018quadratically,lorenz2021quadratically} for more details. Mathematically, the QROT problem is given as follows:
\begin{equation}\label{QROTprob}
\min\limits_{X}~~\frac{\lambda}{2}\|X\|_F^2 + \langle C, \,X\rangle
\quad\mathrm{s.t.}\quad
X\bm{1}_n = \bm{\alpha}, ~~X^{\top}\bm{1}_m = \bm{\beta}, ~~X\geq0,
\end{equation}
where $\lambda>0$ is a given regularization parameter, $C\in\mathbb{R}^{m \times n}_+$ is a given cost matrix, $\bm{\alpha}:=(\alpha_1,\cdots,\alpha_m)^{\top}\in\Sigma_m$ and  $\bm{\beta}:=(\beta_1,\cdots,\beta_n)^{\top}\in\Sigma_n$ are given probability vectors with $\Sigma_{m}$ (resp., $\Sigma_{n}$) denoting the $m$ (resp., $n$)-dimensional unit simplex, and $\bm{1}_{m}$ (resp., $\bm{1}_{n}$) denotes the $m$ (resp., $n$)-dimensional vector of all ones.

To apply the proposed ripALM, we first reformulate \eqref{QROTprob} as follows:
\begin{equation}\label{eq:QROTre}
\min_{X\in\mathbb{R}^{m \times n}}~~
f_{\texttt{q}}(X) := \frac{\lambda}{2}\left\|{X}\right\|_{F}^{2}
+ \left\langle {C}, {X}\right\rangle
+ \delta_{\mathbb{R}^{m\times n}_{+}}(X)
\quad\mathrm{s.t.}\quad
X\bm{1}_{n} = \bm{\alpha}, ~
X^{\top}\bm{1}_{m} = \bm{\beta}.
\end{equation}
Clearly, it falls into the form of problem \eqref{eq:mainprob} (upon the vectorization of $X$). Further, one can show that the dual problem of \eqref{eq:QROTre} is given by (modulo a minus sign)
\begin{equation}\label{eq:dualQROT}
\min\limits_{\bm{u},\,\bm{v}}~~
f_{\texttt{q}}^{*}\left(\bm{u}\bm{1}_{n}^{\top} + \bm{1}_{m}\bm{v}^{\top}\right)
- \bm{\alpha}^{\top}\bm{u}
- \bm{\beta}^{\top}\bm{v},
\end{equation}
and its associated augmented Lagrangian function is given by (using a similar deduction as Section \ref{sec:ripALM})
\begin{equation*}
\begin{aligned}
\mathcal{L}_{\sigma}(\bm{u},\bm{v},X)
= & -\bm{\alpha}^{\top}\bm{u} -\bm{\beta}^{\top}\bm{v} + \frac{1}{2\sigma}\left\|X + \sigma\left(\bm{u}\bm{1}_{n}^{\top} + \bm{1}_{m}\bm{v}^{\top}\right)\right\|_{F}^{2} \\
& - \frac{1}{2\sigma}\|X\|_{F}^{2} - \mathtt{M}_{\sigma f_{\texttt{q}}}\left(X + \sigma\left(\bm{u}\bm{1}_{n}^{\top} + \bm{1}_{m}\bm{v}^{\top}\right)\right),
\end{aligned}
\end{equation*}
where $\bm{u} \in \mathbb{R}^{m}$, $\bm{v} \in \mathbb{R}^{n}$ are the Lagrange multipliers corresponding to $X\bm{1}_{n} = \bm{\alpha}$ and $X^{\top}\bm{1}_{m} = \bm{\beta}$, respectively, and the conjugate function $f_{\texttt{q}}^{*}$ admits the following expression:
\begin{equation*}
f_{\texttt{q}}^{*}(Z)=\left\{
\begin{aligned}
&\delta_{\mathbb{R}^{m\times n}_{+}}\left(C-Z\right), && \mathrm{if}~~ \lambda = 0, \\[3pt]
&\frac{1}{2\lambda}\left\|\Pi_{\mathbb{R}^{m\times n}_{+}}\left(Z - C\right)\right\|_{F}^{2}, && \mathrm{if}~~ \lambda > 0.
\end{aligned}\right.
\end{equation*}
Thus, given an arbitrary initial {guess} $\left(\bm{u}^{0},\bm{v}^{0}, X^{0}\right) \in \mathbb{R}^{m} \times \mathbb{R}^{n} \times \mathbb{R}^{m\times n}$, the basic iterative scheme of our ripALM for solving \eqref{eq:dualQROT} reads as follows:
\begin{numcases}{}
\left(\bm{u}^{k+1}, \bm{v}^{k+1}\right) \approx \mathop{\arg\min}\limits_{\bm{u}\in\mathbb{R}^{m},\bm{v}\in\mathbb{R}^{n}}
\left\{\Psi_k(\bm{u},\bm{v})\right\}, \label{eq:ripALM_QROT_uvnext} \\[3pt]
X^{k+1} = \mathtt{prox}_{\sigma_{k} f_{\texttt{q}}}\left(X^k+\sigma_{k}\left(\bm{u}^{k+1}\bm{1}_{n}^{\top} + \bm{1}_{m}(\bm{v}^{k+1})^{\top}\right)\right),  \label{eq:ripALM_QROT_Xnext}
\end{numcases}
where
\begin{equation*}
\Psi_k(\bm{u},\bm{v}):= \mathcal{L}_{\sigma_{k}}(\bm{u},\bm{v},X^k) + \frac{\tau_{k}}{2\sigma_{k}}\left\|\bm{u}-\bm{u}^k\right\|^2 + \frac{\tau_{k}}{2\sigma_{k}}\left\|\bm{v}-\bm{v}^k\right\|^2.
\end{equation*}

To truly implement ripALM for solving problem \eqref{eq:QROTre}, it is essential to efficiently solve the subproblem \eqref{eq:ripALM_QROT_uvnext} to find $(\bm{u}^{k+1},\bm{v}^{k+1})$ satisfying the error criterion \eqref{ripALM-inexcond}. In the following, we shall describe how to apply a semismooth Newton ({\sc Ssn}) method to achieve this goal. For simplicity, we drop the index $k$ and explicitly rewrite the subproblem \eqref{eq:ripALM_QROT_uvnext} as follows:
\begin{equation}\label{eq:iPALM_QROT_Tuv}
\hspace{-2mm}
\min\limits_{\bm{u}\in\mathbb{R}^{m},\,\bm{v}\in\mathbb{R}^{n}}\left\{
\begin{aligned}
&\Psi(\bm{u},\bm{v}) :=   -\bm{\alpha}^{\top}\bm{u} -\bm{\beta}^{\top}\bm{v} + \frac{1}{2\sigma}\left\|\bar{X} + \sigma\left(\bm{u}\bm{1}_{n}^{\top} + \bm{1}_{m}\bm{v}^{\top}\right)\right\|_{F}^{2} - \frac{1}{2\sigma}\|\bar{X}\|_{F}^{2} \\
& \quad - \mathtt{M}_{\sigma f_{\texttt{q}}}\left(\bar{X} + \sigma\left(\bm{u}\bm{1}_{n}^{\top} + \bm{1}_{m}\bm{v}^{\top}\right)\right) + \frac{\tau}{2\sigma}\left\|\bm{u}-\bar{\bm{u}}\right\|^2 + \frac{\tau}{2\sigma}\left\|\bm{v}-\bar{\bm{v}}\right\|^2
\end{aligned}\right\},
\end{equation}
where $\bar{\bm{u}}$, $\bar{\bm{v}}$ and $\bar{X}$ are given. From the property of the Moreau envelope (see, for example, \cite[Proposition 12.30]{bc2017convex}), we see that $\Psi$ is strongly convex and continuously differentiable with the gradient
\begin{equation*}
\nabla \Psi(\bm{u},\bm{v}) =
\left[\begin{matrix}
\left(\mathtt{prox}_{\sigma f_{\texttt{q}}}\left(\bar{X}+\sigma\left(\bm{u}\bm{1}_{n}^{\top} + \bm{1}_{m}\bm{v}^{\top}\right)\right)\right)\bm{1}_{n} - \bm{\alpha} + \tau\sigma^{-1}\left(\bm{u} - \bar{\bm{u}}\right) \vspace{1mm}\\
\left(\mathtt{prox}_{\sigma f_{\texttt{q}}}\left(\bar{X}+\sigma\left(\bm{u}\bm{1}_{n}^{\top} + \bm{1}_{m}\bm{v}^{\top}\right)\right)\right)^{\top}\bm{1}_{m} - \bm{\beta} + \tau\sigma^{-1}\left(\bm{v} - \bar{\bm{v}}\right)
\end{matrix}\right].
\end{equation*}
Note that the proximal mapping $\mathtt{prox}_{\sigma f_{\texttt{q}}}(\cdot)$ can be easily computed as follows:
\begin{equation*}
\mathtt{prox}_{\sigma f_{\texttt{q}}}\left(Z\right)
= \frac{1}{1+\lambda\sigma}\Pi_{\mathbb{R}_{+}^{m\times n}}\left(Z - \sigma{C}\right), \quad \forall\,Z\in\mathbb{R}^{m\times n}.
\end{equation*}
Then, from the first-order optimality condition, solving problem \eqref{eq:iPALM_QROT_Tuv} is equivalent to solving the following non-smooth equation:
\begin{equation}\label{eq:SSN_equation}
\nabla \Psi(\bm{u},\bm{v}) = \bm{0}.
\end{equation}
In view of the nice property of $\nabla\Psi$, we are able to follow \cite{lst2018highly,lst2020asymptotically,lst2020efficient,ylct2024corrected} to apply a globally convergent and locally superlinearly convergent {\sc Ssn} method to solve \eqref{eq:SSN_equation}. Specifically, let $\mathcal{W}(\bm{u},\bm{v}) := \bar{X} + \sigma\left({\bm{u}\bm{1}_{n}^{\top} + \bm{1}_{m}\bm{v}^{\top} - C}\right)$ and define the multifunction $\widehat{\partial}(\nabla\Psi): \mathbb{R}^{m}\times\mathbb{R}^{n} \rightrightarrows \mathbb{R}^{(m+n)\times(m+n)}$ as follows:
\begin{equation*}
\widehat{\partial}(\nabla\Psi)(\bm{u},\bm{v})
:= \left\{H\in\mathbb{R}^{(m+n)\times(m+n)} ~\left\lvert~
\begin{aligned}
H:= &\;\frac{\sigma}{1+\lambda\sigma}\,B\mathrm{Diag}(\texttt{vec}(\Omega))B^{\top}
+ \frac{\tau}{\sigma}I_{m+n}, \\[3pt]
\forall&\,\Omega\in\partial{\Pi_{\mathbb{R}_{+}^{m\times n}}}(\mathcal{W}(\bm{u},\bm{v}))
\end{aligned}
\right.\right\},
\end{equation*}
where $B:=\begin{bmatrix} \bm{1}_n^{\top} \otimes I_m \\ I_n \otimes \bm{1}_m^{\top} \end{bmatrix}\in\mathbb{R}^{(m+n) \times mn}$ with ``$\otimes$" denoting the Kronecker product, $\texttt{vec}(\Omega)$ denotes the vectorization of $\Omega$ with $[\texttt{vec}(\Omega)]_{i+(j-1)m}=\Omega_{ij}$ for any $1\leq i\leq m$ and $1\leq j \leq n$, $\mathrm{Diag}(\bm{z})$ denotes the diagonal matrix whose $i$th diagonal element is given by $z_i$, and $\partial{\Pi_{\mathbb{R}_{+}^{m\times n}}}(Z)$ denotes the generalized Jacobian of the Lipschitz continuous mapping $\Pi_{\mathbb{R}_{+}^{m\times n}}$ at $Z$, which is defined by
\begin{equation*}
\partial\Pi_{\mathbb{R}_{+}^{m\times n}}(Z)
:= \left\{\Omega\in\mathbb{R}^{m\times n} ~\bigg|~
\Omega_{ij} \in \left\{
\begin{aligned}
&\{1\}, & \text{if} ~~Z_{ij} > 0, \\
&[0,1], & \text{if} ~~Z_{ij} = 0,  \\
&\{0\}, & \text{if} ~~Z_{ij} < 0,
\end{aligned}
\right.\quad \right\}.
\end{equation*}
Then, using similar arguments as in the proof of \cite[Proposition 4.1]{ylct2024corrected}, one can show that $\nabla\Psi(\cdot)$ is strongly semi-smooth with respect to $\widehat{\partial}(\nabla\Psi)(\cdot)$, and thus the {\sc Ssn} method is applicable. More importantly, for any $(\bm{u},\bm{v})\in\mathbb{R}^{m}\times\mathbb{R}^{n}$, all elements of $\widehat{\partial}(\nabla\Psi)(\bm{u},\bm{v})$ are positive definite. This ensures the direct applicability of the {\sc Ssn} method without the need for a specific regularity condition, such as the primal constraint nondegeneracy condition \cite{zst2010newton-cg}, highlighting the advantage of incorporating a proximal term. The detailed description of the {\sc Ssn} method for solving equation \eqref{eq:SSN_equation} is presented in Algorithm \ref{alg:SSN}. We refer readers to \cite[Theorem 3.6]{lst2018highly} for its detailed convergence results.

\begin{algorithm}[!]
	\caption{A semismooth Newton ({\sc Ssn}) method for solving equation \eqref{eq:SSN_equation}}\label{alg:SSN}
	\begin{algorithmic}
		\STATE \textbf{Input:} Choose  $\bar{\mu}\in(0,1)$, $\mu\in(0,1]$,
		$\eta\in(0,1/2)$, ${\delta} \in (0, 1)$, and an initial point $(\bm{u}^0,\bm{v}^0)\in \mathbb{R}^{m}\times\mathbb{R}^{n}$. Set $t = 0$.
		\WHILE{\textit{the termination criterion is not met}}
		
		\STATE \textbf{Step 1.} Compute $\nabla\Psi(\bm{u}^t,\bm{v}^t)$ and choose an element $H^t\in\widehat{\partial}(\nabla\Psi)(\bm{u}^{t},\bm{v}^{t})$. Solve the linear system
		\begin{equation}\label{eq:gen_dir_Newton}
			H^t\bm{d}^{t} = -\nabla\Psi(\bm{u}^t,\bm{v}^t),
		\end{equation}
		nearly exactly by the (sparse) Cholesky decomposition with forward and backward substitutions, or approximately by the preconditioned conjugate gradient (CG) method to find $\bm{d}^{t}:=(\bm{d}_{u}^{t},\bm{d}_{v}^{t})$ such that
		\begin{equation*}\label{eq:CG_criterion}
			\left\|H^t\bm{d}^{t}+\nabla\Psi(\bm{u}^t,\bm{v}^t)\right\|
			\leq \min \left\{\bar{\mu}, \,\left\|\nabla\Psi(\bm{u}^t,\bm{v}^t)\right\|^{1+\mu}\right\}.
		\end{equation*}
		
		\STATE \textbf{Step 2.} (\textbf{Line search}) Find a step size $\alpha_t:={\delta}^{i_{t}}$, where $i_{t}$ is the smallest non-negative integer such that
		\begin{equation*}
			\Psi\left(\bm{u}^{t} + {\delta}^{i_{t}}\bm{d}_{u}^{t}, \,\bm{v}^{t} + {\delta}^{i_{t}}\bm{d}_{v}^{t}\right)
			- \Psi(\bm{u}^{t},\bm{v}^{t})
			\leq \eta {\delta}^{i_{t}} \left\langle\nabla\Psi(\bm{u}^t,\bm{v}^t), \,\bm{d}^{t}\right\rangle.
		\end{equation*}
		
		\STATE \textbf{Step 3.} Set $\bm{u}^{t+1} = \bm{u}^{t} + \alpha_t\bm{d}_{u}^{t}$, $\bm{v}^{t+1} = \bm{v}^{t} + \alpha_t\bm{d}_{v}^{t}$, $t = t + 1$, and go to \textbf{Step 1.}
		
		\ENDWHILE
	\end{algorithmic}
\end{algorithm}

\subsection{Basis pursuit denoising}\label{sec:BPDN_detail}

The basis pursuit denoising (BPDN) problem is a fundamental optimization problem in signal processing and compressed sensing, which aims to recover a sparse signal from noisy linear measurements \cite{van2009probing}. Specifically, given a measurement $\bm{b}\in \mathbb{R}^m$, a dictionary $D\in \mathbb{R}^{m\times n}$, and an estimate of the noise level $\widehat{\kappa}>0$ in the data, the BPDN problem seeks a sparse solution $\bm{s}\in\mathbb{R}^n$ by solving the following constrained optimization problem:
\begin{equation}\label{eq-bpdn}
\min_{\bm{s}\in\mathbb{R}^n}~ \|\bm{s}\|_1 \quad \text{s.t.}\quad \|D\bm{s} - \bm{b}\| \leq \widehat{\kappa}. 
\end{equation}
Note that in the sparse recovery regime, one focuses on the case that $m \ll n$. By introducing an auxiliary variable $\bm{t}\in \mathbb{R}^m$, the above problem can be reformulated as
\begin{equation}\label{eq-bpdn-new}
\min_{\bm{s}\in\mathbb{R}^n, \,\bm{t}\in \mathbb{R}^m}\; \|\bm{s}\|_1 + \delta_{\mathcal{B}^2_{\widehat{\kappa}}}(\bm{t})
\quad \text{s.t.} \quad D\bm{s} - \bm{t} = \bm{b},
\end{equation}
where $\mathcal{B}^2_{\widehat{\kappa}}\subset\mathbb{R}^m$ denotes the $\ell_2$-norm ball centered at $\bm{0}$ with radius $\widehat{\kappa}$. Then, we see that the reformulated problem \eqref{eq-bpdn-new} is a special case of the general problem \eqref{eq:mainprob} by letting $\bm{x}:=(\bm{s};\bm{t})\in\mathbb{R}^{n+m}$, $f(\bm{x}):= \|\bm{s}\|_1 + \delta_{\mathcal{B}^2_{\widehat{\kappa}}}(\bm{t})$ and $A:=[D,-I]\in \mathbb{R}^{m\times (n+m)}$. Consequently, the proposed ripALM is directly applicable to solve the BPDN problem.

To apply ripALM, we derive the dual problem of \eqref{eq-bpdn-new} as follow (modulo a minus sign)
\begin{equation}\label{eq:BPDN-dual}
\min_{\bm{y}\in\mathbb{R}^{m}} ~~
\delta_{\mathcal{B}_1^\infty}({D}^{\top}\bm{y}) + \widehat{\kappa}\|\bm{y}\|
- \bm{b}^{\top}\bm{y},
\end{equation}
where $\mathcal{B}_1^\infty\subset\mathbb{R}^{n}$ denotes the unit $\ell_{\infty}$-norm ball centered at $\bm{0}$. The associated augmented Lagrangian function is given by
\begin{equation*}
\begin{aligned}
\mathcal{L}_{\sigma}(\bm{y},\bm{s},\bm{t}) =
& -\bm{b}^{\top}\bm{y} + \frac{1}{2\sigma}\big\|\bm{s}+\sigma{D}^{\top}\bm{y}\big\|^2 - \frac{1}{2\sigma}\|\bm{s}\|^2 - \mathtt{M}_{\sigma \|\cdot\|_1}\big(\bm{s}+\sigma{D}^{\top}\bm{y}\big) \\
& + \frac{1}{2\sigma}\big\|\bm{t}-\sigma\bm{y}\big\|^2
- \frac{1}{2\sigma}\|\bm{t}\|^2
- \mathtt{M}_{\sigma\delta_{\mathcal{B}_{\widehat{\kappa}}^2}}\big(\bm{t}-\sigma\bm{y}\big).
\end{aligned}
\end{equation*}
Thus, given an arbitrary initial guess $(\bm{y}^{0},\bm{s}^{0},\bm{t}^{0})\in\mathbb{R}^{m}\times\mathbb{R}^{n}\times\mathbb{R}^{m}$, 
the basic iterative scheme of ripALM for solving problem \eqref{eq:BPDN-dual} reads as follows:
\begin{numcases}{}
\bm{y}^{k+1} \approx \arg\min_{\bm{y}\in\mathbb{R}^{m}}
\left\{\mathcal{L}_{\sigma_k}(\bm{y},\bm{s}^{k},\bm{t}^{k})
+ \frac{\tau_{k}}{2\sigma_{k}}\|\bm{y}-\bm{y}^{k}\|^{2}\right\}, \label{eq:BPDN-y} \\
\bm{s}^{k+1} = \mathtt{prox}_{\sigma_{k}\|\cdot\|_1}\big(\bm{s}^{k}+\sigma_{k} D^{\top}\bm{y}^{k+1}\big), \label{eq:BPDN-x1} \\
\bm{t}^{k+1} = \Pi_{\mathcal{B}_{\widehat{\kappa}}^2}\big(\bm{t}^{k}-\sigma_{k}\bm{y}^{k+1}\big), \label{eq:BPDN-x2}
\end{numcases}
where $\Pi_{\mathcal{B}_{\widehat{\kappa}}^2}$ denotes the projection operator onto $\mathcal{B}_{\widehat{\kappa}}^2$.

To truly implement ripALM for solving problem \eqref{eq:BPDN-dual}, it is essential to efficiently solve the subproblem \eqref{eq:BPDN-y} to find $\bm{y}^{k+1}$ satisfying the relative error criterion \eqref{ripALM-inexcond}. In the following, we describe how to apply the {\sc Ssn} method to achieve this goal. For simplicity, we drop the index $k$ and explicitly rewrite the subproblem \eqref{eq:BPDN-y} as follows:
\begin{equation*}
\min_{\bm{y}\in\mathbb{R}^{m}} ~~
\left\{\begin{aligned}
\Phi(\bm{y}) :=
& -\bm{b}^{\top}\bm{y} + \frac{1}{2\sigma}\big\|\bar{\bm{s}}+\sigma{D}^{\top}\bm{y}\big\|^2 - \frac{1}{2\sigma}\|\bar{\bm{s}}\|^2 - \mathtt{M}_{\sigma \|\cdot\|_1}\big(\bar{\bm{s}}+\sigma{D}^{\top}\bm{y}\big) \\
& + \frac{1}{2\sigma}\big\|\bar{\bm{t}}-\sigma\bm{y}\big\|^2- \frac{1}{2\sigma}\|\bar{\bm{t}}\|^2 - \mathtt{M}_{\sigma \delta_{\mathcal{B}^2_{\widehat{\kappa}}}}\big(\bar{\bm{t}}-\sigma\bm{y}\big) + \frac{\tau}{2\sigma}\|\bm{y}-\bar{\bm{y}}\|^{2}
\end{aligned}\right\}.
\end{equation*}
where $\bar{\bm{s}}$, $\bar{\bm{t}}$ and $\bar{\bm{y}}$ are given vectors. From the property of the Moreau envelope, we see that $\Psi$ is strongly convex and continuously differentiable with the gradient
\begin{equation}\label{eq:iPALM_BPDN_ynext}
\begin{aligned}
\nabla\Phi(\bm{y})
= {D}\mathtt{prox}_{\sigma\|\cdot\|_{1}}\big(\bar{\bm{s}} + \sigma D^{\top}\bm{y}\big)
- \Pi_{{\mathcal{B}_{\widehat{\kappa}}^2}}\big(\bar{\bm{t}} - \sigma\bm{y}\big)
- \bm{b} + \frac{\tau}{\sigma}\big(\bm{y} - \bar{\bm{y}}\big).
\end{aligned}
\end{equation}
Then, from the first-order optimality condition, solving problem \eqref{eq:BPDN-y} is equivalent to solving the following non-smooth equation:
\begin{equation}\label{eq:optcond-y}
\nabla\Phi(\bm{y}) = 0.
\end{equation}
Next, we define the multifunction $\widehat{\partial}(\nabla\Phi):\mathbb{R}^{m}\rightrightarrows\mathbb{R}^{m}$ as follows:
\begin{equation}\label{defpatPhiBPDN}
\widehat{\partial}(\nabla\Phi)(\bm{y}) := \left\{H:= \sigma DUD^{\top} + \sigma V + \frac{\tau}{\sigma}I_{m} ~\left|~\,
\begin{aligned}
&\forall\,U\in\partial\mathtt{prox}_{\sigma\|\cdot\|_1}\big(\bar{\bm{s}}+\sigma D^{\top}\bm{y}\big),  \\
&\forall\,V\in\partial\Pi_{\mathcal{B}_{\widehat{\kappa}}^2}\big(\bar{\bm{t}}-\sigma\bm{y}\big).
\end{aligned} \right.\right\},
\end{equation}
where $\partial\mathtt{prox}_{\sigma\|\cdot\|_1}
:\mathbb{R}^{n}\rightrightarrows\mathbb{R}^{n \times n}$ is defined by
\begin{equation*}
\partial \mathtt{prox}_{\sigma\|\cdot\|_1 }(\bm{u}) := \left\{ \operatorname{Diag}(\bm{w}) ~ \left| ~
\begin{aligned}
& w_{i} = \{1\}, &\text{if}~~\left|u_{i}\right|>\sigma, \vspace{1mm} \\
& w_{i} \in \left[0, 1\right], &\text{if}~~\left|u_{i}\right|=\sigma, \\
& w_{i} = \left\{0\right\}, &\text{if}~~\left|u_{i}\right|<\sigma.
\end{aligned} \right.
\right\},
\end{equation*}
and $\partial\Pi_{\mathcal{B}_{\widehat{\kappa}}^2}
:\mathbb{R}^{m}\rightrightarrows\mathbb{R}^{m \times m}$ is defined by
\begin{equation*}
\partial \Pi_{\mathcal{B}_{\widehat{\kappa}}^2}(\bm{v}) =
\begin{cases}
\left\{\cfrac{\widehat{\kappa}}{\|\bm{v}\|}\left(I-\cfrac{\bm{v} \bm{v}^{\top}}{\|\bm{v}\|^2}\right)\right\},
&\text{if}~~\left\|\bm{v}\right\|>\widehat{\kappa}, \vspace{1mm} \\
\left\{I_m - \cfrac{t}{\widehat{\kappa}^2}\bm{v}\bm{v}^{\top} \,\mid\, 0 \leq t \leq 1\right\},
&\text{if}~~\left\|\bm{v}\right\|=\widehat{\kappa}, \vspace{1mm} \\
\big\{I_m\big\},
&\text{if}~~\left\|\bm{v}\right\|<\widehat{\kappa}.
\end{cases}
\end{equation*}
Using the relation $\Pi_{\mathcal{B}_{\widehat{\kappa}}^2}(\bm{x}) = \bm{x} - \mathtt{prox}_{\widehat{\kappa}\|\cdot\|_2}(\bm{x})$, Lemma 2.1 in \cite{zhang2020efficient}, and similar arguments to those in \cite[Proposition 4.1]{ylct2024corrected}, one can show that $\nabla\Phi$ is strongly semismooth with respect to the multifunction $\widehat{\partial}(\nabla\Phi)$. This property enables the application of the {\sc Ssn} method to solve \eqref{eq:optcond-y}. Since the algorithmic framework is identical to Algorithm \ref{alg:SSN}, we omit the detailed description for simplicity. Moreover, note from \eqref{defpatPhiBPDN} that, for any $\bm{y}\in\mathbb{R}^m$, all matrices in $\widehat{\partial}(\nabla \Phi)(\bm{y})$ are symmetric positive definite. This ensures the local superlinear (even quadratic) convergence rate of the {\sc Ssn} method, without requiring additional regularity conditions. These favorable theoretical properties form the basis for the encouraging numerical performance, as observed from our numerical results in the next section.

Finally, we would like to point out that the main computational cost of the {\sc Ssn} method for solving \eqref{eq:optcond-y} lies in solving a sequence of linear systems of the form
\begin{equation}\label{eq:gen_dir_Newton-BPDN}
H\bm{d} = -\nabla\Phi(\bm{y}),
\end{equation}
to obtain the Newton direction $\bm{d}$, where $H \in \widehat{\partial}(\nabla\Phi)(\bm{y})$ with $\widehat{\partial}(\nabla\Phi)$ defined in \eqref{defpatPhiBPDN}. Inspired by the techniques developed in \cite[Section 3.3]{lst2018highly}, we can exploit the special structure of $H$ to enhance the computational efficiency of solving \eqref{eq:gen_dir_Newton-BPDN}. However, the presence of the additional block $V$  in $H$ introduces nontrivial implementation difficulties. Detailed discussions on efficiently solving \eqref{eq:gen_dir_Newton-BPDN} are provided in Appendix \ref{sec:SMW_BPDN}.

\section{Numerical experiments}\label{sec:numexp}

In this section, we conduct numerical experiments to evaluate the performance of our ripALM in Algorithm \ref{algo:ripALM} for solving the QROT problem \eqref{QROTprob} and the BPDN problem \eqref{eq-bpdn}. Our experiments are organized as follows: 
\begin{itemize}
\item Section \ref{sec:qrot-exp} reports experiments on QROT. We first compare ripALM with {\sc Snipal} \cite{lst2020asymptotically} and ciPALM \cite{ylct2024corrected}, which represent the latest absolute-type and relative-type inexact pALMs, respectively, for solving QROT problems on image data. These comparisons are intended to illustrate how different error criteria and tolerance parameters affect the numerical performance of the inexact pALM. We then compare ripALM with Gurobi, Eckstein and Silva's method \cite{es2013practical} (denoted by ES-iALM), and a dual alternating direction method of multipliers (dADMM) presented in Appendix \ref{sec:ADMM}, for solving large-scale QROT problems. Here, Gurobi is one of the state-of-the-art commercial solvers, ES-iALM is a relative-type inexact version of the vanilla ALM, and dADMM is a widely used first-order method for large-scale constrained convex optimization; see, e.g, \cite{boyd2011distributed,gabay1976dual}.
	
\item Section \ref{sec:bpdn-exp} reports experiments on BPDN. Similar to the QROT experiments, we first compare ripALM with {\sc Snipal} and ciPALM to examine how different error criteria and tolerance parameters influence the numerical performance. We then compare ripALM with SPGL1 \cite{van2009probing}, Mosek, and the dADMM presented in Appendix \ref{sec:BPDN-dadmm}, for solving the BPDN problem on the UCI dataset. Here, SPGL1 is a state-of-the-art solver specifically designed for BPDN, while Mosek is a highly efficient commercial solver that can be used to solve the second-order cone programming reformulation of the BPDN problem; see Appendix \ref{sec:BPDN_SOCP} for more details.
\end{itemize}
All experiments are run in {\sc Matlab} R2023a on a PC with Intel processor i7-12700K@3.60GHz (with 12 cores and 20 threads) and 64GB of RAM, equipped with a Windows OS. The implementation details are given as follows.

\vspace{1mm}
\textbf{Termination conditions.} We first describe the termination conditions for QROT. Let $\mathcal{Z}(\bm{u},\bm{v},X) := C + \lambda X - \bm{u}\bm{1}_{n}^{\top} - \bm{1}_{m}\bm{v}^{\top}$. The Karush-Kuhn-Tucker (KKT) system for problem \eqref{QROTprob} and its dual \eqref{eq:dualQROT} is given by
\begin{equation}\label{kkt-QROT}
	\begin{array}{l}
		X\bm{1}_{n} = \bm{\alpha}, ~~
		X^{\top}\bm{1}_{m} = \bm{\beta}, ~~
		\langle X, \,\mathcal{Z}(\bm{u},\bm{v},X)\rangle = 0, ~~X\geq0,
		~~\mathcal{Z}(\bm{u},\bm{v},X) \geq 0.
	\end{array}
\end{equation}
Note that $(\bm{u},\bm{v},X)$ satisfies the KKT system \eqref{kkt-QROT} if and only if $X$ solves \eqref{QROTprob} and $(\bm{u}, \bm{v})$ solves \eqref{eq:dualQROT}, respectively. Based on \eqref{kkt-QROT}, we define the relative KKT residual for any $(X,\,\bm{u},\,\bm{v})$ as follows:
\begin{equation*}
	\Delta_{\rm kkt}^{\rm qrot}(\bm{u},\bm{v},X)
	:= \max\left\{\Delta_p^{\rm qrot}(X), \,\Delta_d^{\rm qrot}(\bm{u},\bm{v},X), \,\Delta_c^{\rm qrot}(\bm{u},\bm{v},X)\right\},
\end{equation*}
where
\begin{align*}
	&\;\Delta_p^{\rm qrot}(X):= \max\left\{\frac{\|X\bm{1}_{n}-\bm{\alpha}\|}{1+\|\bm{\alpha}\|}, \,\frac{\|X^{\top}\bm{1}_{m}-\bm{\beta}\|}{1+\|\bm{\beta}\|}, \,\frac{\|\min\{X, \,0\}\|_F}{1+\|X\|_F}\right\}, \\[5pt]
	&\; \Delta_d^{\rm qrot}(\bm{u},\bm{v},X):= \frac{\|\min\{\mathcal{Z}(\bm{u},\,\bm{v},\,X), \,0\}\|_F}{1+\|C\|_F},
	\quad
	\;\Delta_c^{\rm qrot}(\bm{u},\bm{v},X) := \frac{\left|\langle X, \,\mathcal{Z}(\bm{u},\,\bm{v},\,X)\rangle\right|}{1+\|C\|_F}.
\end{align*}
Moreover, we define the relative duality gap as follows:
\begin{equation*}
	\Delta_{\rm gap}^{\rm qrot}(\bm{u},\,\bm{v},\,X) := \cfrac{\left|\mathtt{pobj}(X) - \mathtt{dobj}(\bm{u},\bm{v})\right|}{1 + \left|\mathtt{pobj}(X)\right| + \left|\mathtt{dobj}(\bm{u},\bm{v})\right|},
\end{equation*}
where $\mathtt{pobj}(X) := f_{\texttt{q}}(X)$ and $\mathtt{dobj}(\bm{u},\bm{v}) := - f_{\texttt{q}}^{*}\left(\bm{u}\bm{1}_{n}^{\top} + \bm{1}_{m}\bm{v}^{\top}\right) + \bm{\alpha}^{\top}\bm{u} + \bm{\beta}^{\top}\bm{v}$. Using these relative residuals, we will terminate ripALM when it returns a point $(\bm{u}^k,\,\bm{v}^k,\,X^k)$ satisfying
\begin{equation*}
	\Delta_{\mathrm{res}}^{\rm qrot}(\bm{u}^k,\,\bm{v}^k,\,X^k):=\max\left\{\Delta_{\rm kkt}^{\rm qrot}(\bm{u}^k,\,\bm{v}^k,\,X^k), \,\Delta_{\rm gap}^{\rm qrot}(\bm{u}^k,\,\bm{v}^k,\,X^k) \right\} < 10^{-6}.
\end{equation*}

Similarly, we set up the termination conditions for BPDN based on the associated KKT system. Specifically, given an output solution $(\bm{s}, \bm{t}, \bm{y})$, we define 
\begin{align*}
\Delta_{p}^{\rm bpdn}(\bm{s},\bm{t})
&:= \cfrac{\left\|D\bm{s}-\bm{b}-\bm{t}\right\|}{1+\|\bm{b}\|}, \\
\Delta_{d}^{\rm bpdn}(\bm{s},\bm{t},\bm{y})
&:= \cfrac{\sqrt{\big\|\bm{s}-\mathtt{prox}_{\|\cdot\|_{1}}\left(\bm{s} + {D}^{\top}\bm{y}\right)\big\|^{2}
+ \big\|\bm{t} - \Pi_{\mathcal{B}_{\widehat{\kappa}}^2}\left(\bm{t} - \bm{y}\right)\big\|^{2}}}{1 + \|D\|_F}, \\
{\Delta}_{\rm gap}^{\rm bpdn}(\bm{s}, \bm{t}, \bm{y})
&:= \cfrac{\left|\mathtt{pobj}(\bm{s}, \bm{t}) - \mathtt{dobj}(\bm{y})\right|}{1 + \left|\mathtt{pobj}(\bm{s}, \bm{t})\right| + \left|\mathtt{dobj}(\bm{y})\right|},
\end{align*}
where $\mathtt{pobj}(\bm{s}, \bm{t}) := \|\bm{s}\|_1$ and $\mathtt{dobj}(\bm{y}) := -\widehat{\kappa}\|\bm{y}\| + \bm{b}^{\top}\bm{y}$. Then, we terminate ripALM when it returns $(\bm{s}^k, \bm{t}^k, \bm{y}^k)$ such that
\begin{equation*}
\Delta_{\rm res}^{\rm bpdn}(\bm{s}^k, \bm{t}^k, \bm{y}^k) := \max\left\{\Delta_{p}^{\rm bpdn}(\bm{s}^k, \bm{t}^k),
\,\Delta_{d}^{\rm bpdn}(\bm{s}^k, \bm{t}^k, \bm{y}^k),
\,\Delta_{\rm gap}^{\rm bpdn}(\bm{s}^k, \bm{t}^k, \bm{y}^k)\right\} < 10^{-6}.
\end{equation*}

\textbf{Initial points.} 
For comparisons among inexact pALMs, we initialize all methods at the origin in order to isolate the impact of inexactness. For comparisons with other representative methods in Section \ref{sec:qrot-exp} and Section \ref{sec:bpdn-exp}, we will employ a \textit{warm-start} strategy to improve the efficiency of ripALM.
Specifically, for the experiments in Section \ref{sec:qrot-exp}, we first employ an inexact Bregman proximal gradient method (iBPGM), with Sinkhorn's algorithm as a subsolver, for approximately solving the QROT problem \eqref{QROTprob} to generate an initial point for ripALM; see Appendix \ref{sec:iBPGM} for more details. We terminate iBPGM once it produces a point $(\bm{u}^k,\,\bm{v}^k,\,X^k)$ satisfying $\Delta_{\mathrm{res}}^{\rm qrot}(\bm{u}^k,\,\bm{v}^k,\,X^k)<10^{-3}$, or once the number of iterations reaches 500. Moreover, we also apply this warm-start strategy to ES-iALM in Section \ref{sec:qrot-exp} to ensure a fair comparison. For the experiments in Section \ref{sec:bpdn-exp}, we use dADMM in Appendix \ref{sec:BPDN-dadmm} to generate an initial point. We terminate dADMM once it produces a point $(\bm{s}^k,\bm{t}^k,\bm{y}^k)$ satisfying $\Delta_{\mathrm{res}}^{\rm bpdn}(\bm{s}^k,\bm{t}^k,\bm{y}^k)<10^{-3}$, or once the iteration number reaches 100.
Note that the time consumed by the above warm-starting phase is included in the total computational time for both ripALM and ES-iALM. In addition, the initial error variable $\bm{w}^0$ in ripALM and ES-iALM is always set to $\bm{0}$. 

For the {\sc Ssn} method in Algorithm \ref{alg:SSN}, we will initialize it with the origin at the first pALM iteration and then employ a \textit{warm-start} strategy thereafter. Specifically, at each pALM iteration, we initialize the {\sc Ssn} method with the approximate solution obtained by the {\sc Ssn} method in the previous pALM iteration.

\vspace{1mm}
\textbf{Hyperparameters.} Our ripALM as well as {\sc Snipal} and ciPALM require appropriate choices of $\{\sigma_k\}$ and $\{\tau_k\}$ to achieve superior performance. In our experiments, for all these algorithms, we simply set $\sigma_k=\min\big\{10^4, \,\max\big\{10^{-4}, \,1.5^{k}\big\}\big\}$ and $\tau_k\equiv5$ for all $k\geq0$, \textit{without} employing any sophisticated parameter tuning. These choices of $\{\tau_k\}$ and $\{\sigma_k\}$ can satisfy the conditions required in Theorems \ref{thm:convergence} and \ref{thm:Q-linear-rate}; see also Remark \ref{rmkmu}. Moreover, we would like to emphasize that adopting more refined updating rules for $\sigma_k$ and $\tau_k$ is certainly possible and may further improve the numerical performance. In addition, for the {\sc Ssn} method in Algorithm \ref{alg:SSN}, we set $\eta=10^{-4}$, $\delta=0.5$, $\bar{\mu}=10^{-3}$ and $\mu=0.2$.

\subsection{Experiments on QROT}\label{sec:qrot-exp}

In this subsection, we first compare ripALM with {\sc Snipal} \cite{lst2020asymptotically} and ciPALM \cite{ylct2024corrected} for solving the QROT problem \eqref{QROTprob} to illustrate how different error criteria and tolerance parameters affect the numerical performance of the inexact pALM. Notably, both ripALM and ciPALM are of relative-type and only require a single tolerance parameter $\rho$ in $[0,1)$, as shown in \eqref{ripALM-inexcond} and \eqref{eq:inexact_ciPALM}, respectively, while {\sc Snipal} is of absolute-type and requires two summable tolerance parameter sequences $\{\varepsilon_k\}\subseteq\mathbb{R}_+$ and $\{\delta_k\}\subseteq\mathbb{R}_+$, as shown in \eqref{eq:inexact_SNIPAL}. For simplicity, in our comparisons, we set $\varepsilon_k = \varepsilon_0/(k+1)^p,\; \delta_k = \delta_0/(k+1)^q$ with $\varepsilon_0=\delta_0\in\big\{1, 0.1, 0.01\big\}$ and $p,q\in\big\{1.1, \,2.1, \,3.1\big\}$ (hence, there are 27 combinations in total) for {\sc Snipal}. For ripALM and ciPALM, we choose $\rho\in\left\{0.01, 0.05, 0.1, 0.3, 0.5, 0.7, 0.9, 0.99, 0.999\right\}$ (9 choices).

We use images of resolution $64\times64$ from the \texttt{ClassicImages} class in the DOTmark collection \cite{ssg2016DOTmark}, which serves as a benchmark dataset for the OT problem and its variants, to generate QROT instances. Specifically, the \texttt{ClassicImages} class contains 10 different images, and each QROT instance is generated by randomly selecting two different images from this class. The cost matrix $C$ is constructed by calculating the squared Euclidean distances between pixels.

Table \ref{Table:Comp-iPALMs-QROT} presents the average results over 10 instances for the regularization parameter $\lambda\in\{1,\,0.1\}$. Since all three algorithms attain a similar solution accuracy and their computational times are proportional to the total number of {\sc Ssn} iterations, we report only the numbers of outer (pALM) and inner ({\sc Ssn}) iterations, denoted by ``\# (\#)", for a clearer comparison. From the results, it is not surprising to see that the performance of all algorithms depends on the choice of tolerance parameters. In particular, ripALM and ciPALM involve only a \textit{single} tolerance parameter $\rho$ in the smaller interval $[0,1)$, and are therefore easier to tune through a simple one-dimensional grid search. In contrast, {\sc Snipal} requires a three-dimensional grid search over parameters chosen from a much larger range. This supports the practical motivation for developing a relative-type error criterion that avoids prescribing summable tolerance sequences and is more convenient to implement. Moreover, ripALM consistently outperforms ciPALM in terms of the total number of {\sc Ssn} iterations and exhibits more robust performance, especially when a larger $\rho\geq0.5$ is used. This difference may be attributed to the additional correction step required by ciPALM to guarantee convergence, which tends to considerably increase the number of {\sc Ssn} iterations\footnote{One possible reason is that the variable obtained after the correction step may no longer serve as a suitable initial point for the {\sc Ssn} method in the next iteration.} needed to satisfy its error criterion \eqref{eq:inexact_ciPALM}, even though these extra iterations do not contribute to the progress of the pALM iterations. These observations further support the advantage of the vanilla inexact pALM framework advocated in this paper.

\begin{table}[ht]
\caption{Comparisons among {\sc Snipal}, ciPALM, and ripALM under different choices of tolerance parameters for solving the QROT problem \eqref{QROTprob} with $\lambda\in\{1, 0.1\}$, where the test instances are generated using images of resolution $64\times64$ from the \texttt{ClassicImages} class in the DOTmark collection.  }\label{Table:Comp-iPALMs-QROT}
\centering \tabcolsep 6.5pt
\scalebox{0.78}{
	\begin{tabular}{clclcl|cl|cl}
		\toprule
		\multicolumn{10}{c}{$\lambda = 1$} \\
		\cmidrule(lr){1-10}
		\multicolumn{6}{c}{{\sc Snipal}} & \multicolumn{2}{|c}{{ciPALM}} & \multicolumn{2}{|c}{{ripALM}}\\
		\cmidrule(lr){1-6} \cmidrule(lr){7-8} \cmidrule(lr){9-10}
		$(\varepsilon_{0}\!=\!\delta_{0},p,q)$ & \# (\#) & $(\varepsilon_{0}\!=\!\delta_{0},p,q)$ & \# (\#) & $(\varepsilon_{0}\!=\!\delta_{0},p,q)$ & \# (\#) & $\rho$ & \# (\#) & $\rho$ & \# (\#) \\
		\cmidrule(lr){1-6} \cmidrule(lr){7-8} \cmidrule(lr){9-10}
(1,\,1.1,\,1.1) & 17 (45) & (0.1,\,1.1,\,1.1) & 17 (53) & (0.01,\,1.1,\,1.1) & 17 (58) & 0.999 & 17 (165) & 0.999 & 17 (37)  \\ 
(1,\,1.1,\,2.1) & 17 (51) & (0.1,\,1.1,\,2.1) & 17 (57) & (0.01,\,1.1,\,2.1) & 17 (61) & 0.99 & 17 (165) & 0.99 & 17 (37)  \\ 
(1,\,1.1,\,3.1) & 17 (55) & (0.1,\,1.1,\,3.1) & 17 (59) & (0.01,\,1.1,\,3.1) & 17 (63) & 0.9 & 17 (162) & 0.9 & 17 (38)  \\ 
(1,\,2.1,\,1.1) & 17 (45) & (0.1,\,2.1,\,1.1) & 17 (53) & (0.01,\,2.1,\,1.1) & 17 (58) & 0.7 & 17 (139) & 0.7 & 17 (38)  \\ 
(1,\,2.1,\,2.1) & 17 (51) & (0.1,\,2.1,\,2.1) & 17 (57) & (0.01,\,2.1,\,2.1) & 17 (61) & 0.5 & 17 (126) & 0.5 & 17 (39)  \\ 
(1,\,2.1,\,3.1) & 17 (55) & (0.1,\,2.1,\,3.1) & 17 (59) & (0.01,\,2.1,\,3.1) & 17 (63) & 0.3 & 17 (100) & 0.3 & 17 (40)  \\ 
(1,\,3.1,\,1.1) & 17 (45) & (0.1,\,3.1,\,1.1) & 17 (53) & (0.01,\,3.1,\,1.1) & 17 (58) & 0.1 & 17 (75) & 0.1 & 17 (42)  \\ 
(1,\,3.1,\,2.1) & 17 (51) & (0.1,\,3.1,\,2.1) & 17 (57) & (0.01,\,3.1,\,2.1) & 17 (61) & 0.05 & 17 (65) & 0.05 & 17 (44)  \\ 
(1,\,3.1,\,3.1) & 17 (55) & (0.1,\,3.1,\,3.1) & 17 (59) & (0.01,\,3.1,\,3.1) & 17 (63) & 0.01 & 17 (56) & 0.01 & 17 (48)  \\ 
		\bottomrule  
		\vspace{-2mm} \\
		\toprule
		\multicolumn{10}{c}{$\lambda = 0.1$} \\
		\cmidrule(lr){1-10}
		\multicolumn{6}{c}{{\sc Snipal}} & \multicolumn{2}{|c}{{ciPALM}} & \multicolumn{2}{|c}{{ripALM}}\\
		\cmidrule(lr){1-6} \cmidrule(lr){7-8} \cmidrule(lr){9-10}
		$(\varepsilon_{0}\!=\!\delta_{0},p,q)$ & \# (\#) & $(\varepsilon_{0}\!=\!\delta_{0},p,q)$ & \# (\#) & $(\varepsilon_{0}\!=\!\delta_{0},p,q)$ & \# (\#) & $\rho$ & \# (\#) & $\rho$ & \# (\#) \\
		\cmidrule(lr){1-6} \cmidrule(lr){7-8} \cmidrule(lr){9-10}
(1,\,1.1,\,1.1) & 19 (86) & (0.1,\,1.1,\,1.1) & 19 (96) & (0.01,\,1.1,\,1.1) & 19 (100) & 0.999 & 19 (236) & 0.999 & 19 (77)  \\ 
(1,\,1.1,\,2.1) & 19 (93) & (0.1,\,1.1,\,2.1) & 19 (99) & (0.01,\,1.1,\,2.1) & 19 (103) & 0.99 & 19 (247) & 0.99 & 19 (77)  \\ 
(1,\,1.1,\,3.1) & 19 (96) & (0.1,\,1.1,\,3.1) & 19 (101) & (0.01,\,1.1,\,3.1) & 19 (106) & 0.9 & 19 (235) & 0.9 & 19 (77)  \\ 
(1,\,2.1,\,1.1) & 19 (86) & (0.1,\,2.1,\,1.1) & 19 (96) & (0.01,\,2.1,\,1.1) & 19 (100) & 0.7 & 19 (213) & 0.7 & 19 (78)  \\ 
(1,\,2.1,\,2.1) & 19 (93) & (0.1,\,2.1,\,2.1) & 19 (99) & (0.01,\,2.1,\,2.1) & 19 (103) & 0.5 & 19 (178) & 0.5 & 19 (78)  \\ 
(1,\,2.1,\,3.1) & 19 (96) & (0.1,\,2.1,\,3.1) & 19 (101) & (0.01,\,2.1,\,3.1) & 19 (106) & 0.3 & 19 (150) & 0.3 & 19 (80)  \\ 
(1,\,3.1,\,1.1) & 19 (86) & (0.1,\,3.1,\,1.1) & 19 (96) & (0.01,\,3.1,\,1.1) & 19 (100) & 0.1 & 19 (112) & 0.1 & 19 (82)  \\ 
(1,\,3.1,\,2.1) & 19 (93) & (0.1,\,3.1,\,2.1) & 19 (99) & (0.01,\,3.1,\,2.1) & 19 (103) & 0.05 & 19 (109) & 0.05 & 19 (84)  \\ 
(1,\,3.1,\,3.1) & 19 (96) & (0.1,\,3.1,\,3.1) & 19 (101) & (0.01,\,3.1,\,3.1) & 19 (106) & 0.01 & 19 (100) & 0.01 & 19 (89)  \\ 
		\bottomrule
\end{tabular}}
\end{table}

To further evaluate the numerical performance of ripALM, we compare ripALM with Gurobi (version 12.0.3), dADMM, and ES-iALM for solving large-scale QROT problems. For ripALM, we set $\rho=0.99$ based on the numerical observations from Table \ref{Table:Comp-iPALMs-QROT}. For ES-iALM, we adopt the same settings for $\rho$, ${\sigma_{k}}$, and the termination conditions as those used in ripALM. Moreover, we employ FISTA \cite{beck2009fast} to solve the ALM-subproblems, with the maximum number of \textit{total} inner iterations set to 15000.\footnote{Since ES-iALM does not involve a proximal term in its subproblem, the {\sc Ssn} method may not be applicable, as the nonsingularity of the generalized Jacobian cannot be guaranteed. Therefore, we employ FISTA as the subsolver in ES-iALM.} For Gurobi, we use its default termination conditions and set the termination tolerances as $10^{-6}$, aligning it with our tolerance. For dADMM, we initialize the penalty parameter as $\sigma_0=0.01\|C\|_{F}^{-1}$, and dynamically adjust it based on the primal-dual residuals, following the approach described in \cite[Section 5.1]{ylst2021fast}, to achieve superior performance. Moreover, we will terminate dADMM when it returns a point $(\bm{u}^k,\bm{v}^k,X^k)$ satisfying $\Delta_{\mathrm{res}}^{\rm qrot}(\bm{u}^k,\bm{v}^k,X^k)<10^{-6}$, or its number of iterations reaches 10000.

We follow \cite[Section 5.1]{yt2023inexact} to generate a random instance. First, we generate two discrete probability distributions $\big\{(\alpha_i, \bm{p}_i)\in \mathbb{R}_+\times\mathbb{R}^3 : i = 1,\cdots,m \big\}$ and $\big\{ (\beta_j, \bm{q}_j)\in \mathbb{R}_+\times\mathbb{R}^3 : j = 1,\cdots,n\big\}$. Here, $\bm{\alpha}:=(\alpha_1, \cdots, \alpha_{m})^{\top}$ and $\bm{\beta}:=(\beta_1, \cdots, \beta_{n})^{\top}$ are probabilities/weights, which are generated from the uniform distribution on the open interval $(0,1)$ and further normalized such that $\sum^{m}_i\alpha_i=\sum^{n}_j\beta_j=1$. Moreover, $\{\bm{p}_i\}$ and $\{\bm{q}_j\}$ are support points whose entries are drawn from a five-component multivariate Gaussian mixture distribution, with a mean vector $(-20,10,0,10,20)^{\top}$ and a variance vector $(5,5,5,5,5)^{\top}$, using randomly generated weights. Then, the cost matrix $C$ is generated by $c_{ij}=\|\bm{p}_i-\bm{q}_j\|^2$ for $1\leq i\leq m$ and $1\leq j\leq n$ and normalized by dividing (element-wise) by its maximal entry.

We then generate a set of random instances with $m = n\in \{1000,\,2000,\dots,10000\}$. For each $m$, we generate 10 instances with different random seeds, and present the average results of Gurobi, dADMM, ES-iALM, and ripALM in Table \ref{Table:vsGuADMM}, with $\lambda\in\{1,\,0.1\}$. We observe that both ripALM and Gurobi are able to solve the tested problems accurately, in the sense that the residual $\Delta_{\mathrm{res}}^{\rm qrot}$ falls below $10^{-6}$. In contrast, ES-iALM often fails to attain solutions of comparable accuracy, even after reaching its maximum iteration limit and incurring substantially higher computational costs. This is mainly due to the limited efficiency of the first-order subsolver used in ES-iALM. Moreover, dADMM performs even worse: despite consuming a similar amount of computational time as ES-iALM, it yields solutions of considerably lower quality. In addition, Gurobi can also be time-consuming and memory-consuming when problem size is large. For example, when $m=n=10000$, the resulting QP contains $10^8$ nonnegative variables and 20000 equality constraints. In this case, Gurobi is about $1.5\!\sim\!3$ times slower than our ripALM. 

\begin{table}[ht]
\caption{Comparisons among Gurobi, dADMM, ES-iALM, and ripALM for solving the QROT problem \eqref{QROTprob} with $\lambda\in\{1, 0.1\}$. In the table, ``$\Delta_{\mathrm{res}}^{\rm qrot}$" denotes the terminating $\Delta_{\mathrm{res}}^{\rm qrot}(\bm{u}^k,\bm{v}^k,X^k)$; ``\#" denotes the number of iterations (the total number of {\sc Ssn}/FISTA iterations is given in the bracket); ``\texttt{time}" denotes the computational time.  }\label{Table:vsGuADMM}
\renewcommand\arraystretch{1}
\centering \tabcolsep 4.5pt
\scalebox{0.76}{
\begin{tabular}{ccllcllcllcll}
\toprule
		& \multicolumn{3}{c}{Gurobi} & \multicolumn{3}{c}{dADMM} & \multicolumn{3}{c}{ES-iALM} & \multicolumn{3}{c}{ripALM} \\
\cmidrule(l){2-4} \cmidrule(l){5-7} \cmidrule(l){8-10} \cmidrule(l){11-13}
		$m=n$ & $\Delta_{\mathrm{res}}^{\rm qrot}$ & \# & \texttt{time} & $\Delta_{\mathrm{res}}^{\rm qrot}$ & \# & \texttt{time} & $\Delta_{\mathrm{res}}^{\rm qrot}$ & \# & \texttt{time} & $\Delta_{\mathrm{res}}^{\rm qrot}$ & \# & \texttt{time} \\
\midrule
\multicolumn{13}{c}{$\lambda=1$}   \\
\midrule
1000 & 2.32e-07 & 33 & 4.94  & 2.71e-06 & 10000 & 17.91  & 9.90e-07 & 16 (12600) & 15.68  & 4.98e-07 & 16 (37) & 1.81 \\ 
2000 & 3.16e-07 & 37 & 21.86  & 1.20e-05 & 10000 & 106.38  & 2.18e-06 & 16 (14965) & 107.76  & 6.70e-07 & 17 (52) & 8.80 \\ 
3000 & 3.55e-07 & 39 & 55.52  & 2.11e-05 & 10000 & 252.32  & 5.62e-06 & 16 (15000) & 253.64  & 5.40e-07 & 17 (56) & 20.68 \\ 
4000 & 6.86e-07 & 42 & 110.57  & 4.68e-05 & 10000 & 451.16  & 1.32e-05 & 15 (15000) & 454.48  & 6.12e-07 & 17 (66) & 39.90 \\ 
5000 & 6.09e-07 & 43 & 178.82  & 7.32e-05 & 10000 & 707.77  & 6.69e-06 & 15 (15000) & 709.99  & 6.61e-07 & 17 (67) & 64.28 \\ 
6000 & 4.79e-07 & 45 & 275.56  & 1.04e-04 & 10000 & 1016.18  & 2.01e-05 & 15 (15000) & 1019.01  & 6.08e-07 & 17 (71) & 96.38 \\ 
7000 & 4.69e-07 & 45 & 389.91  & 1.32e-04 & 10000 & 1382.35  & 2.54e-05 & 15 (15000) & 1385.37  & 5.23e-07 & 18 (74) & 136.98 \\ 
8000 & 5.47e-07 & 48 & 628.24  & 1.56e-04 & 10000 & 1804.32  & 1.05e-04 & 15 (15000) & 1811.85  & 4.73e-07 & 18 (81) & 190.85 \\ 
9000 & 6.09e-07 & 49 & 734.84  & 2.16e-04 & 10000 & 2289.86  & 2.52e-05 & 14 (15000) & 2293.79  & 6.00e-07 & 18 (82) & 244.17 \\ 
10000 & 4.36e-07 & 50 & 949.28  & 2.53e-04 & 10000 & 2816.87  & 3.02e-05 & 14 (15000) & 2815.60  & 6.08e-07 & 18 (85) & 308.99 \\ 
\midrule
\multicolumn{13}{c}{$\lambda=0.1$}   \\
\midrule
1000 & 4.74e-07 & 32 & 4.98  & 3.17e-05 & 10000 & 17.74  & 4.74e-05 & 16 (15000) & 17.77  & 6.37e-07 & 18 (66) & 2.12 \\ 
2000 & 3.61e-07 & 38 & 22.00  & 1.39e-04 & 10000 & 106.15  & 1.55e-05 & 15 (15000) & 105.57  & 5.96e-07 & 18 (78) & 11.08 \\ 
3000 & 4.69e-07 & 40 & 56.67  & 2.78e-04 & 10000 & 251.36  & 1.72e-05 & 15 (15000) & 254.22  & 6.41e-07 & 19 (84) & 29.38 \\ 
4000 & 4.20e-07 & 43 & 114.19  & 5.01e-04 & 10000 & 450.94  & 4.11e-05 & 15 (15000) & 454.35  & 6.37e-07 & 19 (93) & 61.01 \\ 
5000 & 5.39e-07 & 46 & 194.45  & 7.26e-04 & 10000 & 706.19  & 4.63e-05 & 15 (15000) & 716.31  & 6.58e-07 & 19 (101) & 109.55 \\ 
6000 & 1.33e-06 & 45 & 278.79  & 1.08e-03 & 10000 & 1017.11  & 6.80e-05 & 15 (15000) & 1030.13  & 6.47e-07 & 20 (109) & 171.26 \\ 
7000 & 8.21e-07 & 47 & 395.60  & 1.30e-03 & 10000 & 1382.91  & 1.08e-04 & 14 (15000) & 1399.50  & 5.65e-07 & 20 (117) & 255.10 \\ 
8000 & 7.15e-07 & 48 & 604.52  & 1.62e-03 & 10000 & 1803.87  & 1.53e-04 & 14 (15000) & 1824.82  & 6.29e-07 & 20 (119) & 333.22 \\ 
9000 & 6.89e-07 & 50 & 747.40  & 2.15e-03 & 10000 & 2283.73  & 2.01e-04 & 14 (15000) & 2309.05  & 5.80e-07 & 21 (129) & 472.94 \\ 
10000 & 1.10e-06 & 52 & 1026.56  & 2.43e-03 & 10000 & 2814.83  & 2.23e-04 & 14 (15000) & 2849.56  & 5.01e-07 & 21 (137) & 626.74 \\  
\bottomrule
\end{tabular}}
\end{table}

\subsection{Experiments on BPDN}\label{sec:bpdn-exp}

In this section, we evaluate the performance of different algorithms for solving the BPDN problem \eqref{eq-bpdn}. We use the datasets \texttt{housing7} and \texttt{mpg7} from the UCI data repository\footnote{Available at: \url{https://www.csie.ntu.edu.tw/~cjlin/libsvmtools/datasets/regression.html}} to generate the test instances $(D, \bm{b})$, where $D$ has dimensions $506\times77520$ for \texttt{housing7} and $392\times3432$ for \texttt{mpg7}. Here, \texttt{housing7} and \texttt{mpg7} are expanded versions of the original datasets, with the suffix `7' indicating that a polynomial of degree 7 is used to generate the basis functions; see \cite[Section 4.1]{lst2018highly} for more details. The parameter $\widehat{\kappa}$ in \eqref{eq-bpdn} is set to $\widehat{\kappa}=\delta\|\bm{b}\|$ with $\delta \in \{0.01, \,0.1\}$.

We first examine how different error criteria and tolerance parameters influence the numerical performance of different inexact pALMs. The tolerance parameters are chosen in the same way as in Table \ref{Table:Comp-iPALMs-QROT}. Table \ref{Table:Comp-iPALMs-BPDN} presents the results for \texttt{housing7}. As in the QROT experiments, the performance of all three algorithms is affected by the choice of tolerance parameters, while ripALM and ciPALM require less tuning effort as they involve only a single tolerance parameter $\rho\in[0,1)$. Moreover, ripALM consistently outperforms ciPALM in terms of the total number of {\sc Ssn} iterations and shows better efficiency and robustness.

\begin{table}[ht]
\caption{Comparisons among {\sc Snipal}, ciPALM and ripALM under different choices of tolerance parameters for solving BPDN problem \eqref{eq-bpdn} with $\widehat{\kappa} = \delta\|\bm{b}\|$ and $\delta\in\{0.1, 0.01\}$, using \texttt{housing7} from the UCI data repository. }\label{Table:Comp-iPALMs-BPDN}
	\centering \tabcolsep 6.3pt
	\scalebox{0.78}{
		\begin{tabular}{clclcl|cl|cl}
			\toprule
			\multicolumn{10}{c}{$\delta = 0.1$} \\
			\cmidrule(lr){1-10}
			\multicolumn{6}{c}{{\sc Snipal}} & \multicolumn{2}{|c}{{ciPALM}} & \multicolumn{2}{|c}{{ripALM}}\\
			\cmidrule(lr){1-6} \cmidrule(lr){7-8} \cmidrule(lr){9-10}
			$(\varepsilon_{0}\!=\!\delta_{0},p,q)$ & \# (\#) & $(\varepsilon_{0}\!=\!\delta_{0},p,q)$ & \# (\#) & $(\varepsilon_{0}\!=\!\delta_{0},p,q)$ & \# (\#) & $\rho$ & \# (\#) & $\rho$ & \# (\#) \\
			\cmidrule(lr){1-6} \cmidrule(lr){7-8} \cmidrule(lr){9-10}
(1,\,1.1,\,1.1) & 19 (71) & (0.1,\,1.1,\,1.1) & 19 (78) & (0.01,\,1.1,\,1.1) & 19 (84) & 0.999 & 19 (149) & 0.999 & 19 (67)  \\ 
(1,\,1.1,\,2.1) & 19 (75) & (0.1,\,1.1,\,2.1) & 19 (84) & (0.01,\,1.1,\,2.1) & 19 (86) & 0.99 & 19 (149) & 0.99 & 19 (67)  \\ 
(1,\,1.1,\,3.1) & 19 (78) & (0.1,\,1.1,\,3.1) & 19 (85) & (0.01,\,1.1,\,3.1) & 19 (94) & 0.9 & 19 (149) & 0.9 & 19 (64)  \\ 
(1,\,2.1,\,1.1) & 19 (75) & (0.1,\,2.1,\,1.1) & 19 (87) & (0.01,\,2.1,\,1.1) & 19 (85) & 0.7 & 19 (132) & 0.7 & 19 (64)  \\ 
(1,\,2.1,\,2.1) & 19 (77) & (0.1,\,2.1,\,2.1) & 19 (84) & (0.01,\,2.1,\,2.1) & 19 (86) & 0.5 & 19 (132) & 0.5 & 19 (64)  \\ 
(1,\,2.1,\,3.1) & 19 (80) & (0.1,\,2.1,\,3.1) & 19 (90) & (0.01,\,2.1,\,3.1) & 19 (94) & 0.3 & 19 (94) & 0.3 & 19 (67)  \\ 
(1,\,3.1,\,1.1) & 19 (83) & (0.1,\,3.1,\,1.1) & 19 (84) & (0.01,\,3.1,\,1.1) & 19 (91) & 0.1 & 19 (87) & 0.1 & 19 (76)  \\ 
(1,\,3.1,\,2.1) & 19 (83) & (0.1,\,3.1,\,2.1) & 19 (84) & (0.01,\,3.1,\,2.1) & 19 (92) & 0.05 & 18 (76) & 0.05 & 19 (74)  \\ 
(1,\,3.1,\,3.1) & 19 (84) & (0.1,\,3.1,\,3.1) & 19 (90) & (0.01,\,3.1,\,3.1) & 19 (93) & 0.01 & 19 (76) & 0.01 & 19 (73)  \\ 
			\bottomrule  
			\vspace{-2mm} \\
			\toprule
			\multicolumn{10}{c}{$\delta = 0.01$} \\
			\cmidrule(lr){1-10}
			\multicolumn{6}{c}{{\sc Snipal}} & \multicolumn{2}{|c}{{ciPALM}} & \multicolumn{2}{|c}{{ripALM}}\\
			\cmidrule(lr){1-6} \cmidrule(lr){7-8} \cmidrule(lr){9-10}
			$(\varepsilon_{0}\!=\!\delta_{0},p,q)$ & \# (\#) & $(\varepsilon_{0}\!=\!\delta_{0},p,q)$ & \# (\#) & $(\varepsilon_{0}\!=\!\delta_{0},p,q)$ & \# (\#) & $\rho$ & \# (\#) & $\rho$ & \# (\#) \\
			\cmidrule(lr){1-6} \cmidrule(lr){7-8} \cmidrule(lr){9-10}
(1,\,1.1,\,1.1) & 24 (156) & (0.1,\,1.1,\,1.1) & 24 (159) & (0.01,\,1.1,\,1.1) & 24 (163) & 0.999 & 24 (537) & 0.999 & 24 (137)  \\ 
(1,\,1.1,\,2.1) & 24 (157) & (0.1,\,1.1,\,2.1) & 24 (161) & (0.01,\,1.1,\,2.1) & 24 (170) & 0.99 & 24 (537) & 0.99 & 24 (137)  \\ 
(1,\,1.1,\,3.1) & 24 (158) & (0.1,\,1.1,\,3.1) & 24 (165) & (0.01,\,1.1,\,3.1) & 24 (170) & 0.9 & 24 (537) & 0.9 & 24 (137)  \\ 
(1,\,2.1,\,1.1) & 24 (158) & (0.1,\,2.1,\,1.1) & 24 (170) & (0.01,\,2.1,\,1.1) & 24 (176) & 0.7 & 24 (310) & 0.7 & 24 (138)  \\ 
(1,\,2.1,\,2.1) & 24 (158) & (0.1,\,2.1,\,2.1) & 24 (170) & (0.01,\,2.1,\,2.1) & 24 (176) & 0.5 & 24 (256) & 0.5 & 24 (137)  \\ 
(1,\,2.1,\,3.1) & 24 (161) & (0.1,\,2.1,\,3.1) & 24 (172) & (0.01,\,2.1,\,3.1) & 24 (176) & 0.3 & 24 (265) & 0.3 & 24 (137)  \\ 
(1,\,3.1,\,1.1) & 24 (164) & (0.1,\,3.1,\,1.1) & 24 (176) & (0.01,\,3.1,\,1.1) & 24 (185) & 0.1 & 24 (214) & 0.1 & 24 (143)  \\ 
(1,\,3.1,\,2.1) & 24 (164) & (0.1,\,3.1,\,2.1) & 24 (176) & (0.01,\,3.1,\,2.1) & 24 (185) & 0.05 & 24 (222) & 0.05 & 24 (144)  \\ 
(1,\,3.1,\,3.1) & 24 (164) & (0.1,\,3.1,\,3.1) & 24 (176) & (0.01,\,3.1,\,3.1) & 24 (185) & 0.01 & 24 (169) & 0.01 & 24 (145)  \\
			\bottomrule
	\end{tabular}}
\end{table}

We next compare ripALM with Mosek (version 11.0.27), SPGL1 (version 2.1), and dADMM (see Appendix \ref{sec:BPDN-dadmm}). For ripALM, we set $\rho=0.99$ based on results in Table \ref{Table:Comp-iPALMs-BPDN}. For dADMM, we initialize the penalty parameter as $\sigma=1$, and dynamically update it based on the primal-dual residuals, following the strategy in \cite[Section 5.1]{ylst2021fast}. We terminate dADMM when it returns a point $(\bm{s}^k, \bm{t}^k, \bm{y}^k)$ satisfying $\Delta_{\mathrm{res}}^{\rm bpdn}(\bm{s}^k, \bm{t}^k, \bm{y}^k) < 10^{-6}$ or reaches 10000 iterations. For SPGL1, we use its default termination conditions, set the tolerance as $10^{-6}$, (aligning it with our tolerance), and set the maximum number of iterations as 10000. For Mosek, we use its default termination conditions, but adopt two tolerance levels: a tolerance of $10^{-6}$ (denoted as Mosek (1e-6)) to match our tolerance and a tighter tolerance of $10^{-10}$ (denoted as Mosek (1e-10)) to generate a high-precision solution to use as a benchmark. In addition, to evaluate the solution quality, we compute the relative feasibility residual and  normalized objective gap, defined respectively as:
\begin{equation*}
\texttt{feas}(\bm{s}):= \frac{\max\left\{\|D\bm{s} - \bm{b}\|-\widehat{\kappa}, \,0\right\}}{1+\|\bm{b}\|}, \quad \text{ and } \quad \texttt{nobj}(\bm{s}):= \frac{\big|\|\bm{s}\|_1-\|\bm{s}^*\|_1\big|}{1 + \|\bm{s}^*\|_1},
\end{equation*}
where $\bm{s}^*$ denote the solution returned by Mosek (1e-10). 

Table \ref{tab:BPDN-UCI} reports the results for \texttt{housing7} and \texttt{mpg7}. These results further demonstrate the effectiveness of ripALM in comparison with the other representative algorithms.

\begin{table}[ht]
\caption{Comparisons among Mosek, SPGL1, dADMM, and ripALM for solving the BPDN problem \eqref{eq-bpdn} with $\widehat{\kappa} = \delta\|\bm{b}\|$ and $\delta\in\{0.1, 0.01\}$, using \texttt{housing7} and \texttt{mpg7} from the UCI data repository. In the table, ``\texttt{nobj}" denotes the normalized objective gap; ``\texttt{feas}" denotes the relative feasibility residual; ``\#" denotes the number of iterations (the total number of {\sc Ssn} iterations is given in the bracket); ``\texttt{time}" denotes the computational time.}\label{tab:BPDN-UCI}%
\renewcommand\arraystretch{1}
\centering \tabcolsep 5.4pt
\scalebox{0.86}{
\begin{tabular}{llllllllll}
\toprule
& & \texttt{nobj} & \texttt{feas} & \texttt{iters} & \texttt{time} & \texttt{nobj} & \texttt{feas} & \texttt{iters} & \texttt{time}  \\
\cmidrule(l){3-6} \cmidrule(l){7-10} \texttt{dataset} & \texttt{method} & \multicolumn{4}{c}{$\delta = 0.1$} & \multicolumn{4}{c}{$\delta = 0.01$} \\
\midrule
\multirow{5}{*}{\texttt{housing7}} & Mosek (1e-10) & 0 & 2.40e-12 & 27 & 33.21 & 0 & 2.37e-11 & 36 & 43.57 \\ 
& Mosek (1e-6) & 4.31e-05 & 1.44e-08 & 19 & 22.75 & 1.83e-05 & 4.60e-08 & 22 & 24.49 \\ 
& SPGL1 & 9.21e-03 & 0 & 1539 & 48.46 & 1.76e-01 & 1.42e-02 & 10000 & 237.72 \\ 
& dADMM & 2.00e-06 & 0 & 3151 & 70.09 & 4.98e-06 & 0 & 10000 & 221.44 \\ 
& ripALM & 2.08e-09 & 0 & 18 (48) & {\bf 1.65} & 6.58e-06 & 0 & 22 (74) & {\bf 2.11} \\ 
\midrule
\multirow{5}{*}{\texttt{mpg7}} & Mosek (1e-10) & 0 & 1.52e-13 & 16 & 0.62 & 0 & 1.49e-09 & 26 & 0.99 \\ 
& Mosek (1e-6) & 8.92e-06 & 1.82e-08 & 13 & 0.51 & 1.45e-07 & 2.98e-08 & 17 & 0.60 \\ 
& SPGL1 & 5.35e-03 & 0 & 443 & 0.21 & 8.82e-01 & 3.85e-02 & 10000 & 2.14 \\ 
& dADMM & 2.33e-06 & 0 & 2001 & 0.70 & 1.48e-04 & 0 & 10000 & 3.24 \\ 
& ripALM & 2.06e-07 & 1.58e-10 & 18 (23) & {\bf 0.02} & 1.03e-06 & 1.73e-08 & 30 (126) & {\bf 0.45} \\ 
\bottomrule 
\end{tabular}}
\end{table}

\section{Conclusions}\label{sec:conclusions}

In this paper, we developed a relative-type inexact proximal augmented Lagrangian method (ripALM) for solving a class of linearly constrained convex optimization problems. To the best of our knowledge, the proposed ripALM is the first relative-type inexact version of the vanilla pALM with provable convergence guarantees. By employing a suitable relative-type error criterion, it simplifies implementation and parameter tuning, compared to the classical absolute-type inexact pALM. We conducted a thorough convergence analysis and demonstrated the competitive efficiency of ripALM through numerical experiments on solving quadratically regularized optimal transport problems and basis pursuit denoising problems. Since our model covers a wide range of applications, future work may explore the potential of applying ripALM to other practical problems. 

\section*{Acknowledgements}

The research of Lei Yang is supported in part by the National Key Research and Development Program of China under grant 2023YFB3001704 and the National Natural Science Foundation of China under grant 12301411. The research of Kim-Chuan Toh is supported by the Ministry of Education, Singapore, under its Academic Research Fund Tier 2 grant (MOE-T2EP20224-0017).

\appendix

\section{Missing Proofs in Section \ref{sec:conver}}

\subsection{Proof of Theorem \ref{thm:convergence}}\label{sec:proofconver}

\begin{proof}
\textit{Statement (i)}. First, recall that $\ell(\bm{y}, \bm{x}) = -\bm{b}^\top \bm{y}+\langle \bm{x}, A^\top \bm{y}\rangle -f(\bm{x})$, which is convex in $\bm{y}$ and concave in $\bm{x}$. We see that  $\partial\ell(\bm{y},\bm{x})=\left\{A\bm{x}-\bm{b}\right\}\times\left\{v-A^{\top}\bm{y} \mid v\in \partial f(\bm{x})\right\}$. By using the relations in \eqref{ripALM-inexcond} and \eqref{ripALM_xwupdate}, along with some manipulations, we can obtain the following results for all $k \geq 0$:
\begin{numcases}{}
\big(\Delta^{k+1} - \tau_{k}\sigma_{k}^{-1}(\bm{y}^{k+1}-\bm{y}^{k}), \,\sigma_{k}^{-1}(\bm{x}^{k}-\bm{x}^{k+1})\big) \in \partial \ell(\bm{y}^{k+1}, \bm{x}^{k+1}), \label{optcond1} \\[5pt]
2\big|\langle\bm{w}^k-\bm{y}^{k+1}, \,\sigma_{k}\Delta^{k+1}\rangle\big| + \big\|\sigma_{k}\Delta^{k+1}\big\|^2 \leq \rho\big(\big\|\bm{x}^{k+1}-\bm{x}^{k}\big\|^2 + \tau_{k}\big\|\bm{y}^{k+1}-\bm{y}^k\big\|^2\big),  \label{optcond2} \\[5pt]
\bm{w}^{k+1} = \bm{w}^k - \sigma_{k}\Delta^{k+1}.   \label{optcond3}
\end{numcases}

Let $(\bm{y}^*,\bm{x}^*)\in\mathbb{R}^M\times\mathbb{R}^N$ be an arbitrary saddle point of $\ell$ and hence $(\bm{0},\bm{0})\in\partial\ell(\bm{y}^*,\bm{x}^*)$. For all $k\geq0$,
\begin{equation*}
\begin{aligned}
\|\bm{x}^k-\bm{x}^*\|^2
&= \|\bm{x}^k-\bm{x}^{k+1}+\bm{x}^{k+1}-\bm{x}^*\|^2 \\
&= \|\bm{x}^k-\bm{x}^{k+1}\|^2 + 2\langle\bm{x}^k-\bm{x}^{k+1},\,\bm{x}^{k+1}-\bm{x}^*\rangle + \|\bm{x}^{k+1}-\bm{x}^*\|^2.
\end{aligned}
\end{equation*}
By letting $\bm{\xi}^{k+1}:=\sigma_{k}^{-1}(\bm{x}^{k}-\bm{x}^{k+1})$, the above equation can be reformulated as
\begin{equation}\label{eq:xxineq}
\|\bm{x}^{k+1}-\bm{x}^*\|^2 = \|\bm{x}^k-\bm{x}^*\|^2 - 2\sigma_{k}\langle\bm{x}^{k+1}-\bm{x}^*,\,\bm{\xi}^{k+1}\rangle - \|\bm{x}^{k+1}-\bm{x}^{k}\|^2.
\end{equation}
Then, using the relation $\bm{w}^{k+1} = \bm{w}^k - \sigma_{k}\Delta^{k+1}$ (by \eqref{optcond3}), we see that
\begin{equation}\label{wwineq}
\begin{aligned}
 \|\bm{w}^{k+1}-\bm{y}^*\|^2
= &\; \|\bm{w}^k-\sigma_{k}\Delta^{k+1}-\bm{y}^*\|^2  \\
= &\; \|\bm{w}^k-\bm{y}^*\|^2 - 2\langle\bm{w}^k-\bm{y}^*,\,\sigma_{k}\Delta^{k+1}\rangle + \|\sigma_{k}\Delta^{k+1}\|^2 \\
= &\; \|\bm{w}^k-\bm{y}^*\|^2 - 2\langle\bm{w}^k-\bm{y}^{k+1},\,\sigma_{k}\Delta^{k+1}\rangle + \|\sigma_{k}\Delta^{k+1}\|^2 \\
&\; - 2\sigma_{k}\langle\bm{y}^{k+1}-\bm{y}^*,\,\bm{\theta}^{k+1}\rangle
- 2\tau_{k}\langle\bm{y}^{k+1}-\bm{y}^*,\,\bm{y}^{k+1}-\bm{y}^k\rangle.
\end{aligned}
\end{equation}
where $\bm{\theta}^{k+1}:=\Delta^{k+1}-\tau_{k}\sigma_{k}^{-1}(\bm{y}^{k+1}-\bm{y}^{k})$.
Similarly,
\begin{equation}\label{eq:yyineq}
\begin{aligned}
\tau_{k}\|\bm{y}^{k+1}-\bm{y}^*\|^2
= \tau_{k}\|\bm{y}^{k}-\bm{y}^*\|^2 - \tau_{k}\|\bm{y}^{k+1}-\bm{y}^k\|^2 + 2\tau_{k}\langle\bm{y}^{k+1}-\bm{y}^*,\,\bm{y}^{k+1}-\bm{y}^k\rangle.
\end{aligned}
\end{equation}
By summing \eqref{eq:xxineq}, \eqref{wwineq} and \eqref{eq:yyineq}, we have that
\begin{equation}\label{eq:xwyineq}
\begin{aligned}
&\; \|\bm{x}^{k+1}-\bm{x}^*\|^2 + \|\bm{w}^{k+1}-\bm{y}^*\|^2 + \tau_{k}\|\bm{y}^{k+1}-\bm{y}^*\|^2  \\
= &\; \|\bm{x}^{k}-\bm{x}^*\|^2 + \|\bm{w}^{k}-\bm{y}^*\|^2 + \tau_{k}\|\bm{y}^k-\bm{y}^*\|^2  \\
&\; - 2\sigma_{k}\left(\langle\bm{y}^{k+1}-\bm{y}^*,\,\bm{\theta}^{k+1}\rangle
+ \langle\bm{x}^{k+1}-\bm{x}^*,\,\bm{\xi}^{k+1}\rangle\right)
- 2\langle\bm{w}^k-\bm{y}^{k+1},\,\sigma_{k}\Delta^{k+1}\rangle\\
&\; + \|\sigma_{k}\Delta^{k+1}\|^2
- \|\bm{x}^{k+1}-\bm{x}^{k}\|^2 - \tau_{k}\|\bm{y}^{k+1}-\bm{y}^k\|^2.
\end{aligned}
\end{equation}
Note from \eqref{optcond1} that
\begin{equation}\label{optcond1re}
\big(\bm{\theta}^{k+1}, \,\bm{\xi}^{k+1}\big) \in \partial \ell(\bm{y}^{k+1}, \bm{x}^{k+1}),
\end{equation}
which, together with $\bm{0}\in\partial\ell(\bm{y}^*,\bm{x}^*)$ and the monotonicity of $\partial\ell$, yields
\begin{equation*}
\langle\bm{y}^{k+1}-\bm{y}^*,\,\bm{\theta}^{k+1}\rangle
+ \langle\bm{x}^{k+1}-\bm{x}^*,\,\bm{\xi}^{k+1}\rangle \geq 0.
\end{equation*}
Moreover, by using \eqref{optcond2}, we see that
\begin{equation*}\label{errbdineq}
\begin{aligned}
- 2\langle\bm{w}^k-\bm{y}^{k+1},\,\sigma_{k}\Delta^{k+1}\rangle + \|\sigma_{k}\Delta^{k+1}\|^2
&\leq 2\big|\langle\bm{w}^k-\bm{y}^{k+1},\,\sigma_{k}\Delta^{k+1}\rangle\big| + \|\sigma_{k}\Delta^{k+1}\|^2  \\
&\leq \rho\left(\|\bm{x}^{k+1}-\bm{x}^{k}\|^2
+ \tau_{k}\|\bm{y}^{k+1}-\bm{y}^k\|^2\right).
\end{aligned}
\end{equation*}
Substituting the above two inequalities into \eqref{eq:xwyineq}, we obtain a key inequality for the subsequent convergence analysis:
\begin{equation}\label{eq:recursion_of_xyw}
\begin{aligned}
&\; \|\bm{x}^{k+1}-\bm{x}^*\|^2 + \|\bm{w}^{k+1}-\bm{y}^*\|^2 + \tau_{k}\|\bm{y}^{k+1}-\bm{y}^*\|^2\\
\leq &\; \|\bm{x}^{k}-\bm{x}^*\|^2 + \|\bm{w}^{k}-\bm{y}^*\|^2 + \tau_{k}\|\bm{y}^k-\bm{y}^*\|^2 \\
&\; - (1-\rho)\left(\|\bm{x}^{k+1}-\bm{x}^{k}\|^2 + \tau_{k}\|\bm{y}^{k+1}-\bm{y}^k\|^2\right).
\end{aligned}
\end{equation}
The inequality \eqref{eq:recursion_of_xyw}, together with $\rho\in[0,1)$ and $\tau_{k+1}\leq(1+\nu_{k})\tau_{k}$ with $\nu_k\geq0$ and $\sum\nu_i<\infty$ for all $k \geq 0$, implies that
\begin{equation}\label{eq:relaxed_monotone}
\begin{aligned}
&\; \|\bm{x}^{k+1}-\bm{x}^*\|^2 + \|\bm{w}^{k+1}-\bm{y}^*\|^2 + \tau_{k+1}\|\bm{y}^{k+1}-\bm{y}^*\|^2\\
\leq &\; (1+\nu_{k})\left(\|\bm{x}^{k}-\bm{x}^*\|^2 + \|\bm{w}^{k}-\bm{y}^*\|^2 + \tau_{k}\|\bm{y}^k-\bm{y}^*\|^2\right).
\end{aligned}
\end{equation}
Since $\{\nu_{k}\}$ is a non-negative summable sequence, it then follows from \cite[Lemma 2 in Section 2.2 on Page 44]{polyak1987introduction} that the sequence $\left\{\|\bm{x}^{k}-\bm{x}^{*}\|^2+\|\bm{w}^{k}-\bm{y}^{*}\|^{2} + \tau_{k}\|\bm{y}^k-\bm{y}^{*}\|^2\right\}$ is convergent (but do not necessarily converge to $0$). This together with $\tau_k\geq\tau_{\min}>0$ implies that all sequences $\{\bm{x}^k\}$, $\{\bm{w}^k\}$, $\{\bm{y}^{k}\}$ are bounded.

\vspace{2mm}
\textit{Statement (ii)}. Using \eqref{eq:recursion_of_xyw} again with $\tau_{k+1}\leq(1+\nu_{k})\tau_{k}$ with $\nu_k\geq0$ and $\sum\nu_i<\infty$ for all $k \geq 0$, we have that
\begin{equation}\label{addineq1}
\begin{aligned}
0 \leq &\; (1-\rho){(1+\nu_k)}\left(\|\bm{x}^{k+1}-\bm{x}^{k}\|^2 + \tau_{k}\|\bm{y}^{k+1}-\bm{y}^k\|^2\right) \\
\leq &\; (1+\nu_{k})\left(\|\bm{x}^{k}-\bm{x}^*\|^2 + \|\bm{w}^{k}-\bm{y}^*\|^2 + \tau_{k}\|\bm{y}^k-\bm{y}^*\|^2\right) \\
&\quad - \left(\|\bm{x}^{k+1}-\bm{x}^*\|^2 + \|\bm{w}^{k+1}-\bm{y}^*\|^2 + \tau_{k+1}\|\bm{y}^{k+1}-\bm{y}^*\|^2\right).
\end{aligned}
\end{equation}
Since $\left\{\|\bm{x}^{k}-\bm{x}^{*}\|^2+\|\bm{w}^{k}-\bm{y}^{*}\|^{2}+ \tau_{k}\|\bm{y}^k-\bm{y}^{*}\|^2\right\}$ is convergent, $\nu_{k}\to0$ (due to $\nu_k\geq0$ and $\sum\nu_i<\infty$) and $\rho\in[0,1)$, it then follows from \eqref{addineq1} that
\begin{equation}\label{eq:conv_succhg_xy}
\lim\limits_{k\to\infty}\|\bm{x}^{k+1}-\bm{x}^{k}\| = 0
\quad \mathrm{and} \quad \lim\limits_{k\to\infty}\|\sqrt{\tau_{k}}(\bm{y}^{k+1}-\bm{y}^k)\| = 0.
\end{equation}
Note that both sequences $\{\sigma_k\}$ and $\{\tau_k\}$ are bounded away from 0. Thus, we further have that $\lim\limits_{k\to\infty}\bm{\xi}^{k+1} (:= \sigma_{k}^{-1}(\bm{x}^{k}-\bm{x}^{k+1}))=\bm{0}$ and $\lim\limits_{k\to\infty}\|\bm{y}^{k+1}-\bm{y}^k\|=\bm{0}$. Moreover, using \eqref{eq:conv_succhg_xy} together with \eqref{optcond2} implies that
\begin{equation*}
\lim\limits_{k\to\infty}~|\langle\bm{w}^k-\bm{y}^{k+1},\,\sigma_k\Delta^{k+1}\rangle|
= 0
\quad \mathrm{and} \quad
\lim\limits_{k\to\infty}~\|\sigma_k\Delta^{k+1}\|^2 = 0.
\end{equation*}
Since $\{\sigma_{k}\}$ is bounded away from $0$, we then obtain that $\lim\limits_{k\to\infty}\langle\bm{w}^k-\bm{y}^{k+1},\,\Delta^{k+1}\rangle=0$ and $\lim\limits_{k\to\infty}\Delta^{k+1}=\bm{0}$. Finally, recall again that $\tau_{k+1}\leq(1+\nu_{k})\tau_{k}$ with $\nu_k\geq0$ and $\sum\nu_i<\infty$ for all $k\geq0$. Thus, $\tau_k$ must be bounded from above and hence $\tau_{k}\sigma_{k}^{-1}$ is also bounded from above. Consequently, we can obtain that $\lim\limits_{k\to\infty}\tau_{k}\sigma_{k}^{-1}(\bm{y}^{k+1}-\bm{y}^{k})=\bm{0}$ and hence $\lim\limits_{k\to\infty}\bm{\theta}^{k+1}=\bm{0}$.

\vspace{2mm}
\textit{Statement (iii)}. We first study the limit of $\{F\big(\bm{\theta}^{k+1}, \,\bm{x}^{k+1}\big)\}$. Using relations \eqref{optcond1re} and \eqref{subdiffrelation}, we have that $(\bm{y}^{k+1}, \bm{\xi}^{k+1})\in\partial F\big(\bm{\theta}^{k+1}, \,\bm{x}^{k+1}\big)$. Then, by the concavity of $F$, it holds that, for all $k\geq 0$,
\begin{equation*}
F(\bm{0},\,\bm{x}^*) \leq F\big(\bm{\theta}^{k+1}, \,\bm{x}^{k+1}\big)
+ \langle\bm{y}^{k+1}, \,\bm{\theta}^{k+1}\rangle
+ \langle\bm{\xi}^{k+1}, \,\bm{x}^{k+1}-\bm{x}^{*}\rangle.
\end{equation*}
Since $\lim\limits_{k\to\infty}\bm{\theta}^{k+1}=\bm{0}$, $\lim\limits_{k\to\infty}\bm{\xi}^{k+1} = \bm{0}$, and the sequences $\{\bm{x}^k\}$ and $\{\bm{y}^{k}\}$ are bounded, we can obtain from the above inequality that
\begin{equation}\label{liminfres}
\liminf\limits_{k\to\infty}\,F\big(\bm{\theta}^{k+1}, \,\bm{x}^{k+1}\big)
\geq F(\bm{0},\,\bm{x}^*).
\end{equation}
On the other hand, since $\{\bm{x}^k\}$ is bounded, it has at least one accumulation point. Suppose that $\bm{x}^{\infty}$ is an accumulation point and $\{\bm{x}^{k_i}\}$ is a convergent subsequence such that $\lim\limits_{i\to\infty}\bm{x}^{k_i} = \bm{x}^{\infty}$. Since $\lim\limits_{k\to\infty}\|\bm{x}^{k+1} - \bm{x}^{k}\| = 0$, we also have that $\lim\limits_{i\to\infty}\bm{x}^{k_{i}+1} = \bm{x}^{\infty}$. Thus, by passing to a further subsequence if necessary, we may assume without loss of generality that the subsequence  $\big\{F\big(\bm{\theta}^{k_{i}+1}, \,\bm{x}^{k_{i}+1}\big)\big\}$ satisfies
\begin{equation*}
\lim\limits_{i\to\infty}F\big(\bm{\theta}^{k_{i}+1},\bm{x}^{k_{i}+1}\big)
=\limsup\limits_{k\to\infty}\,F\big(\bm{\theta}^{k+1}, \,\bm{x}^{k+1}\big).
\end{equation*}
Note that $F$ is closed upper semicontinuous concave (see, for example, \cite[Theorem 7]{r1974conjugate}), and thus $\mathrm{dom}\,F$ is closed. This together with $(\bm{0},\,\bm{x}^{\infty})=\lim\limits_{i\to\infty}\big(\bm{\theta}^{k_{i}+1}, \,\bm{x}^{k_{i}+1}\big)$ induces that $(\bm{0},\,\bm{x}^{\infty})\in\mathrm{dom}\,F$. Then, we see that
\begin{equation*}
\begin{aligned}
&\;\;\;\;\;F(\bm{0},\,\bm{x}^*) \\
&\geq F(\bm{0},\,\bm{x}^{\infty}) && \quad \mbox{(since $\bm{x}^*$ is optimal for problem \eqref{eq:para-proorgdual})} \\
&= F\bigg(\lim\limits_{i\to\infty}\bm{\theta}^{k_{i}+1},\,\lim\limits_{i\to\infty}\bm{x}^{k_{i}+1}\bigg) && \quad \mbox{(since $\lim\limits_{k\to\infty}\bm{\theta}^{k+1}=\bm{0}$ and $\lim\limits_{i\to\infty}\bm{x}^{k_{i}+1} = \bm{x}^{\infty}$)} \\
&\geq \limsup\limits_{i\to\infty}\,F\big(\bm{\theta}^{k_{i}+1}, \,\bm{x}^{k_{i}+1}\big)
&& \quad \mbox{(since $F$ is upper semicontinuous)} \\
&= \limsup\limits_{k\to\infty}\,F\big(\bm{\theta}^{k+1}, \,\bm{x}^{k+1}\big).  && \quad \mbox{(by the choice of subsequence $\{\bm{x}^{k_{i}+1}\}$)}
\end{aligned}
\end{equation*}
This together with \eqref{liminfres} implies that
\begin{equation}\label{eq:limofF}
\lim\limits_{k\to\infty}\,F\big(\bm{\theta}^{k+1}, \,\bm{x}^{k+1}\big)
= F(\bm{0},\,\bm{x}^*).
\end{equation}

We next study the limit of $\{G(\bm{y}^{k+1}, \bm{\xi}^{k+1})\}$. Since $-F$ and $G$ are convex conjugate and $\big(\bm{\theta}^{k+1},\,\bm{x}^{k+1}\big)\in\partial G(\bm{y}^{k+1},\,\bm{\xi}^{k+1})$, we can get  the following equality by using the Fenchel equality (see, for example, \cite[Theorem 23.5]{r1970convex}):
\begin{equation*}
G(\bm{y}^{k+1}, \,\bm{\xi}^{k+1})
= F\big(\bm{\theta}^{k+1}, \,\bm{x}^{k+1}\big)
+ \langle\bm{\theta}^{k+1}, \,\bm{y}^{k+1}\rangle
+ \langle\bm{x}^{k+1}, \,\bm{\xi}^{k+1}\rangle.
\end{equation*}
Since $\lim\limits_{k\to\infty}\bm{\theta}^{k+1}=\bm{0}$, $\lim\limits_{k\to\infty}\bm{\xi}^{k+1} = \bm{0}$, and the sequences $\{\bm{x}^k\}$ and $\{\bm{y}^{k}\}$ are bounded, we obtain that
\begin{equation*}
\lim\limits_{k\to\infty}\,G(\bm{y}^{k+1}, \bm{\xi}^{k+1})
= F(\bm{0},\,\bm{x}^*) = G(\bm{y}^*,\,\bm{0}).
\end{equation*}
This proves statement (iii).

\vspace{2mm}
\textit{Statement (iv)}. We first prove that any accumulation point of $\{\bm{y}^k\}$ is an optimal solution of problem \eqref{eq:para-proorg}. Since $\{\bm{y}^k\}$ is bounded by statement (i), the sequence $\{\bm{y}^k\}$ has at least one accumulation point. Suppose that $\bm{y}^{\infty}$ is an accumulation point and $\{\bm{y}^{k_j}\}$ is a convergent subsequence such that $\lim\limits_{j\to\infty}\bm{y}^{k_j} = \bm{y}^{\infty}$. Since $\lim\limits_{k\to\infty}\|\bm{y}^{k+1}-\bm{y}^{k}\| = 0$, we also have that $\lim\limits_{j\to\infty}\bm{y}^{k_j+1} = \bm{y}^{\infty}$. Then, using the fact that $G$ is lower semicontinuous and convex, and $\lim\limits_{k\to\infty}\bm{\xi}^{k+1}=\bm{0}$, we obtain that
\begin{equation*}
G(\bm{y}^{\infty},\bm{0})
= G(\lim_{j\to\infty}\bm{y}^{k_j+1},\,\lim_{j\to\infty}\bm{\xi}^{k_j+1})
\leq \liminf_{j\to\infty}~G(\bm{y}^{k_j+1},\,\bm{\xi}^{k_j+1})
= G(\bm{y}^*,\bm{0}).
\end{equation*}
This implies that $\bm{y}^{\infty}$ is an optimal solution of problem \eqref{eq:para-proorg}. Similarly, using the upper semicontinuity of $F$ and analogous manipulations, we can prove that any accumulation point of $\{\bm{x}^k\}$ is an optimal solution of problem \eqref{eq:para-proorgdual}. This proves statement (iv).

\vspace{2mm}
\textit{Statement (v)}. We next prove that the whole sequence $\{\bm{x}^{k}\}$ is convergent. Let
\begin{equation*}
D_{\tau_k}\left((\bm{w}^k,\bm{y}^k), \,\mathcal{Y}^*\right)
:= \inf\limits_{\bm{y}^*\in\mathcal{Y}^*}\left\{\|\bm{w}^{k}-\bm{y}^*\|^2 + \tau_{k}\|\bm{y}^{k}-\bm{y}^*\|^2\right\},
\end{equation*}
and define that
\begin{equation*}
\phi := \liminf_{k\to\infty}\,D_{\tau_k}\left((\bm{w}^k,\bm{y}^k), \,\mathcal{Y}^*\right),
\end{equation*}
where $\mathcal{Y}^*$ is the solution set of problem \eqref{eq:para-proorg} (i.e., problem \eqref{eq:maindual}). Since $\{\bm{w}^k\}$, $\{\bm{y}^{k}\}$ and $\{\tau_k\}$ are bounded, we see that $0 \leq \phi < \infty$ and there exists a subsequence $\{(\bm{w}^{k_j},\bm{y}^{k_j},\tau_{k_j})\}$ such that
\begin{equation*}
\lim_{j\to\infty}\,D_{\tau_{k_j}}\left((\bm{w}^{k_j},\bm{y}^{k_j}), \,\mathcal{Y}^*\right) = \phi.
\end{equation*}
Then, by passing to a further subsequence if necessary, we may also assume without loss of generality that the subsequence $\{\bm{x}^{k_j}\}\subseteq\{\bm{x}^k\}$ converges to some accumulation point $\bm{x}^{\infty}$, which, in view of statement (iv), belongs to $\mathcal{X}^*$ (the optimal solution set of \eqref{eq:para-proorgdual}). Thus, for such $\bm{x}^{\infty}$ and any $\bm{y}^*\in\mathcal{Y}^*$, using \eqref{eq:relaxed_monotone} with some manipulations, we obtain, for all $k > k_j$,
\begin{equation*}
\begin{aligned}
&\,\|\bm{x}^{k}-\bm{x}^{\infty}\|^2 + \|\bm{w}^{k}-\bm{y}^*\|^2 + \tau_{k}\|\bm{y}^{k}-\bm{y}^*\|^2 \\
\leq & \; \left(\prod_{i=k_j}^{k-1}\big(1+\nu_i\big)\right)
\Big(\|\bm{x}^{k_j}-\bm{x}^{\infty}\|^2 + \|\bm{w}^{k_j}-\bm{y}^*\|^2
+ \tau_{k_j}\|\bm{y}^{k_{j}}-\bm{y}^*\|^2\Big).
\end{aligned}
\end{equation*}
Since $0 \leq \phi \leq \liminf\limits_{k\to\infty}\left\{\|\bm{w}^{k}-\bm{y}^*\|^2 + \tau_{k}\|\bm{y}^{k}-\bm{y}^*\|^2\right\}$ for any $\bm{y}^*\in\mathcal{Y}^*$, passing to the limit superior when $k \to \infty$ on the both sides of this inequality, we obtain that, for any $\bm{y}^*\in\mathcal{Y}^*$,
\begin{equation*}
\begin{aligned}
&\; \phi + \limsup\limits_{k\to\infty}\left\{\|\bm{x}^{k}-\bm{x}^{\infty}\|^2\right\} \\
\leq &\; \limsup\limits_{k\to\infty}\left\{\|\bm{x}^{k}-\bm{x}^{\infty}\|^2 + \|\bm{w}^{k}-\bm{y}^*\|^2 + \tau_{k}\|\bm{y}^{k}-\bm{y}^*\|^2\right\} \\
\leq &\; \left(\prod_{i=k_j}^{\infty}(1+\nu_i)\right)
\left(\|\bm{x}^{k_j}-\bm{x}^{\infty}\|^2 + \|\bm{w}^{k_j}-\bm{y}^*\|^2
+ \tau_{k_j}\|\bm{y}^{k_j}-\bm{y}^*\|^2\right), \quad \forall j \geq 0.
\end{aligned}
\end{equation*}
Taking the infimum in $\bm{y}^*\in\mathcal{Y}^*$ on the right-hand side of the last inequality, we have that
\begin{equation*}
\begin{aligned}
&\limsup\limits_{k\to\infty}\left\{\|\bm{x}^{k}-\bm{x}^{\infty}\|^2\right\} \\
\leq &\; \left(\prod_{i=k_j}^{\infty}(1+\nu_i)\right)
\|\bm{x}^{k_j}-\bm{x}^{\infty}\|^2
+ \left(\prod_{i=k_j}^{\infty}(1+\nu_i)\right)
D_{\tau_{k_j}}\left((\bm{w}^{k_j},\bm{y}^{k_j}), \,\mathcal{Y}\right)
- \phi, \quad \forall j \geq 0.
\end{aligned}
\end{equation*}
Since $\ln\left(\prod_{i=k_{j}}^{\infty}(1+\nu_{i})\right) = \sum_{i=k_{j}}^{\infty}\ln(1+\nu_{i}) \leq \sum_{i=k_{j}}^{\infty}\nu_{i}$ and $\lim\limits_{j\to\infty}\sum_{i=k_{j}}^{\infty}\nu_{i} = 0$ (due to the summability of $\{\nu_k\}$), we see that $\lim\limits_{j\to\infty}\prod_{i=k_{j}}^{\infty}(1+\nu_{i}) = 1$. Using this fact, we can observe that the right-hand side of the above inequality converges to 0 as $j\to\infty$. Then, we conclude that $\lim\limits_{k\to\infty} \bm{x}^{k} = \bm{x}^{\infty}$, which completes the proof.
\end{proof}

\subsection{Proof of Theorem \ref{thm:Q-linear-rate}}\label{sec:proof_linear_rate}

\begin{proof}
\textit{Statement (i)}. For the sake of clarity, we will present our proof in three steps.

\vspace{2mm}
\textbf{Step I.} Let $(\bm{y}^*,\bm{x}^*)\in\mathbb{R}^M\times\mathbb{R}^N$ be an arbitrary saddle point of $\ell$. Similar to the proof of Theorem \ref{thm:convergence}(i), we combine \eqref{eq:xxineq} with \eqref{eq:yyineq} to obtain that
\begin{equation*}
\begin{aligned}
&\; \|\bm{x}^{k+1}-\bm{x}^*\|^2 + \tau_{k}\|\bm{y}^{k+1}-\bm{y}^*\|^2  \\
= &\; \|\bm{x}^{k}-\bm{x}^*\|^2 + \tau_{k}\|\bm{y}^k-\bm{y}^*\|^2
\; - 2\sigma_{k}\left(\langle\bm{y}^{k+1}-\bm{y}^*,\,\bm{\theta}^{k+1}\rangle
+ \langle\bm{x}^{k+1}-\bm{x}^{*},\,\bm{\xi}^{k+1}\rangle\right) \\
&\; + 2\sigma_{k}\langle\bm{y}^{k+1}-\bm{y}^{*}, \Delta^{k+1}\rangle
- \|\bm{x}^{k+1}-\bm{x}^{k}\|^2 - \tau_{k}\|\bm{y}^{k+1}-\bm{y}^k\|^2.
\end{aligned}
\end{equation*}
Since $(\bm{0},\bm{0})\in\partial\ell(\bm{y}^*,\bm{x}^*)$ and $\left(\bm{\theta}^{k+1}, \,\bm{\xi}^{k+1}\right) \in \partial \ell(\bm{y}^{k+1}, \bm{x}^{k+1})$, it then follows from the monotonicity of $\partial\ell$ that
\begin{equation*}
\langle\bm{y}^{k+1}-\bm{y}^*,\,\bm{\theta}^{k+1}\rangle
+ \langle\bm{x}^{k+1}-\bm{x}^{*},\,\bm{\xi}^{k+1}\rangle \geq 0.
\end{equation*}
Thus, we conclude that
\begin{equation}\label{eq:succ-xychg}
\begin{aligned}
&\;\left\|\begin{matrix}
\sqrt{\tau_{k}}(\bm{y}^{k}-\bm{y}^{*}) \\
\bm{x}^{k} - \bm{x}^{*}
\end{matrix}\right\|^{2}
- \left\|\begin{matrix}
\sqrt{\tau_{k}}(\bm{y}^{k+1}-\bm{y}^{*}) \\
\bm{x}^{k+1} - \bm{x}^{*}
\end{matrix} \right\|^{2} \\
\geq &\; \left\|\begin{matrix}
\sqrt{\tau_{k}}(\bm{y}^{k+1}-\bm{y}^{k}) \\
\bm{x}^{k+1} - \bm{x}^{k}
\end{matrix} \right\|^2
- 2\sigma_{k}\|\bm{y}^{k+1}-\bm{y}^{*}\|\|\Delta^{k+1}\|.
\end{aligned}
\end{equation}
Define the sequences $\{\overline{\bm{y}}^{k}\}\subseteq\mathbb{R}^M$ and $\{\overline{\bm{x}}^{k}\}\subseteq\mathbb{R}^N$ as follows:
\begin{equation*}\label{eq:def-of-proj}
\overline{\bm{y}}^{k} := \Pi_{\mathcal{Y}^*}(\bm{y}^{k})
\quad \text{and} \quad
\overline{\bm{x}}^{k} := \Pi_{\mathcal{X}^*}(\bm{x}^{k}),\quad \forall\; k\geq 0,
\end{equation*}
where $\mathcal{Y}^*$ is the solution set of problem \eqref{eq:para-proorg} (i.e., problem \eqref{eq:maindual}), $\mathcal{X}^*$ is the solution set of problem \eqref{eq:para-proorgdual} (i.e., problem \eqref{eq:mainprob}), and $\Pi_{\mathcal{Y}^*}(\bm{y}^{k})$ (resp. $\Pi_{\mathcal{X}^*}(\bm{x}^{k})$) denotes the projection of $\bm{y}^{k}$ (resp. $\bm{x}^k$) onto set $\mathcal{Y}^*$ (resp. $\mathcal{X}^*$). Since \eqref{eq:succ-xychg} holds for any $\bm{y}^{*}\in\mathcal{Y}^*$ and $\bm{x}^{*}\in\mathcal{X}^*$, we can replace $\bm{y}^{*}$ and $\bm{x}^{*}$ with $\overline{\bm{y}}^{k}$ and $\overline{\bm{x}}^{k}$, respectively, to obtain that
\begin{align*}
&\; \left\|\begin{matrix}
\sqrt{\tau_{k}}(\bm{y}^{k}-\overline{\bm{y}}^{k}) \\
\bm{x}^{k} - \overline{\bm{x}}^{k}
\end{matrix} \right\|^{2}
- \left\|\begin{matrix}
\sqrt{\tau_{k}}(\bm{y}^{k+1}-\overline{\bm{y}}^{k}) \\
\bm{x}^{k+1} - \overline{\bm{x}}^{k}
\end{matrix} \right\|^{2} \\
\geq &\; \left\|\begin{matrix}
\sqrt{\tau_{k}}(\bm{y}^{k+1}-\bm{y}^{k}) \\
\bm{x}^{k+1} - \bm{x}^{k}
\end{matrix} \right\|^2 - 2\sigma_{k}\|\bm{y}^{k+1}-\overline{\bm{y}}^{k}\|\|\Delta^{k+1}\|.
\end{align*}
Moreover, from the definitions of $\overline{\bm{y}}^{k}$ and $\overline{\bm{x}}^{k}$, we have that
$\left\|\bm{y}^{k+1} - \overline{\bm{y}}^{k+1}\right\| \leq \left\|\bm{y}^{k+1} - \overline{\bm{y}}^{k}\right\|$ and $\left\|\bm{x}^{k+1} - \overline{\bm{x}}^{k+1}\right\| \leq \left\|\bm{x}^{k+1} - \overline{\bm{x}}^{k}\right\|$. These, together with the above inequality, yield that
\begin{equation}\label{eq:succ-xychgbar}
\begin{aligned}
&\; \left\|\begin{matrix}
\sqrt{\tau_{k}}(\bm{y}^{k}-\overline{\bm{y}}^{k}) \\
\bm{x}^{k} - \overline{\bm{x}}^{k}
\end{matrix} \right\|^{2}
- \left\|\begin{matrix}
\sqrt{\tau_{k}}(\bm{y}^{k+1}-\overline{\bm{y}}^{k+1}) \\
\bm{x}^{k+1} - \overline{\bm{x}}^{k+1}
\end{matrix} \right\|^{2} \\
\geq &\; \left\|\begin{matrix}
\sqrt{\tau_{k}}(\bm{y}^{k+1}-\bm{y}^{k}) \\
\bm{x}^{k+1} - \bm{x}^{k}
\end{matrix} \right\|^2
- 2\sigma_{k}\|\bm{y}^{k+1}-\overline{\bm{y}}^{k}\|\|\Delta^{k+1}\|.
\end{aligned}
\end{equation}

\vspace{2mm}
\textbf{Step II.} We next derive an upper bound for $\|\bm{y}^{k+1}-\overline{\bm{y}}^{k}\|\|\Delta^{k+1}\|$. On the one hand, we have from \eqref{optcond2} that
\begin{equation*}
\sigma_{k}^{2}\|\Delta^{k+1}\|^2
\leq \rho\left(\|\bm{x}^{k+1}-\bm{x}^{k}\|^2
+ \tau_{k}\|\bm{y}^{k+1}-\bm{y}^k\|^2\right),
\end{equation*}
which implies that
\begin{equation}\label{eq:delta-upbd-xy}
\|\Delta^{k+1}\|
\leq \frac{\sqrt{\rho}}{\sigma_{k}}
\left\|\begin{matrix}
\sqrt{\tau_{k}}(\bm{y}^{k+1}-\bm{y}^{k}) \\
\bm{x}^{k+1} - \bm{x}^{k}
\end{matrix} \right\|.
\end{equation}
On the other hand, we see that
\begin{equation}\label{eq:upbd-yybark}
\begin{aligned}
\|\bm{y}^{k+1}-\overline{\bm{y}}^{k}\|
&\leq \|\bm{y}^{k+1}-\overline{\bm{y}}^{k+1}\|
+ \|\overline{\bm{y}}^{k+1}-\overline{\bm{y}}^{k}\|
\leq \|\bm{y}^{k+1}-\overline{\bm{y}}^{k+1}\|
+ \|\bm{y}^{k+1}-{\bm{y}}^{k}\| \\
&\leq \|\bm{y}^{k+1}-\overline{\bm{y}}^{k+1}\|
+ \frac{1}{\sqrt{\tau_{k}}}\left\|\begin{matrix}
\sqrt{\tau_{k}}(\bm{y}^{k+1}-{\bm{y}}^{k}) \\
\bm{x}^{k+1} - {\bm{x}}^{k}
\end{matrix}\right\|,
\end{aligned}
\end{equation}
where the second inequality follows from the non-expansiveness of the projection operator $\Pi_{\mathcal{Y}^*}(\cdot)$. Moreover, since $\{\bm{y}^{k}\}$ and $\{\bm{x}^{k}\}$ are bounded (by Theorem \ref{thm:convergence}(i)), there must exist a positive scalar $r$ such that
\begin{equation*}
\mathrm{dist}\big((\bm{y}^{k},\bm{x}^{k}),\,(\partial\ell)^{-1}(\bm{0},\bm{0})\big)\leq r, \quad \forall \,k \geq 0.
\end{equation*}
Thus, we apply Assumption \ref{asp:error-bound-Li} with this $r$ and know that, there exists a $\kappa>0$ such that
\begin{equation*}
\operatorname{dist}\left((\bm{y}^{k+1},\bm{x}^{k+1}), \,\mathcal{Y}^*\times\mathcal{X}^*\right)
\leq \kappa\operatorname{dist}\left((\bm{0},\bm{0}), \,\partial\ell(\bm{y}^{k+1},\bm{x}^{k+1})\right)
\leq \kappa\left\|(\bm{\theta}^{k+1}, \,\bm{\xi}^{k+1})\right\|,
\end{equation*}
where the last inequality is due to \eqref{optcond1re} (i.e., $\left(\bm{\theta}^{k+1}, \,\bm{\xi}^{k+1}\right) \in \partial \ell(\bm{y}^{k+1}, \bm{x}^{k+1})$) with $\bm{\theta}^{k+1} := \Delta^{k+1} - \tau_{k}\sigma_{k}^{-1}(\bm{y}^{k+1}-\bm{y}^{k})$ and $\bm{\xi}^{k+1} := \sigma_{k}^{-1}(\bm{x}^{k}-\bm{x}^{k+1})$. This inequality further implies that
\begin{equation}\label{eq:upbd-ydist}
\begin{aligned}
 & \|\bm{y}^{k+1} - \overline{\bm{y}}^{k+1}\|
\leq \; \sqrt{\|\bm{y}^{k+1} - \overline{\bm{y}}^{k+1}\|^2
+ \|\bm{x}^{k+1} - \overline{\bm{x}}^{k+1}\|^2} \\
\leq &\; \kappa\left\|\begin{matrix}
\Delta^{k+1} - \tau_{k}\sigma_{k}^{-1}(\bm{y}^{k+1} - \bm{y}^{k}) \\
\sigma_{k}^{-1}(\bm{x}^{k} - \bm{x}^{k+1})
\end{matrix}\right\|
\;\leq \; \kappa\left(\|\Delta^{k+1}\|
+ \frac{1}{\sigma_{k}}\left\|\begin{matrix}
\tau_{k}(\bm{y}^{k+1} - \bm{y}^{k}) \\
\bm{x}^{k+1} - \bm{x}^{k}
\end{matrix}\right\|\right)  \\
\leq &\; \kappa\left(\|\Delta^{k+1}\|
+ \frac{\sqrt{\overline{\tau}_{k}}}{\sigma_k}
\left\|\begin{matrix}
\sqrt{\tau_k}(\bm{y}^{k+1} - \bm{y}^{k}) \\
\bm{x}^{k+1} - \bm{x}^{k}
\end{matrix}\right\|\right) \\
\leq  &\; \frac{\kappa\left(\sqrt{\rho}+\sqrt{\overline{\tau}_{k}}\right)}{\sigma_{k}}
\left\|\begin{matrix}
\sqrt{\tau_{k}}(\bm{y}^{k+1}-\bm{y}^{k}) \\
\bm{x}^{k+1} - \bm{x}^{k}
\end{matrix}\right\|,
\end{aligned}
\end{equation}
where $\overline{\tau}_{k}$ is defined as $\overline{\tau}_{k}:=\max\left\{1, \tau_{k}\right\}$ and the last inequality follows from \eqref{eq:delta-upbd-xy}. Now, combining \eqref{eq:delta-upbd-xy}, \eqref{eq:upbd-yybark} and \eqref{eq:upbd-ydist}, with some manipulations, we can obtain that
\begin{equation*}
\|\bm{y}^{k+1}-\overline{\bm{y}}^{k}\|\|\Delta^{k+1}\|
\leq \left(\frac{\kappa\sqrt{\tau_{k}}
\left(\rho +\sqrt{\rho\overline{\tau}_{k}}\right)
+\sigma_{k}\sqrt{\rho}}{\sigma_{k}^{2}\sqrt{\tau_{k}}}\right)
\left\|\begin{matrix}
\sqrt{\tau_{k}}(\bm{y}^{k+1}-\bm{y}^{k}) \\
\bm{x}^{k+1} - \bm{x}^{k}
\end{matrix}\right\|^{2}.
\end{equation*}
Then, substituting this inequality into \eqref{eq:succ-xychgbar} yields that
\begin{equation}\label{eq:dist-xybark}
\begin{aligned}
&\; \left\|\begin{matrix}
\sqrt{\tau_{k}}(\bm{y}^{k}-\overline{\bm{y}}^{k}) \\
\bm{x}^{k} - \overline{\bm{x}}^{k}
\end{matrix} \right\|^{2}
-
\left\|\begin{matrix}
\sqrt{\tau_{k}}(\bm{y}^{k+1}-\overline{\bm{y}}^{k+1}) \\
\bm{x}^{k+1} - \overline{\bm{x}}^{k+1}
\end{matrix}\right\|^{2} \\
\geq &\; \left(1 - \frac{2\kappa\sqrt{\tau_{k}}\left(\rho+\sqrt{\rho\overline{\tau}_{k}}\right)
+2\sigma_{k}\sqrt{\rho}}{\sigma_{k}\sqrt{\tau_{k}}}\right)
\left\|\begin{matrix}
\sqrt{\tau_{k}}(\bm{y}^{k+1}-{\bm{y}}^{k}) \\
\bm{x}^{k+1} - {\bm{x}}^{k}
\end{matrix}\right\|^{2}.
\end{aligned}
\end{equation}

\vspace{2mm}
\textbf{Step III.} In the following, we will establish the convergence rate based on \eqref{eq:dist-xybark}. First, by recalling the conditions on $\{\tau_{k}\}$: $\tau_{k}\geq\tau_{\min}>0$ and $\tau_{k+1}\leq(1+\nu_{k})\tau_{k}$ with $\nu_{k}\geq0$ and $\sum_{k=0}^{\infty}\nu_{k}<+\infty$, we know that there exists $\tau_{\max}:=\left(\prod_{k=0}^{\infty}(1+\nu_{k})\right)\tau_{0}$ such that $0<\tau_{\min} \leq \tau_{k} \leq \tau_{\max} < +\infty$ for all $k\geq 0$. This together with condition \eqref{para-conds} implies that there exists a positive integer $k_{0}$ such that
\begin{equation*}
\sqrt{\tau_{k}} - 2\sqrt{\rho} > 0
\quad \text{and} \quad
\sigma_{k} > c\cdot\frac{2\kappa\sqrt{\tau_{k}}\left(\rho+\sqrt{\rho\overline{\tau}_{k}}\right)}{\sqrt{\tau_{k}} - 2\sqrt{\rho}}, \quad \forall\,k \geq k_{0},
\end{equation*}
where $c>1$. Hence, one can verify that
\begin{equation}\label{para-condsre}
\left(1 - \frac{2\kappa\sqrt{\tau_{k}}\left(\rho+\sqrt{\rho\overline{\tau}_{k}}\right)
+2\sigma_{k}\sqrt{\rho}}{\sigma_{k}\sqrt{\tau_{k}}}\right)
> \widetilde{c}:=\Big(\frac{c-1}{c}\Big)\cdot\frac{\sqrt{\tau_{\min}} - 2\sqrt{\rho}}{\sqrt{\tau_{\min}}}
> 0,
\quad \forall\,k \geq k_{0},
\end{equation}
which means that the factor in the right-hand side of \eqref{eq:dist-xybark} will be positive when $k\geq k_{0}$. On the other hand, using \eqref{eq:upbd-ydist} again, we deduce that
\begin{equation*}
\begin{aligned}
\left\|\begin{matrix}
\sqrt{\tau_{k}}(\bm{y}^{k+1}-\bm{y}^{k}) \\
\bm{x}^{k+1} - \bm{x}^{k}
\end{matrix}\right\|^{2}
&\geq \frac{\sigma_{k}^{2}}{\kappa^2\left(\sqrt{\rho}+\sqrt{\overline{\tau}_{k}}\right)^2}
\left\|\begin{matrix}
\bm{y}^{k+1}-\overline{\bm{y}}^{k+1} \\
\bm{x}^{k+1} - \overline{\bm{x}}^{k+1}
\end{matrix}\right\|^{2}  \\
&\geq
\frac{\sigma_{k}^{2}}{\kappa^2\left(\sqrt{\rho}+\sqrt{\overline{\tau}_{k}}\right)^2\overline{\tau}_{k}}
\left\|\begin{matrix}
\sqrt{\tau_{k}}(\bm{y}^{k+1}-\overline{\bm{y}}^{k+1}) \\
\bm{x}^{k+1} - \overline{\bm{x}}^{k+1}
\end{matrix}\right\|^{2}.
\end{aligned}
\end{equation*}
This, together with \eqref{eq:dist-xybark} and \eqref{para-condsre}, yields that
\begin{equation}\label{eq:linear_rate1}
\left\|\begin{matrix}
\sqrt{\tau_{k}}(\bm{y}^{k}-\overline{\bm{y}}^{k}) \\
\bm{x}^{k} - \overline{\bm{x}}^{k}
\end{matrix}\right\|^{2}
\geq \left(1 + \gamma_{k} \right)\left\|\begin{matrix}
\sqrt{\tau_{k}}(\bm{y}^{k+1}-\overline{\bm{y}}^{k+1}) \\
\bm{x}^{k+1} - \overline{\bm{x}}^{k+1}
\end{matrix} \right\|^{2}, \quad \forall \,k \geq k_{0},
\end{equation}
where
\begin{equation}\label{gammalbd}
\begin{aligned}
\gamma_{k} \,:=&\, \left(1 - \frac{2\kappa\sqrt{\tau_{k}}\left(\rho+\sqrt{\rho\overline{\tau}_{k}}\right)
+2\sigma_{k}\sqrt{\rho}}{\sigma_{k}\sqrt{\tau_{k}}}\right)
\frac{\sigma_{k}^{2}}{\kappa^2\left(\sqrt{\rho}+\sqrt{\overline{\tau}_{k}}\right)^2\overline{\tau}_{k}} \\
\geq \,&\, \widetilde{c}\cdot\frac{\sigma_{k}^{2}}{\kappa^2\left(\sqrt{\rho}+\sqrt{\overline{\tau}_{\max}}\right)^2\overline{\tau}_{\max}} \\
\geq \,&\, \gamma_{\min}:= \frac{\widetilde{c}\,\sigma_{\min}^{2}}{\kappa^2\left(\sqrt{\rho}+\sqrt{\overline{\tau}_{\max}}\right)^2\overline{\tau}_{\max}} > 0, \quad \forall\,k \geq k_{0}.
\end{aligned}
\end{equation}
Let $\Lambda^{k} := \operatorname{Diag}(\tau_{k}I_{M},I_{N})$. Since $\{\tau_{k}\}$ is bounded away from 0 and satisfies that $(1+\nu_{k})\tau_{k}\geq\tau_{k+1}$, we have that $(1+\nu_{k})\Lambda^{k} \succeq \Lambda^{k+1}\succ 0$. Then, one can obtain from \eqref{eq:linear_rate1} that
\begin{equation*}
(1+\nu_{k})\left\|\begin{matrix}
\bm{y}^{k}-\overline{\bm{y}}^{k} \\
\bm{x}^{k} - \overline{\bm{x}}^{k}
\end{matrix}\right\|_{\Lambda^{k}}^{2}
\geq
\left(1+\gamma_{k}\right)
\left\|\begin{matrix}	
\bm{y}^{k+1}-\overline{\bm{y}}^{k+1} \\
\bm{x}^{k+1} - \overline{\bm{x}}^{k+1}
\end{matrix}\right\|_{\Lambda^{k+1}}^{2},
\end{equation*}
which readily implies that
\begin{equation*}
\mathrm{dist}_{\Lambda^{k+1}}\left((\bm{y}^{k+1},\bm{x}^{k+1}), \,(\partial\ell)^{-1}(\bm{0},\bm{0})\right) \leq \mu_{k}\, \mathrm{dist}_{\Lambda^{k}}\left((\bm{y}^{k},\bm{x}^{k}), \,(\partial\ell)^{-1}(\bm{0},\bm{0})\right),
\end{equation*}
where $\mu_{k}:=\sqrt{\frac{1+\nu_{k}}{1+\gamma_{k}}}$. Since $\nu_{k}\to0$ and $\gamma_{k}\geq\gamma_{\min}> 0$ for all $k \blue{\geq} k_{0}$, one can verify that $\limsup\limits_{k\to\infty}\,\{\mu_{k}\}<1$. Thus, we obtain the desired results in statement (i).

\vspace{2mm}
\textit{Statement (ii)}. Using \eqref{eq:dist-xybark} and \eqref{para-condsre} again, we see that
\begin{equation*}
\widetilde{c}\,\tau_{\min}\|\bm{y}^{k+1}-{\bm{y}}^{k}\|^2
\leq
\widetilde{c}\left\|\begin{matrix}
\sqrt{\tau_{k}}(\bm{y}^{k+1}-{\bm{y}}^{k}) \\
\bm{x}^{k+1} - {\bm{x}}^{k}
\end{matrix}\right\|^{2}
\leq \mathrm{dist}_{\Lambda^{k}}^{\blue{2}}\left((\bm{y}^{k},\bm{x}^{k}), \,(\partial\ell)^{-1}(\bm{0},\bm{0})\right),
\end{equation*}
for any $ k \geq k_{0}$. Using this inequality and the fact that the sequence $\left\{\mathrm{dist}_{\Lambda^{k}}\left((\bm{y}^{k},\bm{x}^{k}), \,(\partial\ell)^{-1}(\bm{0},\bm{0})\right)\right\}$ is asymptotically Q-(super)linear convergent, we can conclude that there exist a positive integer $k_1$, $0<\beta<1$ and $C>0$ such that
\begin{equation*}
\|\bm{y}^{k+1}-{\bm{y}}^{k}\| \leq C \beta^{k}, \quad \forall \,k \geq k_{1},
\end{equation*}
which further implies that $\sum_{k=0}^{\infty}\|\bm{y}^{k+1}-{\bm{y}}^{k}\|<\infty$. Consequently, $\{\bm{y}^k\}$ is a Cauchy sequence and hence convergent. Therefore, the proof is completed.
\end{proof}

\section{A dual ADMM for QROT}\label{sec:ADMM}

In this section, we present the alternating direction method of multipliers (ADMM, see, e.g. \cite{boyd2011distributed,gabay1976dual}) for solving the dual problem \eqref{eq:dualQROT}, which can be reformulated as:
\begin{equation}\label{eq:maindual-admm2}
\begin{aligned}	&\min\limits_{\bm{u}\in\mathbb{R}^m,\,\bm{v}\in\mathbb{R}^n,\,W\in\mathbb{R}^{m\times n}}
\Big\{ f_{\texttt{q}}^{*}(W) - \bm{\alpha}^{\top}\bm{u} - \bm{\beta}^{\top}\bm{v}
\;\mid\;
\bm{u}\bm{1}_{n}^{\top} + \bm{1}_{m}\bm{v}^{\top} = W\Big\}.
\end{aligned}
\end{equation}
Given a penalty parameter $\sigma > 0$, the augmented Lagrangian function of \eqref{eq:maindual-admm2} is
\begin{equation*}
\begin{aligned}
&\; \mathcal{L}_{\sigma}\left(\bm{u},\bm{v},W;X\right)  \\
:= &\; f_{\texttt{q}}^{*}(W)
- \bm{\alpha}^{\top}\bm{u}
- \bm{\beta}^{\top}\bm{v}
+ \big\langle X, \,\bm{u}\bm{1}_{n}^{\top}+\bm{1}_{m}\bm{v}^{\top}-W\big\rangle
+ \frac{\sigma}{2}\big\| \bm{u}\bm{1}_{n}^{\top}
+ \bm{1}_{m}\bm{v}^{\top}-W\big\|_F^{2}.
\end{aligned}
\end{equation*}
Then, the ADMM for solving \eqref{eq:maindual-admm2} can be described as in Algorithm \ref{alg:warm_start_ADMM}.

\begin{algorithm}[ht]
\caption{ADMM for solving \eqref{eq:maindual-admm2}}\label{alg:warm_start_ADMM}
\begin{algorithmic}
\STATE \textbf{Input:} a penalty parameter $\sigma>0$, and initializations $\bm{u}^{0}\in\mathbb{R}^m$, $\bm{v}^{0}\in\mathbb{R}^n$,
		$W^{0}, X^0\in\mathbb{R}^{m\times n}$. Set $k = 0$.
\WHILE{the termination criterion is not met,}
\STATE \textbf{Step 1.} Compute 	
		\begin{equation*}
		\big(\bm{u}^{k+1},\,\bm{v}^{k+1}\big)
        =\arg\min\limits_{\bm{u},\bm{v}}~\mathcal{L}_{\sigma}\big(\bm{u},\bm{v},W^k;X^k\big).
		\end{equation*}
		
\STATE \textbf{Step 2.} Compute
		\begin{equation*}
			W^{k+1} = \arg\min\limits_{W}
			~\mathcal{L}_{\sigma}\big(\bm{u}^{k+1},\bm{v}^{k+1},W;X^k\big).
		\end{equation*}
		
\STATE \textbf{Step 3.} Set $X^{k+1} = X^{k} + \tau\sigma\left(\bm{u}^{k+1}\bm{1}_{n}^{\top}
		+ \bm{1}_{m}(\bm{v}^{k+1})^{\top}
		- W^{k+1}\right)$, where $\tau \in \left(0, \frac{1+\sqrt{5}}{2}\right)$ is the dual step-size that is typically set to $1.618$.
		
\vspace{2mm}
\STATE \textbf{Step 4.} Set $k \leftarrow k+1$ and go to \textbf{Step 1}.
\ENDWHILE
		
\STATE \textbf{Output:} $\big(\bm{u}^{k},\bm{v}^{k},W^{k},X^{k}\big)$.
\end{algorithmic}
\end{algorithm}

Both subproblems in ADMM can be solved efficiently. Specifically, $\big(\bm{u}^{k+1},\,\bm{v}^{k+1}\big)$ can be obtained by solving the following unconstrained convex minimization problem:
\begin{equation}\label{eq:subprobuv_ADMM}	
	\min\limits_{\bm{u},\bm{v}}~
	h_k(\bm{u},\bm{v}) := - \bm{\alpha}^{\top}\bm{u} - \bm{\beta}^{\top}\bm{v} + \frac{\sigma}{2}\left\|\bm{u}\bm{1}_{n}^{\top} + \bm{1}_{m}\bm{v}^{\top} - S^k\right\|_{F}^{2},
\end{equation}
where $S^k:=W^k-\sigma^{-1}X^{k}$. From the first-order optimality conditions of \eqref{eq:subprobuv_ADMM}, we see that solving problem \eqref{eq:subprobuv_ADMM} is equivalent to solving the equation $\nabla h_k(\bm{u},\bm{v}) = \bm{0}$. This, in turn, reduces to solving the following linear system
\begin{numcases}{}
	n\bm{u} + (\bm{1}_{n}^{\top}\bm{v})\bm{1}_{m} = \sigma^{-1}\bm{\alpha} + S^k\bm{1}_{n}, \label{eq:lsAAT_u} \\[3pt]
	(\bm{1}_{m}^{\top}\bm{u})\bm{1}_{n} + m\bm{v} = \sigma^{-1}\bm{\beta} + (S^k)^{\top}\bm{1}_{m}. \label{eq:lsAAT_v}
\end{numcases}
With some algebraic manipulations, it is not difficult to show that
\begin{eqnarray*}
	\bm{u}^{*}(t) &=& \frac{\sigma^{-1}\bm{\alpha} + S^k\bm{1}_{n}}{n} + t\,\bm{1}_{m}, \quad \forall\,t\in\mathbb{R},
	\\
	\bm{v}^{*}(t) &=& \frac{\sigma^{-1}\bm{\beta} + (S^k)^{\top}\bm{1}_{m}}{m} - \frac{\bm{1}_{m}^{\top}\bm{u}^*(t)}{m}\bm{1}_{n}, \quad \forall\,t\in\mathbb{R}.
\end{eqnarray*}
solves the above linear system.
Thus, we obtain a solution pair $(\bm{u}^{*}(t),\bm{v}^{*}(t))$ with any $t\in\mathbb{R}$. It can be routinely shown that $(\bm{u}^{*}(t),\bm{v}^{*}(t))$ satisfies the linear system \eqref{eq:lsAAT_u} and \eqref{eq:lsAAT_v}, and therefore solves problem \eqref{eq:subprobuv_ADMM}. On the other hand, $W^{k+1}$ can be obtained by computing the proximal operator of the function $\sigma^{-1} f_{\texttt{q}}^{*}$, i.e.,
\begin{equation*}
	W^{k+1} ~:= ~\mathtt{prox}_{\sigma^{-1} f_{\texttt{q}}^{*}}\big(Q^{k}\big)  
	=~ \left\{\begin{aligned}
		&C - \Pi_{\mathbb{R}^{m\times n}_{+}} \big(C-Q^{k}\big), && \lambda=0, \\[3pt]
		&Q^k - (1+\lambda\sigma)^{-1}\Pi_{\mathbb{R}^{m\times n}_{+}}(Q^k-C), && \lambda>0,
	\end{aligned}\right.
\end{equation*}
where $Q^k:=\bm{u}^{k+1}\bm{1}_{n}^{\top} + \bm{1}_{m}({\bm{v}}^{k+1})^{\top} + \sigma^{-1}X^{k}$.

\section{An iBPGM for QROT}\label{sec:iBPGM}

In this section, we briefly discuss how to employ an inexact Bregman proximal gradient method (iBPGM) with Sinkhorn's algorithm as a subsolver for solving the QROT problem \eqref{QROTprob}. We refer readers to \cite[Section 5]{yt2023inexact} for more details. Specifically, the iBPGM with the entropy kernel function for solving \eqref{QROTprob} can be given as follows: let $X^0>0$ and $\phi(X):=\sum_{ij}x_{ij}(\log x_{ij}-1)$, at the $k$-th iteration, compute
\begin{equation}\label{iBPGM-subprob}
	\begin{aligned}
		X^{k+1}~\approx~
		&\min\limits_{X} \Big\{ \langle C+\lambda X^k, \,X\rangle + \mu_k\mathcal{D}_{\phi}(X, \,X^k)   \; \mid \;  X\bm{1}_n = \bm{\alpha}, ~~X^{\top}\bm{1}_m = \bm{\beta}
		\Big\},
	\end{aligned}
\end{equation}
where $\mu_k\geq\lambda$ is a positive proximal parameter, and $\mathcal{D}_{\phi}(U, \,V)$ denotes the Bregman distance between $U$ and $V$ associated with the kernel function $\phi$ which is defined as $\mathcal{D}_{\phi}(U, \,V) := \phi(U) - \phi(V) - \langle \nabla\phi(V), \,U - V \rangle$. Problem \eqref{iBPGM-subprob} can be rewritten as
\begin{equation}\label{iBPGM-subprob-reform}
	\min\limits_{X} \Big\{\langle M^k, \,X\rangle + \mu_k\,{\textstyle\sum_{ij}}x_{ij}(\log x_{ij}-1) \;\mid\;   X\bm{1}_n=\bm{\alpha}, ~~ X^{\top}\bm{1}_m = \bm{\beta}\Big\},
\end{equation}
where $M^k:=C + \lambda X^k - \mu_k\log X^k$. Note that problem \eqref{iBPGM-subprob-reform} has the same form as the entropic regularized optimal transport problem and hence can be readily solved by the popular Sinkhorn's algorithm; see \cite[Section 4.2]{pc2019computational} for more details. Specifically, let $\Xi^k:=e^{-M^k/\mu_k}$. Then, given an initial positive vector $\bm{v}^{k,0}$, the iterative scheme is given by
\begin{equation}\label{sinkalg}
	\bm{u}^{k,t} = \bm{\alpha} ./ \big(\Xi^k\bm{v}^{k,t-1}\big), \quad
	\bm{v}^{k,t} = \bm{\beta} ./ \big((\Xi^k)^{\top}\bm{u}^{k,t}\big),\quad \forall\; t\geq 0,
\end{equation}
where `$./$' denotes the entrywise division between two vectors. When a pair $(\bm{u}^{k,t}, \,\bm{v}^{k,t})$ is obtained based on a certain stopping criterion, an approximate solution of \eqref{iBPGM-subprob-reform} (and hence \eqref{iBPGM-subprob}) can be recovered by setting $X^{k,t}:= \mathrm{Diag}(\bm{u}^{k,t})\,\Xi^k\,\mathrm{Diag}(\bm{v}^{k,t})$. Meanwhile, a pair of approximate dual solutions can be recovered by setting $\bm{f}^{k,t}:=\mu_k\log\bm{u}^{k,t}$ and $\bm{g}^{k,t}:=\mu_k\log\bm{v}^{k,t}$. In our experiments, we simply execute Sinkhorn's iteration \eqref{sinkalg} only \textit{once} for each subproblem, and observe that this is sufficient for obtaining a promising initial point for warm-starting our ripALM.

\section{Efficient implementation of ripALM for BPDN}\label{sec:SMW_BPDN}

When applying ripALM to the BPDN problem \eqref{eq-bpdn}, a key component is the efficient implementation of the semismooth Newton ({\sc Ssn}) method for solving the non-smooth equation \eqref{eq:optcond-y}. The main computational task in {\sc Ssn} lies in computing the Newton direction, which requires solving the linear system \eqref{eq:gen_dir_Newton-BPDN}. We next describe how to solve it efficiently.

Let $(\bar{\bm{s}}, \bar{\bm{t}}, \bm{y}) \in \mathbb{R}^n \times \mathbb{R}^m \times \mathbb{R}^m$ and $\sigma>0$ be given. We restate \eqref{eq:gen_dir_Newton-BPDN} as follows
\begin{equation}\label{eq:gen_dir_Newton_2}
H\bm{d} = - \bm{g},
\end{equation}
where $\bm{g}:= \nabla\Phi(\bm{y})$, and $H\in\widehat{\partial}(\nabla\Phi)(\bm{y})$ has the form
\begin{equation*}
H := \sigma\left(DUD^{\top} + V\right) + \frac{\tau}{\sigma}I_{m},
\end{equation*}
with ${U} \in \partial \mathtt{prox}_{\sigma\|\cdot\|_1 }(\bm{u})$, $\bm{u} := \bar{\bm{s}}+\sigma {D}^{\top}\bm{y}$, and $V \in \partial \Pi_{\mathcal{B}_{\widehat{\kappa}}^2}(\bm{v})$, $\bm{v}:= \bar{\bm{t}}-\sigma \bm{y}$. Moreover, the generalized Jacobians $U$ and $V$ are chosen as follows:
\begin{equation*}
\begin{aligned}
U = &\; \text{Diag}(\bm{w})
~~\text{with}~~w_{i} =
\left\{\begin{array}{l}
0, \quad \text{if}~ |u_{i}| \leq \sigma, \\
1, \quad \text {otherwise, }
\end{array}\right.,~~ i = 1,\dots, n, \\
V = &\;\left\{\begin{array}{ll}
I_{m}, & \|\bm{v}\| \leq \widehat{\kappa}, \\
\frac{\widehat{\kappa}}{\|\bm{v}\|}\left(I_{m}-\frac{1}{\|\bm{v}\|^2}\bm{v} \bm{v}^{\top}\right), & \|\bm{v}\| > \widehat{\kappa}.
\end{array}\right.
\end{aligned}
\end{equation*}
Thanks to the special structure of $H$, we can solve the linear system   \eqref{eq:gen_dir_Newton_2} efficiently by considering the following two cases.

\textbf{Case I.} $V = I_{m}$. In this case, the linear system \eqref{eq:gen_dir_Newton_2} reduces to
\begin{equation}\label{eq:gen_dir_Newton_BPDN_1}
\left(DUD^{\top} + \frac{\sigma^2+\tau}{\sigma^2}I_{m}\right)\bm{d}
= -\frac{1}{\sigma}\bm{g}.
\end{equation}
Define the index set $\mathcal{I}:=\left\{i \mid |u_{i}|>\sigma, \,i=1, \ldots, n\right\}$. Taking into account the special $0$-$1$ sparsity structure of $U$, we see that $DUD^{\top}=D_{\mathcal{I}}{D}_{\mathcal{I}}^{\top}$, where $D_{\mathcal{I}} \in \mathbb{R}^{m \times |\mathcal{I}|}$ denotes the submatrix of $D$ formed by the columns indexed by $\mathcal{I}$. Consequently, the linear system \eqref{eq:gen_dir_Newton_BPDN_1} is further simplified as
\begin{equation}\label{eq:gen_dir_Newton_BPDN_1_reduced}
\left(D_{\mathcal{I}}D_{\mathcal{I}}^{\top} + \bar{\gamma}I_{m}\right)\bm{d}
= - \frac{1}{\sigma}\bm{g},
\end{equation}
where $\bar{\gamma}=\sigma^{-2}(\sigma^2+\tau)$. To solve \eqref{eq:gen_dir_Newton_BPDN_1_reduced}, we proceed as follows: When $m\leq|\mathcal{I}|$, we directly factorize the coefficient matrix $D_{\mathcal{I}}D_{\mathcal{I}}^{\top}+\bar{\gamma}I_{m}$, which requires $O(m^3)$ arithmetic operations. Otherwise, if $m>|\mathcal{I}|$, we can utilize the Sherman--Morrison--Woodbury (SMW) formula to improve efficiency. Specifically, the inverse of the coefficient matrix can be computed as
\begin{equation*}
\left(D_{\mathcal{I}}D_{\mathcal{I}}^{\top}+\bar{\gamma}I_{m}\right)^{-1}
= \frac{1}{\bar{\gamma}}\left(I_{m} - D_{\mathcal{I}}\left(\bar{\gamma}I_{|\mathcal{I}|}
+ D_{\mathcal{I}}^{\top}D_{\mathcal{I}}\right)^{-1}D_{\mathcal{I}}^{\top}\right).
\end{equation*}
As a result, it suffices to factorize the smaller matrix $\bar{\gamma}I_{|\mathcal{I}|} + D_{\mathcal{I}}^{\top}D_{\mathcal{I}}$, which requires $O(|\mathcal{I}|^3)$  operations. This would be much cheaper than factorizing a matrix of size $m\times m$, when the diagonal of $U$ is sparse.

\textbf{Case II.} $V=\frac{\widehat{\kappa}}{\|\bm{v}\|}\left(I_{m}-\frac{1}{\|\bm{v}\|^2}\bm{v} \bm{v}^{\top}\right)$. In this case, the linear system \eqref{eq:gen_dir_Newton_2} reduces to
\begin{equation}\label{eq-coe-mat-c2}
\left(B+D_{\mathcal{I}}D_{\mathcal{I}}^{\top}\right)\bm{d}
= -\frac{1}{\sigma}\bm{g}
\quad\text{with}\quad
B:= \left(\frac{\widehat{\kappa}}{\|\bm{v}\|}+\frac{\tau}{\sigma^2}\right)I_{m}
- \frac{\widehat{\kappa}}{\|\bm{v}\|^3}\bm{v}\bm{v}^{\top},
\end{equation}
which is more complicated than \textbf{Case I} due to the presence of the dense matrix $\bm{v}\bm{v}^{\top}$. When $m\leq|\mathcal{I}|$, we again directly factorize the coefficient matrix, which requires $O(m^3)$ operations. Otherwise, if $m>|\mathcal{I}|$, we proceed as follows: We first apply the SMW formula to compute the inverse of the coefficient matrix:
\begin{equation*}
\left(B + D_{\mathcal{I}}D_{\mathcal{I}}^{\top}\right)^{-1}
= B^{-1} - B^{-1}D_{\mathcal{I}}\left(I_{|\mathcal{I}|}
+ D_{\mathcal{I}}^{\top}B^{-1}D_{\mathcal{I}}\right)^{-1}D_{\mathcal{I}}^{\top}{B}^{-1}.
\end{equation*}
Furthermore, the matrix ${B}^{-1}$ can be efficiently computed via the SMW formula as well:
\begin{equation*}
B^{-1} = \left(\frac{\widehat{\kappa}}{\|\bm{v}\|}
+ \frac{\tau}{\sigma^2}\right)^{-1}I_{m}
+ \frac{1}{\bar{\alpha}}\left(\frac{\widehat{\kappa}}{\|\bm{v}\|}
+ \frac{\tau}{\sigma^2}\right)^{-2}\frac{\widehat{\kappa}}{\|\bm{v}\|^3}\bm{v}\bm{v}^{\top},
\end{equation*}
where $\bar{\alpha} := 1 - \frac{\widehat{\kappa}}{\|\bm{v}\|}\left(\frac{\widehat{\kappa}}{\left\|\bm{v}\right\|} + \frac{\tau}{\sigma^2}\right)^{-1} > 0$. Therefore, in this case, it suffices to factorize the matrix $I_{|\mathcal{I}|} + D_{\mathcal{I}}^{\top}{B}^{-1}D_{\mathcal{I}}$, which requires $O(|\mathcal{I}|^3)$ operations.

Finally, we mention that when direct factorization of the above matrices is computationally prohibitive, one may instead apply an iterative method—such as the preconditioned conjugate gradient method—to solve the linear system \eqref{eq:gen_dir_Newton_2}. Such methods can effectively exploit the underlying sparsity and low-rank structures to reduce computational cost.

\section{A dual ADMM for BPDN}\label{sec:BPDN-dadmm}

In this section, we briefly discuss how to apply ADMM for solving the dual problem \eqref{eq:BPDN-dual} of the BPDN problem. To this end, we first reformulate the dual problem as:
\begin{equation}\label{eq:BPDN-dadmm}
\begin{aligned}
\min_{\bm{u},\,\bm{v},\,\bm{y}} \Big\{
 \delta_{\mathcal{B}_1^\infty}(\bm{u})
+ \widehat{\kappa}\|\bm{v}\| - \bm{b}^{\top}\bm{y} \; \mid\;
 D^{\top}\bm{y} - \bm{u} = 0,
~ -\bm{y} - \bm{v} = 0\Big\}.
\end{aligned}
\end{equation}
Given a penalty parameter $\sigma >0$, the augmented Lagrangian function of \eqref{eq:BPDN-dadmm} is
\begin{equation*}
\begin{aligned}
\mathcal{L}_{\sigma}(\bm{u}, \bm{v}, \bm{y}; \bm{s}, \bm{t})
:= & ~ \delta_{\mathcal{B}_1^\infty}(\bm{u}) + \widehat{\kappa}\|\bm{v}\|
- \bm{b}^{\top}\bm{y} + \langle \bm{s}, \,D^{\top}\bm{y}
- \bm{u}\rangle + \langle \bm{t}, \,-\bm{y} - \bm{v}\rangle \\
& + {\textstyle\frac{\sigma}{2}\left\|D^{\top}\bm{y} - \bm{u}\right\|^2
+ \frac{\sigma}{2}\left\|-\bm{y} - \bm{v}\right\|^2}.
\end{aligned}
\end{equation*}
The ADMM for solving \eqref{eq:BPDN-dadmm} is presented in Algorithm~\ref{alg:BPDN-dadmm}. In Step 1, we can obtain the closed-form solution to the subproblem by evaluating the proximal mappings of $\delta_{\mathcal{B}_1^\infty}(\cdot)$ and $\frac{\widehat{\kappa}}{\sigma}\|\cdot\|$. In Step 2, the subproblem reduces to solving a linear system, where the coefficient matrix $I + DD^\top$ is symmetric positive definite and remains unchanged throughout the iterations, and therefore requires factorization only once. For large-scale problems where such a factorization is prohibited, one may also apply the preconditioned conjugate gradient method for solving the linear system.

\begin{algorithm}[ht]
\caption{ADMM for solving \eqref{eq:BPDN-dadmm}}\label{alg:BPDN-dadmm}
\begin{algorithmic}
\STATE \textbf{Input:} a penalty parameter $\sigma>0$, and initializations $\bm{y}^{0}\in\mathbb{R}^{m}$, $\bm{s}^0\in\mathbb{R}^{n}$, $\bm{t}^{0}\in\mathbb{R}^{m}$.
\WHILE{the termination criterion is not met}
\STATE \textbf{Step 1.} Compute
\begin{equation*}
\begin{aligned}
&\;\big(\bm{u}^{k+1},\,\bm{v}^{k+1}\big)
= \arg\min_{\bm{u},\bm{v}}\left\{\mathcal{L}_{\sigma}(\bm{u}, \bm{v},\bm{y}^{k}; \bm{s}^{k}, \bm{t}^{k})\right\} \\[2pt]
\Leftrightarrow~~ &\;
\bm{u}^{k+1} = \Pi_{\mathcal{B}_1^\infty}\left(D^\top\bm{y}^k
+ \sigma^{-1}\bm{s}^k\right),\quad
\bm{v}^{k+1} = \sigma^{-1}\left(\left(\bm{t}^{k} - \sigma\bm{y}^{k}\right)
- \Pi_{{\mathcal{B}_{\widehat{\kappa}}^2}}\left(\bm{t}^{k} - \sigma\bm{y}^{k}\right)\right).
\end{aligned}
\end{equation*}
		
\STATE \textbf{Step 2.} Compute
\begin{equation*}
\begin{aligned}
&\; \bm{y}^{k+1} = \arg\min_{\bm{y}\in\mathbb{R}^{m}} \left\{\mathcal{L}_{\sigma}(\bm{u}^{k+1},\bm{v}^{k+1},\bm{y};\bm{s}^{k},\bm{t}^{k})\right\} \\
\Leftrightarrow~~ &\;
\big(DD^{\top} + I\big)\bm{y}^{k+1}
=\sigma^{-1}\left(\bm{b} - D\big(\bm{s}^{k} - \sigma\bm{u}^{k+1}\big)
+ \big(\bm{t}^{k} - \sigma\bm{v}^{k+1}\big)\right).
\end{aligned}
\end{equation*}
		
\STATE \textbf{Step 3.} Compute
\begin{equation*}
\bm{s}^{k+1} = \bm{s}^{k} + \tau\sigma\big({D}^{\top}\bm{y}^{k+1} - \bm{u}^{k+1}\big),
\quad
\bm{t}^{k+1} = \bm{t}^{k} + \tau\sigma\big(-\bm{y}^{k+1} - \bm{v}^{k+1}\big),
\end{equation*}
where $\tau \in \left(0, \,\frac{1+\sqrt{5}}{2}\right)$ is the dual step-size that is typically set to $1.618$.
\ENDWHILE
		
\STATE \textbf{Output:} $\left(\bm{s}^{k+1},\,\bm{t}^{k+1},\,\bm{y}^{k+1}\right)$.
\end{algorithmic}
\end{algorithm}

\section{Second-order cone programming reformulation for BPDN}\label{sec:BPDN_SOCP}

In this section, we present a second-order cone programming (SOCP) reformulation of the BPDN problem \eqref{eq-bpdn}. Specifically, we introduce an auxiliary nonnegative variable $\bm{r}\in\mathbb{R}^n_+$ to majorize $|\bm{s}|$, i.e., $|\bm{s}|\leq\bm{r}$. Thus, the objective $\|\bm{s}\|_1$ can be replaced by $\bm{1}^{\top}\bm{r}$ together with the linear constraints $-\bm{r}\leq \bm{s} \leq\bm{r}$ and $\bm{r} \geq \bm{0}$. Next, for a given positive integer $d$, we denote the $(d+1)$-dimensional second-order cone as $\mathcal{Q}^{d+1} := \big\{ (s_0, \bm{s})\in\mathbb{R}^{d+1} \mid s_0 \geq \|\bm{s}\|\big\}$. With these notations, the BPDN problem \eqref{eq-bpdn} admits the following SOCP reformulation:
\begin{equation*}
\begin{aligned}
\min_{\bm{s},\,\bm{r}\in\mathbb{R}^{n}} \; \Big\{
\bm{1}^{\top}\bm{r} \; \mid
 \; -\bm{r}\leq \bm{s} \leq\bm{r}, \quad \bm{r} \geq \bm{0},
\quad \big(\widehat{\kappa}, D\bm{s}-\bm{b}\big) \in \mathcal{Q}^{m+1}\Big\}.
\end{aligned}
\end{equation*}

\bibliographystyle{plain}
\bibliography{Ref_ripALM}

\end{document}